\documentclass[a4paper,twoside]{amsart}
\usepackage[utf8]{inputenc}
\usepackage{enumitem}
\usepackage[T1]{fontenc}
\usepackage{amsmath}
\usepackage{amsfonts}
\usepackage{amssymb}
\usepackage{amsthm}
\usepackage{graphicx}
\usepackage{caption}
\usepackage{subcaption}
\usepackage{mathtools}
\usepackage{import}
\usepackage{fancyhdr}

\graphicspath{{drawings/}}

\newcommand{\p}{\ensuremath{\textgoth{p}}}

\newcommand{\Q}{\ensuremath{\mathbb{Q}}}
\newcommand{\Z}{\ensuremath{\mathbb{Z}}}
\newcommand{\N}{\ensuremath{\mathbb{N}}}

\newcommand{\rp}{\text{rp}}
\newcommand{\Sym}{\text{Sym}}

\DeclareRobustCommand{\qbinom}{\genfrac[]{0pt}{}}

\theoremstyle{plain}
\newtheorem{thm}{Theorem}
\newtheorem*{theorem}{Theorem}
\newtheorem{prop}[thm]{Proposition}
\newtheorem{lemma}[thm]{Lemma}
\newtheorem{cor}[thm]{Corollary}

\newtheorem{Not}[thm]{Notation}

\theoremstyle{definition}
\newtheorem{defn}[thm]{Definition}
				
\theoremstyle{remark}
\newtheorem{rem}[thm]{Remark}
\newtheorem{ex}[thm]{Example}

\usepackage{tikz}
\tikzstyle cross=[preaction={draw=white, -, line width=6pt}]
\tikzstyle normal=[thick]
\usepackage{braids}
\usetikzlibrary{arrows,decorations.markings}
\usepackage{relsize,exscale}
\usetikzlibrary{cd}

\newcommand{\BC}{\mathbb{C}}
\newcommand{\BN}{\mathbb{N}}
\newcommand{\BQ}{\mathbb{Q}}
\newcommand{\BZ}{\mathbb{Z}}

\newcommand{\CA}{\mathcal{A}}
\newcommand{\CB}{\mathcal{B}}
\newcommand{\CE}{\mathcal{E}}
\newcommand{\CK}{\mathcal{K}}
\newcommand{\CH}{\mathcal{H}}
\newcommand{\CP}{\mathcal{P}}
\newcommand{\CR}{\mathcal{R}}

\newcommand{\Bn}{\mathcal{B}_n}
\newcommand{\PBn}{\mathcal{PB}_n}

\newcommand{\Sk}{\mathfrak{S}}

\newcommand{\Id}{\mathop{Id}}
\newcommand{\Homeo}{\mathop{Homeo}}
\newcommand{\Mod}{\mathop{Mod}}
\newcommand{\inv}{\mathop{inv}}

\newcommand{\bapp}{\left. \begin{array}{rcl}}
\newcommand{\eapp}{\end{array} \right.}
\newcommand{\bfct}{\left\lbrace \begin{array}{rcl}}
\newcommand{\efct}{\end{array} \right.}

\newcommand{\Conf}{\operatorname{Conf}}

\newcommand{\Hlf}{\operatorname{H} ^{\mathrm{lf}}}
\newcommand{\Hnot}{\operatorname{H}}

\newcommand{\Hrelm}{\operatorname{\CH}^{\mathrm{rel }}}

\newcommand{\Crelm}{\operatorname{C}^{\mathrm{rel }-}}

\newcommand{\Laurent}{\CR}

\newcommand{\RR}{\operatorname{R}}
\newcommand{\Laurentcomplet}{\widehat{\Laurent}^{\widehat{I}}}

\newcommand{\slt}{{\mathfrak{sl}(2)}}
\newcommand{\Uq}{{U_q\slt}}

\newcommand{\UqhL}{{U^{\frac{L}{2}}_q\slt}}

\newcommand{\qbin}[2]{\left[\begin{array}{c}
      #1 \\
      #2 \end{array}\right]}

\newcommand{\Jones}{\operatorname{J}}
\newcommand{\J}{\operatorname{J}}
\newcommand{\A}{\operatorname{\CA}}
\newcommand{\ADO}{\operatorname{ADO}}
\newcommand{\SB}{\operatorname{SB}}

\newcommand{\End}{\operatorname{End}}
\newcommand{\Tr}{\operatorname{Tr}}
\newcommand{\Trp}{\Tr_{2,\ldots,n}}
\newcommand{\Trpt}{\widetilde{\Tr}_{2,\ldots,n}}
\newcommand{\spec}{\operatorname{spec}}

\begin{document}

\title[Unified invariant of knots from braids]{Unified invariant of knots from homological braid action on Verma modules}

\author[J. Martel]{Jules Martel} 
\address{Institut de mathématiques de Bourgogne, Université de Bourgogne, 9 avenue Alain Savary 21000 Dijon, France} 
\email{jules.martel-tordjman@u-bourgogne.fr}

\author[S. Willetts]{Sonny Willetts} 
\address{Institut de mathématiques de Toulouse, Université de Toulouse, 118 route de Narbonne, 31000 Toulouse, France} 
\email{sonny.willetts@math.univ-toulouse.fr}

\maketitle
\setcounter{tocdepth}{1}

\begin{abstract}
We re-build the quantum $\slt$ unified invariant of knots $F_{\infty}$ from braid groups' action on tensors of Verma modules. It is a two variables series having the particularity of interpolating both families of colored Jones polynomials and ADO polynomials, i.e. semi-simple and non semi-simple invariants of knots constructed from quantum $\slt$. We prove this last fact in our context which re-proves (a generalization of) the famous Melvin-Morton-Rozansky conjecture first proved by Bar-Natan and Garoufalidis. We find a symmetry of $F_{\infty}$ nicely generalizing the well known one of the Alexander polynomial, ADO polynomials also inherit this symmetry. It implies that quantum $\slt$ non semi-simple invariants are not detecting knots' orientation. Using the homological definition of Verma modules we express $F_{\infty}$ as a generating sum of intersection pairing between fixed Lagrangians of configuration spaces of disks. Finally, we give a formula for $F_{\infty}$ using a generalized notion of determinant, that provides one for the ADO family. It generalizes that for the Alexander invariant.  
\end{abstract}

\tableofcontents

\section{Introduction}

\subsection{Quantum invariants of knots associated with $\slt$}

From a quantum group and its category of finite dimensional representations, one can construct invariants of knots, links (and ribbon graphs). It is the original construction of N. Reshetikhin and  V. Turaev \cite{RT}. A quantum group could be think as a one parameter deformation of the enveloping algebra of a given semi-simple Lie algebra. In the present paper we only study the knot invariants arising from $\Uq$ which is a standard notation for the quantum group associated with $\slt$. Such knot invariants could be computed from {\em braid group} representations that are also part of the theory since they are defined using same finite dim. representations of $\Uq$. It is the approach of the present work. Historically, there are two families of knots that were extracted from the corresponding Reshetikhin--Turaev construction. 
\begin{enumerate}
\item The colored Jones polynomials $\lbrace \J_N\in \BZ\left[q^{\pm 1} \right], N \in \BN \rbrace$ obtained from {\em standard irreducible} rep. of dim. $N+1$ of $\Uq$ denoted $S_N$ as input (see e.g. \cite{MuMu}), by use of Reshetikhin--Turaev construction. They could all be derived from the famous Jones polynomial \cite{JonesPoly}. 
\item The {\em ADO} polynomials, sometimes called {\em colored Alexander polynomials}, $\lbrace \ADO_r \in \BC\left[A^{\pm 1} \right], r \in \BN\rbrace$ arising from particular irreducible representations of $\Uq$ when $q$ is evaluated at a root of unity. They were first defined by Akutsu--Deguschi--Ohtsuki \cite{akutsu1992invariants}, but they require a slight modification of the original tool developed by Reshetikhin--Turaev while the philosophy of using $\Uq$ is a constant. The first of the family is the well known Alexander polynomial denoted $\A$. 
\end{enumerate}

The construction of Reshetikhin and Turaev uses the fact that in categories of representations of quantum groups they find inherent tools of the category behaving nicely with {\em Reidemeister moves}. Namely, there are {\em $R$-matrices} allowing to linearly represent braid groups carrying Reidemeister moves for braids, and {\em Markov traces} allowing to extract knot invariants from braid groups representations hence taking care of the remainder Reidemeister moves. Even though finding these two objects in any context is not trivial (e.g. colored Jones vs ADO polynomials, where they are differently defined), they are always operators on $\Uq$ modules satisfying nice equation translating Reidemeister moves in an algebraic language. In the end, one obtains powerful topological invariants but their full algebraic flavor makes the topological interpretation of their content difficult and the subject of many conjectures in the field. One of the most famous expectation of topological content is the hyperbolic volume, which is the subject of the {\em volume conjecture} first stated by R. Kashaev \cite{KashVol} and relocated in the context of colored Jones polynomials by Murakami and Murakami in \cite{MuMu}. The question on how to interpret topologically quantum invariants is more generally central. 

Two other questions could be addressed to the picture:
\begin{itemize}
\item Could we construct knot invariants out of infinite dim. modules of $\Uq$? While Reshetikhin--Turaev construction requires finite dimension. 
\item Are quantum invariants colored Jones and ADO related or even equivalent? Even though the theory of representations of $\Uq$ is singularly different when $q$ is a root of one (ADO case) than when $q$ is generic (colored Jones case). 
\end{itemize}

These three last questions (the two above and the topological interpretation of the construction of knot invariants) have recently reached new steps by use of the same objects: $\Uq$ Verma modules, which are infinite dim. modules on $\Uq$. 

In \cite{willetts2020unification} the second author has constructed a knot invariant denoted $F_{\infty}$ using as input $\Uq$ Verma modules. The obtained object is a two variable infinite sum converging in the sense that it lives in a nice completion of the ring of Laurent polynomials with two variables $\Laurent := \BZ\left[q^{\pm 1} , s^{\pm 1} \right]$. By nice, we mean e.g. that $F_{\infty}$ can be evaluated at $q$ being a root of unity or $s$ being a power of $q$. Moreover in the first case $F_{\infty}$ recovers the ADO polynomials and in the second the colored Jones ones. This double interpolation property implies an equivalence between the two families of knots. 

In \cite{martel2020homological}, the first author has reconstructed $\Uq$ Verma modules, their tensor products, and the quantum braid group representation upon them from {\em homology of configuration spaces of points in punctured disks with coefficients in a local ring isomorphic to $\Laurent$}. The action of braid groups on these modules is given by (more or less) homeomorphisms of configuration spaces, using the fact that braid groups are {\em mapping class groups} (isotopy classes of homeomorphisms) of punctured disks. Hence one can use this purely homological definition of Verma modules and quantum braid group representations avoiding dealing with quantum modules theory, shedding light on the topological content of it. 

The present paper studies in details the tools surrounding Verma modules (their tensor product, braid group representations and knot invariant) developed in the two papers \cite{willetts2020unification,martel2020homological}, more particularly what topological information one could extract out of $F_{\infty}$. 

Next steps could be achieved using $\Uq$ and its modules, for instance constructions of {\em topological quantum field theories (TQFTs)} which is a categorical construction providing invariants of links and embedded graphs (extending those of knots), $3$-manifolds and mapping class groups of surfaces representations (extending those of braids). This was initiated by Reshetikhin and Turaev again \cite{RT2} and the universal construction of Blanchet--Habegger--Masbaum--Vogel \cite{BHMV}. In the colored Jones context (for which the category of $\Uq$ modules is semi-simple), the output is the {\em Witten--Reshetikhin--Turaev TQFT} (WRT). More recently, Blanchet--Costantino--Geer--Patureau have succeeded in constructing TQFT \cite{BCGP} from the category of modules on $\Uq$ when $q$ is a root of $1$ (which is non semi-simple). We call them {\em non semi-simple TQFTs} and the inherent knot invariant is hence the ADO family. These non semi-simple TQFTs are improvements of WRT for some reasons: e.g. detecting lens spaces and Dehn twists, but $F_{\infty}$ shows that at the level of invariants of knots they are the same. The invariant $F_{\infty}$ is still not defined on links, while colored Jones and ADO families might differ at some point. 

\subsection{Content: Unified invariant of knots and homological action of braids}

In \cite{willetts2020unification}, the second author defines an invariant of knots denoted $F_{\infty}$ which is an element living in some completion $\Laurentcomplet$ of $\Laurent := \BZ \left[ q^{\pm 1}, s^{\pm 1} \right]$. This definition implies the application of a universal invariant constructed by Lawrence and Ohtsuki \cite{lawrence1988universal} \cite{lawrence1990universal}\cite{ohtsuki2002quantum} and widely studied by Habiro \cite{habiro2006bottom}, on any vector of some quantum $\slt$ Verma module. In \cite{martel2020homological}, the first author has developed braid group representations on tensor products of these Verma modules with coefficients in $\Laurent$ providing a homological definition arising from local systems on configuration spaces of points in punctured disks. In this paper we express $F_{\infty}$ as a {\em partial trace} of the braid action on tensor products of Verma modules. 

\begin{theorem}[{Theorem~\ref{prop_unifed_braidrep}}]
Let $V$ be the universal Verma $\Laurent$-module of $\Uq$. Let $\CK$ be a knot such that it is the closure of a braid $\beta \in \Bn$. Then:
\[
F_{\infty} = \Trp \left(h \circ \beta , V^{\otimes n } \right),
\]
where the right term is the partial trace of the action of $\beta$ on $V^{\otimes n}$ post composed with the (fixed) operator $h$ explicitly defined later on. 
\end{theorem}

We re-prove the following property.

\begin{theorem}[{Theorem~\ref{thm_factorisation_unified_ADO}}]
For an integer $r \in \BN^*$ and $\zeta_{2r}$ a $2r$-th root of $1$, we have:
\[ F_{\infty} (\zeta_{2r}, A, \mathcal{K}) =  \frac{(q^{\alpha})^{rf} \times \ADO_r(A,\mathcal{K})}{\A_{\mathcal{K}} (A^{2r})} \]
		where $f$ is the framing of the knot, $\ADO_r$ is the $r$-th ADO polynomial (\cite{akutsu1992invariants}) and $\A_{\CK}$ the Alexander polynomial of $\CK$. 
\end{theorem}

The latter was proved in \cite{willetts2020unification} but considering Melvin--Morton--Rozansky (MMR) conjecture which is a theorem due to Bar-Natan and Garoufalidis \cite{MMR}. Here we prove it carefully studying the structure of tensor products of Verma modules when $q$ is a root of $1$. Hence we have re-proved MMR conjecture in a slight generalization, namely an analytic relation between any colored Jones polynomial and the Alexander polynomial.

It is well known that the Alexander polynomial of a knot is invariant under the change of variable $s \mapsto s^{-1}$. We extend this symmetry to the entire $F_{\infty}$ and it gives a nice symmetry for the ADO invariants of knots too.

\begin{theorem}[{Theorem~\ref{symmetry_unified_invariant},~Coro.~\ref{cor_symmetry_ado}}]
For any knot $\CK$:
\begin{itemize}
\item $F_{\infty}(\CK)$ is unchanged under $s \mapsto s^{-1}q^{-2}$,
\item $\ADO_r(\CK)$ is unchanged under $s \mapsto s^{-1} \zeta_{2r}^{-2}$.
\end{itemize}
\end{theorem}

The second bullet point implies that the non semi-simple $\Uq$ invariant of planar graphs introduced by Costantino--Geer--Patureau in \cite{CGP} does not detect orientation of knots (Coro. \ref{cor_orientation_Nr}).

 Using the homological definition from \cite{martel2020homological} for tensor products of Verma modules, and Poincaré duality in homology, we express $F_{\infty}$ as the intersection pairing with coefficients in $\Laurent$ between fixed middle dimension homology classes in configuration spaces of points in punctured disks.

\begin{theorem}[{Theorem~\ref{thm_homol_formula_Foo}}]
Let $\beta \in \Bn$ a braid such that its closure is the knot $\CK$. Then:
\[
F_{\infty}(\CK) = s^{n-1}\sum_{\overline{k}} \left\langle \beta  \cdot A''(\overline{k}) \cap B''(\overline{k}) \right\rangle q^{-2\sum k_i} 
\]
where for any list of $n-1$ integers $\overline{k}$, $A''(\overline{k})$ and $B''(\overline{k})$ are precisely defined middle dimension manifolds of the space of configurations of points in the $n$-th punctured disks. The action of $\beta$ is naturally defined by homeomorphism of the punctured disk, and $\langle \cdot \cap \cdot \rangle$ is a homological intersection pairing in $\Laurent$ given by Poincaré duality.  
\end{theorem}
The latter means that the right term in the equation, which is an infinite sum of intersection pairing of middle dimension homology classes, lives in $\Laurentcomplet$ and is invariant under Markov moves. 

Finally, we express $F_{\infty}$ using a generalized notion of determinant of matrices called {\em quantum determinant of right quantum matrices}, defined for matrices with non-commutative entries. This quantum determinant is presented in \cite{GLZ}. The quantum determinant formula resembles the classical one for the Alexander polynomial: it is the quantum determinant of a deformed Burau matrix instead of a regular determinant of the regular Burau matrix. It is stated in Theorem \ref{thm_F_from_qdet}, and generalizes formula of Lê and Huynh \cite{HL} for colored Jones polynomials. 

\subsection{Plan of the paper}

In Sec. \ref{sec_qVermas} we establish the context of the quantum group $\Uq$ and its Verma module. We define the action of braid groups, the splitting into finite dim. levels, and we carefully study the structure while specializing $q$ at a roots of one, giving rise to a particular {\em r-part} factorization. 

In Sec. \ref{sec_Foo_frombraids} we redefine (Theorem \ref{prop_unifed_braidrep}) the knot invariant $F_{\infty}$ as a partial trace on braid group representations previously defined after having recalled its former definition from \cite{willetts2020unification}. Using the r-part factorization at roots of unity from previous section, we prove Theorem \ref{thm_factorisation_unified_ADO} re-proving the factorization of $F_{\infty}$ at roots of one, re-proving MMR conjecture. We then prove Theorem \ref{symmetry_unified_invariant} providing an Alexander-like symmetry of invariants $F_{\infty}$ and ADO. 

In Sec. \ref{sec_Foo_fromHomology} we prove Theorem \ref{thm_homol_formula_Foo} which expresses $F_{\infty}$ as a sum of intersection pairing between Lagrangians in configuration spaces of punctured disks. This requires first a precise recall of the homological set-up from \cite{martel2020homological}, i.e. the homological definition of Verma modules, their tensor products and the braid action. 

In Sec. \ref{sec_Foo_from_qdet} we recall the definition of \emph{quantum determinant} for right quantum matrices. We recall the context of paper \cite{H-L}, and finally prove Theorem \ref{thm_F_from_qdet} providing a quantum determinant formula for invariants $F_{\infty}$ and ADO. 

%
%
%
%
\section{Quantum $\slt$ and its universal Verma module}\label{sec_qVermas}

We introduce quantum numbers, factorials and binomials.

\begin{defn}\label{quantumq}
Let $i,k,l,n$ be integers. We define the following elements of $\BZ \left[ q^{\pm 1} \right]$.
\begin{align}
\left[ i \right]_q := \frac{q^i-q^{-i}}{q-q^{-1}} ,& \text{  } \left[ k \right]_q! := \prod_{i=1}^k \left[ i \right]_q , \text{  } \qbin{k}{l}_q := \frac{\left[ k \right]_q!}{\left[ k-l \right]_q! \left[ l \right]_q!}\\
\{ n \} = q^{n}-q^{-n} & \text{ and } \{ n \} ! = \prod_{i=1}^n \{ i \} .
\end{align}
with the convention $\qbinom{n}{k}_q=0$ if $n < 0$.

We also fix notation for elements of $\BZ \left[ q^{\pm 1}, s^{\pm 1} \right]$ but using the following notation $q^{\pm \alpha} := s^{\pm 1}$ that will be useful later on.
\begin{align}
\{ \alpha \}_q = q^{\alpha}-q^{-\alpha}& ,  \{ \alpha +k \}_q = q^{\alpha +k } - q^{-\alpha -k } , \{ \alpha;n \}_q= \prod_{i=0}^{n-1} \{ \alpha -i\}_q.
\end{align}
where one can easily deduce how to write them in $\BZ \left[ q^{\pm 1}, s^{\pm 1} \right]$. (To do computation in $\BC$ and think of $q, \alpha$ and $s$ as complex numbers, one must fix a logarithm of $q$. )
\end{defn}

In what follows we will define $\Uq$ in Sec. \ref{halfLusztigversion}, then its Verma modules and the associated action of braid groups in Sec. \ref{VermaBraiding}. We study the structure of this braid group representation while variables are evaluated at some particular value in Sec. \ref{sec_spec_of_variables}. In the case of $q$ being a root of one, we show that the representation splits into \emph{r-parts} sub-representations. 

\subsection{The algebra $\UqhL$}\label{halfLusztigversion}

In this section, we define an integral version for the quantized algebra associated with $\slt$. By integral, we mean as an algebra over the ring of Laurent polynomials in one variable, but first we define the standard algebra $\Uq$ on the rational field. 

\begin{defn}\label{Uqnaif}
The algebra $\Uq$ is the algebra over $\BQ(q)$ generated by elements $E,F$ and $K^{\pm 1}$, satisfying the following relations:
\begin{align*}
KEK^{-1}=q^2E & \text{ , } KFK^{-1}=q^{-2}F \\
\left[E, F \right] = \frac{K-K^{-1}}{q-q^{-1}} & \text{ and }
KK^{-1}=K^{-1}K=1 .
\end{align*}
The algebra $\Uq$ is endowed with a coalgebra structure defined by $\Delta$ and $\epsilon$ as follows:
\[
\begin{array}{rl}
\Delta(E)= 1\otimes E+ E\otimes K, & \Delta(F)= K^{-1}\otimes F+ F\otimes 1 \\
\Delta(K) = K \otimes K, & \Delta(K^{-1}) = K^{-1}\otimes K^{-1} \\
\epsilon(E) = \epsilon(F) = 0, & \epsilon(K) = \epsilon(K^{-1}) = 1
\end{array}
\]
and an antipode defined as follows:
\[
S(E) = EK^{-1}, S(F)=-KF,S(K)=K^{-1},S(K^{-1}) = K.
\]
This provides a {\em Hopf algebra} structure, so that the category of modules over $\Uq$ is monoidal.
\end{defn}

We are interested in an integral version that resembles Lusztig version but with only half of {\em divided powers} for generators. This version is used and introduced in \cite{Hab,JK,martel2020homological,willetts2020unification} (with subtle differences in the definitions of divided powers for $F$). Let:
\[
F^{(n)} :=  \frac{(q-q^{-1})^n}{\left[ n \right]_q!} F^n .
\]
Let $\Laurent_0 = \BZ\left[ q^{\pm 1} \right]$ be the ring of integral Laurent polynomials in the variable $q$. 

\begin{defn}[Half integral algebra]\label{Halflusztig}
Let $\UqhL$ be the $\Laurent_0$-subalgebra of $\Uq$ generated by $E$, $K^{\pm 1}$ and $F^{(n)}$ for $n\in \BN^*$. We call it a {\em half integral version} for $\Uq$, the word half to illustrate that we consider only half of divided powers as generators. 
\end{defn}

\begin{rem}[Relations in $\UqhL$, {\cite[(16)~(17)]{JK}}]\label{relationsUqhL}
The relations among generators involving divided powers are the following ones:
\[
KF^{(n)}K^{-1} = q^{-2n}F^{(n)}
\]
\[
\left[ E, F^{(n+1)}  \right] = F^{(n)} \left( q^{-n} K - q^n K^{-1}  \right) \text{ and }
F^{(n)} F^{(m)} = \qbin{n+m}{n}_q F^{(n+m)} .
\]
Together with relations from Definition \ref{Uqnaif}, they complete a presentation of $\UqhL$. 

$\UqhL$ inherits a Hopf algebra structure with a coproduct given by:
\[
\Delta(K) = K \otimes K \text{ , } \Delta(E) = E \otimes K + 1 \otimes E , \text{ and } \Delta(F^{(n)}) = \sum_{j=0}^n q^{-j(n-j)}K^{j-n} F^{(j)} \otimes F^{(n-j)}. 
\]
\end{rem}

\begin{prop}[Poincaré--Birkhoff--Witt basis]
The algebra $\UqhL$ admits the following set as an $\Laurent_0$-basis:
\[
\left\lbrace K^l E^m F^{(n)} , l \in \BZ, m,n \in \BN \right\rbrace .
\]
\end{prop}

\subsection{Verma modules and braiding}\label{VermaBraiding}

We define the {\em Verma modules}. They are infinite dimensional modules over $\UqhL$ depending on a parameter. Again we work with an integral version by including the parameter in the ring of Laurent polynomials as a formal variable. Let $\Laurent := \BZ \left[ q^{\pm 1} , s^{\pm 1} \right]$. 
%
\begin{defn}[Verma modules for $\UqhL$]\label{GoodVerma}
Let $V^{s}$ be the Verma module of $\UqhL$. It is the infinite $\Laurent$-module, generated by vectors $\lbrace v_i, i \in \BN \rbrace$, and endowed with an action of $\UqhL$, generators acting as follows:
\[
K \cdot v_j = s q^{-2j} v_{j} \text{ , } E \cdot v_j = v_{j-1} \text{ and } F^{(n)} v_j = \left( \qbin{n+j}{j}_q \prod_{k=0}^{n-1} \left( sq^{-k-j} - s^{-1}q^{j+k} \right) \right) v_{j+n} .
\]
\end{defn}

\begin{rem}[Weight vectors]\label{weightdenomination}
We will often make implicitly the change of variable $s := q^{\alpha}$ and denote $V^s$ by $V_{\alpha}$. This choice made to use a practical and usual denomination for eigenvalues of the $K$ action (which is diagonal in the given basis). Namely we say that vector $v_j$ is {\em of weight $\alpha - 2j$}, as $K \cdot v_j = q^{\alpha - 2j} v_j$. The notation with $s$ shows an integral Laurent polynomials structure strictly speaking. In the case $s=q^{\alpha}$ one can use a simpler notation in the action of $F^{(n)}$:
\[
\prod_{k=0}^{n-1} (sq^{-k-j} - s^{-1}q^{j+k}) = \lbrace \alpha - j , n \rbrace
\]
\end{rem}

\begin{defn}[$R$-matrix, {\cite[(21)]{JK}}]\label{goodRmatrix}
Let $s=q^{\alpha}$ , $t=q^{\alpha'}$. The operator $q^{H \otimes H /2}$ is the following:
\[
q^{H \otimes H /2}:
\bfct
V^{s} \otimes V^{t} & \to & V^{s} \otimes V^{t}  \\
v_i \otimes v_j & \mapsto & q^{(\alpha - 2i)(\alpha'-2j)} v_i \otimes v_j 
\efct .
\]
We define the following R-matrix:
\[
R : q^{H \otimes H/2} \sum_{n=0}^\infty q^{\frac{n(n-1)}{2}} E^n \otimes F^{(n)} 
\]
which will be well defined as an operator on Verma modules in what follows. 
\end{defn}

We recall the Artin presentation of the braid groups.

\begin{defn}\label{Artinpres}
Let $n\in \BN$. The {\em braid group} on $n$ strands $\Bn$ is the group generated by $n-1$ elements satisfying the so called {\em ``braid relations"}:
$$\Bn := \left\langle \sigma_1,\ldots,\sigma_{n-1} \Big| \begin{array}{ll} \sigma_i \sigma_j = \sigma_j \sigma_i & \text{ if } |i-j| \ge 2 \\ 
\sigma_i \sigma_{i+1} \sigma_i = \sigma_{i+1} \sigma_i \sigma_{i+1} & \text{ for } i=1,\ldots, n-2 \end{array} \right\rangle$$
\end{defn}

\begin{prop}[{\cite[Theorem~7]{JK}}]\label{UqhLbraiding}
Let $V^s$ and $V^t$ be Verma modules of $\UqhL$ (with $s=q^{\alpha}$ and $t=q^{\alpha'}$). Let $\RR$ be the following operator:
\[
\RR: q^{-\alpha \alpha' /2} T \circ R 
\]
where $T$ is the twist defined by $T(v\otimes w ) = w \otimes v$. Then $\RR$ provides a braiding for $\UqhL$ integral Verma modules. 
Namely, the morphism:
\[
\phi_n: \bfct
\Laurent\left[ \Bn \right] & \to & \End_{\Laurent, \UqhL} \left({V^s}^{\otimes n}\right)  \\
\sigma_i & \mapsto & 1^{\otimes i-1} \otimes \RR \otimes 1^{\otimes n-i-2}
\efct
\]
is an $\Laurent$-algebra morphism. It provides a representation of $\Bn$ such that its action commutes with that of $\UqhL$. In the sequel we will sometime denote $\phi_n(q^\alpha, \cdot)$ to emphasize the dependence in variables.  
\end{prop}

\begin{rem}\label{coloredquantum}
One can consider a braid action over $V^{s_1} \otimes \cdots \otimes V^{s_n}$ (considering more variables in the ring) such that the morphism $\phi_n$ is well defined but becomes an algebra morphism only when restricted to the pure braid group $\PBn$. These braids are actually defining endomorphisms. See \cite[Appendix]{martel2020colored} for a detailed explanation on {\em colored versions}. 
\end{rem}

		Elements of same weight in the tensor product of Verma modules form a sub-representation of the braid group.
		
		\begin{defn}[Sub-weight representations]
		Let: \begin{align*} 
		V_{n,m}(q,q^{\alpha})&:= \left\lbrace v_{i_1} \otimes \dots \otimes v_{i_n} \in V_{\alpha}^{\otimes n} \text{ s. t. } \sum_{k=1}^n i_k = m \right\rbrace
		\end{align*}
		and be the space of {\em sub-weight $m$ vectors}. It is stable under the action of braids so that we denote: \[ \varphi_{n,m}(q, q^{\alpha}, .): B_n \to \End_{\Laurent}(V_{n,m}) .\]
		the associated (restricted) representation. 
		
		When there is no ambiguity on variables, we will write $V_{n,m}:= V_{n,m}(q,q^{\alpha})$ and $\varphi_{n,m}(\beta):=\varphi_{n,m}(q, q^{\alpha}, \beta)$.
		
		\end{defn}

		\begin{rem}
		The stability of sub-weight vectors under braid actions is deduced from the fact that the latter action commutes with that of $\UqhL$ and from the fact that sub-weight vectors are eigenvectors for the $K$ action. Namely:
		\begin{equation}\label{weightfromeigenvalues}
		V_{n,m}(q,q^{\alpha})=\{ v \in V_{\alpha}^{\otimes n} | Kv= q^{\alpha -2m}v \}.
		\end{equation}
		This is for $q$ being a formal variable. At roots of unity (i.e. when $q$ is a root of $1$, see next section), since $q^{-2r}=1$, Eq. \ref{weightfromeigenvalues} does not stand. Still the braid action preserves $V_{n,m}$ and this can be seen directly from the terms of the R-matrix preserving the sum of indices of tensors.
		\end{rem}

		\subsection{Specialization of variables}\label{sec_spec_of_variables}
		
		Working with the ring $\Laurent$ is particularly comfortable for specialization of variables, i.e. giving a complex value to variables $q$ and $s$. This corresponds to a morphism:
		\[
		\spec: \Laurent \to \BC
		\]
		and algebraically speaking, all the data set just presented has to be replaced by:
		\[
		\tilde{\mathcal{U}}_{\spec} := \UqhL \otimes_{\spec} \BC \text{ and } V^s \otimes_{\spec} \BC \text{ and so on.}
		\]
		This is what we will mean by {\em specialization}. (We will simply denote $\tilde{\mathcal{U}}$ when the specialization is clear)

		\subsubsection{Specialization to integral weights}
		
		We can take a specialization at {\em integral weights} setting $s=q^{\alpha}=q^N$ for $N\in \Z$ in the previous formulas and we denote $V^{N}$ the corresponding Verma module with integral weights. We find a classical sub-module in that case:
		\begin{defn}[Simple module of dim. $N$]
		We denote $S_N$ the module spanned by $\lbrace v_0 , \ldots , v_N \rbrace$. It is a sub-module of $V^N$ isomorphic to the highest weight simple module of dim. $N+1$. 
		\end{defn}

		This specialization has a symmetry as shown in the following lemma:

\begin{lemma}\label{lemma_quotient_Verma_integer}
For $N \in \N^*$, we have the isomorphism of $\tilde{\mathcal{U}}$ modules:  \[ V^{-N-2} \cong V^N / S_N .\].
\end{lemma}
\begin{proof}
While $(v_i)_{i\in \BN}$ is set to be the basis of $V^N$, $(\overline{v_i})_i$ the basis of the quotient $V^N/S_N$, we get :
\begin{align*}
E \overline{v_{N+1}} &= 0\\
E\overline{v_{N+1+i+1}} &= \overline{v_{N+1+i}}\\
K\overline{v_{N+1+i}}&=q^{-N-2 -2i} \overline{v_{N+1+i}}\\
F^{(n)} \overline{v_{N+1+i}}&=\qbinom{n+N+1+i}{n}_q \{ -i-1 ; n \}_q \overline{v_{N+1+n+i}} 
\end{align*}
		
We can transform a bit the last equation using:
\begin{align*}
\{n+N+1+i;n \}_q \{ -i-1 ; n \}_q &= \{i+n;n\}_q \{-N-2-i;n \}_q
\end{align*}
Hence, 
\[ F^{(n)} \overline{v_{N+1+i}}=\qbinom{n+i}{n}_q \{-N-2-i;n \}_q \overline{v_{N+1+n+i}}\]

Setting $v_i := \overline{v_{N+1+i}}$ for $i \ge 0$, one recognizes the exact definition $V^{-N-2}$.
\end{proof}
		
		\subsubsection{Specialization of $q$ to $1$}
		
		We treat the case $q=1$ slightly differently from other roots of unity (see next section). We fix particular notations in this context.
%

\begin{Not}
When $q=1$ we fix:
\begin{itemize}
\item $\SB_{n,m} := V_{n,m}(1, q^{\alpha})$,
\item $w_i := v_i$,
\item $\psi_{n,m}(q^{\alpha}, \beta):= \varphi_{n,m}(1,q^{\alpha},\beta)$.
\end{itemize}
(The notation $\SB$ refers to the fact that it is isomorphic to a symmetric power of the Burau representation, see next proposition). 
\end{Not}

		A nice property of the $q=1$ case is that the sub-weight $m$ level representation can be obtained as a symmetric power of the first sub-weight level.
		
		\begin{prop}\label{prop_psi_sym}
		Let $\beta \in B_n$, then : \[ \psi_{n,m}(q^{\alpha}, \beta) = \Sym^m (\psi_{n,1}(q^{\alpha},\beta)) .\]
		
		\end{prop}
		
		\begin{proof}
		First we need to consider a diagonal change of bases. We set $u_j:= j! w_j$.
		Let $e_k:=u_0 \otimes \dots \otimes u_1 \otimes \dots \otimes u_0 \in \SB_{n,1}$ where the only $u_1$ is located at the $k$-th position.  The family $\left\lbrace e_k, k = 1, \ldots, n \right\rbrace$ is a basis of $\SB_{n,1}$. We can identify higher weight tensors with symmetric powers of the $e_k$ using the one to one following correspondence:  \[ u_{j_1} \otimes \dots \otimes u_{j_n} \longleftrightarrow \prod_{k=1}^n e_k^{j_k}. \]
		
		Now, in the basis $e_k$, we have :
		\[ \psi_{n,1}(\sigma_i)= 
		\left(\begin{array}{@{}c|c@{}|c@{}}
		  I_{i-1}
		  & 0 
		  & 0\\
		\hline
		  0 
		  & \begin{matrix}
		  1-q^{-2\alpha} & q^{-\alpha} \\
		  q^{-\alpha} & 0
		  \end{matrix}
		  & 0\\
		\hline
		  0
		  & 0
		  & I_{n-i-1}
		\end{array}\right). \]
		
		We compute the symmetric power action in the $u_j$ basis, 
		\begin{align*}
		\Sym^m(\psi_{n,1}(\sigma_i))u_{\overline{j}} &= \prod_{k=0}^n (\psi_{n,1}(\sigma_i) e_k)^{j_k}\\
		&= e_1^{j_1} \cdots e_{i-1}^{j_{i-1}}((1-q^{-2\alpha})e_i+q^{-\alpha}e_{i+1})^{j_i} \times (q^{-\alpha} e_i)^{j_{i+1}} e_{i+2}^{j_{i+2}} \cdots e_n^{j_n}\\
		&= \sum_{l=0}^{j_i} \binom{j_i}{l} \{\alpha\}^l q^{-(j_i+ j_{i+1})\alpha} u_{j_1} \otimes \dots  \otimes u_{j_{i+1}+l} \otimes u_{j_i-l} \otimes \dots \otimes u_{j_n}
		\end{align*}
		
		If we transpose it back in the basis $w_j$ we get:
		\begin{align*}
		Sym^m(\psi_{n,1}(\sigma_i))w_{\overline{j}} &= \sum_{l=0}^{j_i} \binom{j_{i+1}+l}{l} \{\alpha\}^l q^{-(j_i+ j_{i+1})\alpha} w_{j_1} \otimes \dots  \otimes w_{j_{i+1}+l} \otimes w_{j_i-l} \otimes \dots \otimes w_{j_n} \\
		&= \psi_{n,m}(q^{\alpha}, \beta) w_{\overline{j}}
		\end{align*}
		The last equality is directly checked from the set-up: Defs. \ref{GoodVerma} and \ref{goodRmatrix} and Prop. \ref{UqhLbraiding}.
		\end{proof}

		\subsubsection{Specialization of $q$ to roots of $1$: $r$-part sub-representations}\label{section_rparts}
		
		In this subsection we set $q=\zeta_{2r}$ which corresponds to a specialization as defined above.
		
		\begin{defn}
		The $r$-part of a tensor $v=v_{i_1 + r j_1} \otimes \dots \otimes v_{i_n +r j_n} \in V_{\alpha}^{\otimes n}$ where $i_1, \dots i_n \leq r-1$ is defined by \[ \rp(v):=\sum_{k=0}^n j_k .\]
		
		\end{defn}

		\begin{defn}
		We define subspaces of $V_{\alpha}^{\otimes n}$
		\[ V_n^m (\zeta_{2r},q^{\alpha}):= < v | \ \rp(v)=m > \]
		\[ V_n^{ \leq m} (\zeta_{2r},q^{\alpha}):= \bigoplus_{i=0}^m V_n^m(\zeta_{2r},q^{\alpha})\]
		when there is no ambiguity, we will write $ V_n^m := V_n^m (\zeta_{2r},q^{\alpha}) $ and $V_n^{ \leq m}:=V_n^{ \leq m} (\zeta_{2r},q^{\alpha})$.
		\end{defn}

		\begin{prop}
		$V_n^{ \leq m}$ is a sub-representation of braids designed by $\varphi_n^{\leq m}(\beta)$ (the restriction of $\varphi_{n,m}$, with the implicit specialization of variables).
		\end{prop}
		\begin{proof}
		First remark that \[ \rp(E^r \otimes F^{(r)} (v_i \otimes v_j)) = \rp(v_i \otimes v_j).\]
		Moreover $F^{(i+rj)} v_{a+ru}=0$ if $i, a \leq r-1$ and $a+i \geq r$, hence \[\rp(E^n \otimes F^{(n)} (v_i \otimes v_j)) \leq \rp(v_i \otimes v_j) .\]
		Thus, $V_n^{ \leq m}$ is invariant via the action of the $R$ matrix and its inverse (for the inverse, see e.g. \cite[Prop. 6]{willetts2020unification}).
		\end{proof}
		
		This allows us to have another sub-representation via projection maps
		\begin{prop}
		Let $\rho_n^m : V_n^{\leq m} \to V_n^m$ the canonical projection map, then $V_n^m$ is endowed with a representation of $\Bn$ using the projection of the general action: \[ \varphi_n^m:= \rho_n^m \circ \varphi_n^{\leq m}|_{V_n^m}.\]
		\end{prop}
		
		\begin{proof}
		Since $V_n^{\leq m-1}$ is a sub-representation, if $v \in V_n^{\leq m-1}$ we have $\varphi_n^{\leq m-1}(\beta_1)v \in V_n^{\leq m -1}$ and hence $\rho_n^m \circ \varphi_n^{\leq m}(\beta_1) v=\rho_n^m \circ \varphi_n^{\leq m-1}(\beta_1) v=0$. 
		
		This means the following \[\rho_n^m \circ \varphi_n^{\leq m}(\beta_1) \circ \rho_n^m \circ \varphi_n^{\leq m}(\beta_2)|_{V_n^m}=\rho_n^m \circ \varphi_n^{\leq m}(\beta_1) \circ \varphi_n^{\leq m}(\beta_2)|_{V_n^m}. \] 
		Finally $\varphi_n^m(\beta_1 \beta_2)= \varphi_n^m(\beta_1) \circ \varphi_n^m(\beta_2)$. 
		\end{proof}	
		
		\begin{rem}
		As braid group representation, $V_n^0 \cong V_n^{\leq 0} $ (meaning $\varphi_n^0= \varphi_n^{\leq 0}$).
		
		\end{rem}		
		
		We may now state the factorisation of $r$-part representation, recall that the Frobenius map $F_r: \Z[q^{\alpha}] \to \Z[q^{\alpha}]$ sends $q^{\alpha} \mapsto q^{r\alpha}$ ($s \to s^r$ in the Laurent polynomials language):
		
		\begin{prop}\label{prop_braid_factor}
		The isomorphism 
		\[
		\Phi : \bfct V_n^m &\to& V_n^0 \otimes F_r(\SB_{n,m})\\
			 v_{\overline{i+rj}} &\mapsto& v_{\overline{i}} \otimes F_r(w_{\overline{j}}) \efct,
		\]
		where $\overline{i+rj}=(i_1+r j_1, \dots, i_n + rj_n)$ with $i_1, \dots ,i_n \leq r-1$,
		is a braid group representation isomorphism. 
		
		In other word, the following diagram commutes:
		
		\begin{center}
		\begin{tikzpicture}
		  \matrix (m) [matrix of math nodes,row sep=3em,column sep=4em,minimum width=2em]
		  {
		     V_n^m & V_n^m \\
		     V_n^0 \otimes F_r(\SB_{n,m}) & V_n^0 \otimes F_r(\SB_{n,m}) \\};
		  \path[-stealth]
		    (m-1-1) edge node [left] {$\Phi$} (m-2-1)
		            edge node [below] {$\varphi_n^m$} (m-1-2)
		    (m-2-1) edge node [below=0.2cm] {$\varphi_n^0 \otimes (F_r \circ \psi_{n,m})$} (m-2-2)
		    (m-1-2) edge node [right] {$\Phi$} (m-2-2);
		\end{tikzpicture}
		\end{center}
		
		\end{prop} 
		
		\begin{proof}
		Using Lemma 26 in \cite{willetts2020unification} we can factorise the action of the $R$ matrix as follows.
		
		Let $0 \leq a,b,i \leq r-1$ such that $ 0 \leq a+i \leq r-1$ and $ 0 \leq b-i \leq r-1$, we have:
		
		\begin{align*}
		q^{\frac{H\otimes H}{2}}(q^{\frac{(i+rj)(i+rj-1)}{2}} E^{i+rj} \otimes F^{(i+rj)}). v_{b+rv} \otimes v_{a+ru}&=q^{\frac{\alpha^2}{2}} q^{\frac{(i+rj)(i+rj-1)}{2}} \qbinom{i+rj+a+ru}{i+rj}_{q} \\ & \times \{\alpha-a-ru; i+rj\}_{q} q^{-(a+ru+b+rv) \alpha} \\ & \times q^{ 2(a+ru+i+rj)(b+rv-i-rj)}  v_{b+rv-i-rj} \otimes v_{a+ru+i+rj}\\
		 &=q^{\frac{\alpha^2}{2}} q^{\frac{i(i-1)}{2}} \qbinom{a+i}{i}_{q}\\ & \times \{\alpha-a; i\}_{q}  q^{-(a+b) \alpha} q^{ 2(a+i)(b-i)}v_{b-i} \otimes v_{a+i} \\ & \otimes F_r\left( \binom{u+j}{j} \{\alpha\}^{j} q^{-(u+v)\alpha} w_{v-j} \otimes w_{u+j}\right) .
		\end{align*}
		
		Hence we have,
		
		\begin{align*}
		\Phi \left(\rho_2^{u+v}(R. v_{b+rv} \otimes v_{a+ru}) \right)& = (R. v_{b} \otimes v_{a}) \otimes F_r(R.w_{v} \otimes w_{u})
		\end{align*}
		
		Finally,
		\begin{align*}
		 \Phi \left( \varphi_n^m(\beta).v_{\overline{i+rj}} \right) &= \varphi_n^0(\beta).v_{\overline{i}} \otimes F_r\left( \psi_{n,m}(\beta).w_{\overline{j}} \right)
		\end{align*}
		\end{proof}
		
				\begin{ex}
		Figure \ref{pyramid_factor} illustrates the weight level pyramid at $n=2$ and $q=\zeta_{6}$ where we denote \[v_{a,b}= v_a \otimes v_b .\]
		
		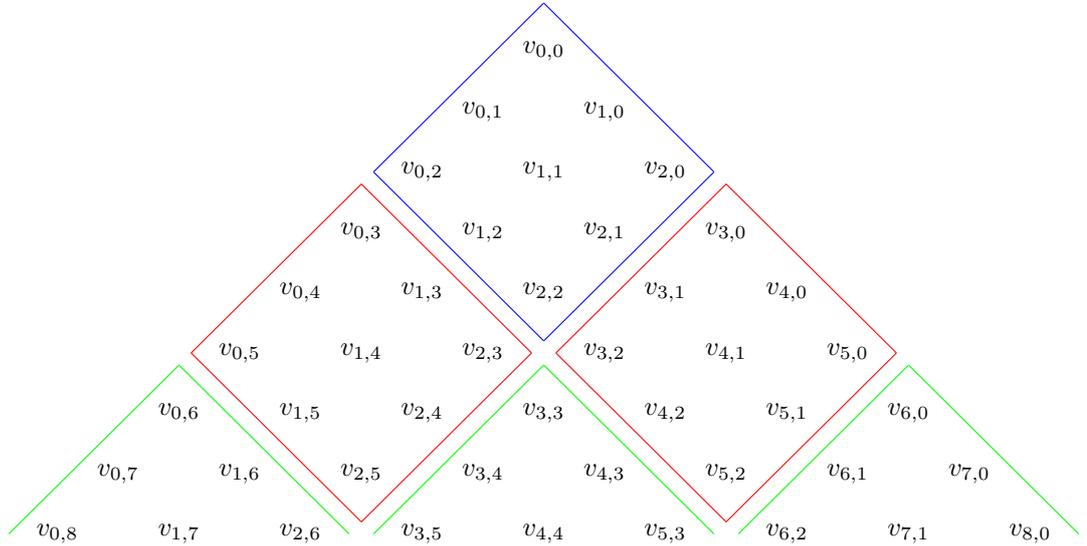
\begin{figure}[h]
		\centering
		\begin{tikzpicture}[scale=0.8]
		\foreach \n in {0,...,8} {
		  \foreach \k in {0,...,\n} {
		  	\edef\aux{\n};
		  	\pgfmathparse{\n-\k};
  			\edef\aux{\pgfmathresult};
		    \node at (2*\k-\n,-\n) {$v_{\k,\pgfmathprintnumber[int trunc]{\aux}}$};
		  }
		}
		\draw[blue] (0,0.8) -- (-2-0.8,-2);
		\draw[blue] (-2-0.8,-2) -- (0,-4-0.8);
		\draw[blue] (0,-4-0.8) -- (2+0.8,-2);
		\draw[blue] (2+0.8,-2) -- (0,0.8);
		
		\foreach \p in {-1,1} {
			\draw[red] (0+3*\p,0.8-3) -- (3*\p-2-0.8,-2-3);
			\draw[red] (3*\p-2-0.8,-2-3) -- (0+3*\p,-4-0.8-3);
			\draw[red] (0+3*\p,-4-0.8-3) -- (2+0.8+3*\p,-2-3);
			\draw[red] (2+0.8+3*\p,-2-3) -- (0+3*\p,0.8-3);
		}
		
		\foreach \p in {-2,0,2} {
			\draw[green] (0+3*\p,0.8-6) -- (3*\p-2-0.8,-2-6);
			\draw[green] (2+0.8+3*\p,-2-6) -- (0+3*\p,0.8-6);
		}
		\end{tikzpicture}
		\caption{Weight level pyramid factorisation at root of unity}
		\label{pyramid_factor}
		\end{figure}
		
		The blue square delimits generators of $V_n^0$, the red squares of $V_n^1$, etc. Each squares correspond to a tensor in the pyramid at $q=1$ as shown in Figure \ref{pyramid_one}. Families of colored squares are stable under the braid action $\varphi_n^m$ where $m$ correspond to a color. The union of a colored family plus higher colored family in the pyramid are stable under the whole quantum braid action $\varphi_n$ (e.g. the union of red and blue vectors from Figure \ref{pyramid_factor} is stable under $\Bn$ action). 
		
		\begin{figure}[h]
		\centering
		\begin{tikzpicture}[scale=1.5],every node/.style={scale=0.5}]
			\node at (0,0) {\scalebox{1.8}{$\color{blue} w_{0,0}$}};
			\node at (-1,-1) {\scalebox{1.8}{$\color{red} w_{0,1}$}};
			\node at (1,-1) {\scalebox{1.8}{$\color{red} w_{1,0}$}};
			\node at (-2,-2) {\scalebox{1.8}{$\color{green} w_{0,2}$}};
			\node at (0,-2) {\scalebox{1.8}{$\color{green} w_{1,1}$}};
			\node at (2,-2) {\scalebox{1.8}{$\color{green} w_{2,0}$}};
			
		\end{tikzpicture}
		\caption{Weight level pyramid at $q=1$}
		\label{pyramid_one}
		\end{figure}
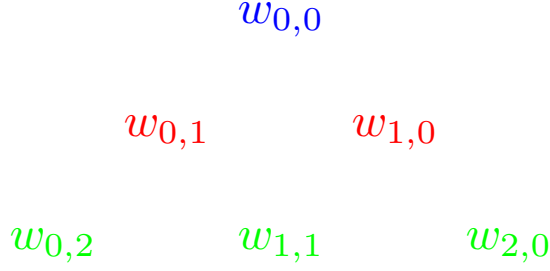

		\end{ex}

	

	\section{Unified invariant of knots from quantum braid representations}\label{sec_Foo_frombraids}
	
We want to define the knot invariant $F_{\infty}$ from \cite{willetts2020unification} from braid group representation on tensor products of Verma modules, defined above. We need a completion of the ring $\Laurent$ as $F_{\infty}$ will be some series living in this completion. We start with definitions for this ring and for the invariant in Sec. \ref{sec_Foo_def}. Then (Sec. \ref{sec_Foo_def_frombraids}) we can define (Theorem \ref{prop_unifed_braidrep}) $F_{\infty}$ from the braid action on tensors of Verma modules. In Sec. \ref{sec_Foo_factorization} we use the r-part factorization of the braid action at roots of unity to prove the factorization of $F_{\infty}$ at roots of unity recovering ADO polynomials (Theorem \ref{thm_factorisation_unified_ADO}). In Sec. \ref{sec_Alex_sym} we prove Theorem \ref{symmetry_unified_invariant} which shows a symmetry in variables for $F_{\infty}$ resembling that of the Alexander polynomial. 	
	
\subsection{Ring completion and unified invariant}\label{sec_Foo_def}

We recall $\Laurent=\Z[q^{\pm 1}, s^{\pm 1}]$, we will construct a completion of that ring. For the sake of simplicity, we will denote $q^{\alpha} := s$ as explained before. 

\begin{defn}
Let $I_n$ be the ideal of $\Laurent$ generated by the following set $\left\{ \ \{ \alpha+l; n \}_q , \ l \in \Z \right\}$.
\end{defn}

We then have a projective system : \[ \hat{I} : I_1 \supset I_2 \supset \dots \supset I_n \supset \dots \]
From which we can define the completion of $\Laurent$ from the projective limit as follows.

\begin{defn}
Let $\Laurentcomplet = \underset{\underset{n}{\leftarrow}}{\lim} \dfrac{R}{I_n} = \{ (a_n)_{n \in \N^*} \in \prod_{i=1}^{\infty} \frac{R}{I_n} \ | \ p_n(a_{n+1})=a_n \}$  where $p_n: \frac{R}{I_{n+1}} \to \frac{R}{I_{n}}$ is the projection map.
\end{defn}

\begin{rem}
\begin{itemize}
\item If $b_0 \in R$ and $b_n \in I_{n-1}$ for $n \geq 1$, the partial sums $ \underset{i=0}{\overset{N}{\sum}} b_n $ converge in $\Laurentcomplet$ as $N$ goes to infinity.
\item We denote the limit $\underset{i=0}{\overset{+\infty}{\sum}} b_n := (\overline{\underset{i=0}{\overset{N}{\sum}} b_n})_{N \in \N^*}$.
\item Conversely, if $a= (\overline{a_N})_{N  \in \N^*} \in \Laurentcomplet$, let $a_n \in R$ be any representative of $\overline{a_n}$ in $R$, then $a= \underset{i=0}{\overset{+\infty}{\sum}} b_n$ where $b_0=a_1$ and $b_{n}=a_{n+1}-a_n$ for $n \in \N^*$.
 \end{itemize}
\end{rem}
The completion $\Laurentcomplet$ contains $\Laurent$: 
\begin{prop}
The canonical projection maps induce an injective map $\Laurent \xhookrightarrow{} \Laurentcomplet$
\end{prop}
\begin{proof}
See Prop 17 in \cite{willetts2020unification}.
\end{proof}
	
		We now recall how the unified invariant $F_{\infty}(q,q^{\alpha},\mathcal{K})$ is defined using states diagrams of the knot, which is the subject of \cite{willetts2020unification}. 
		
		\bigskip
		For any knot seen as a $(1,1)$-tangle, take a diagram $D$ and $\overline{i}=(i_1, \dots, i_N) \in \N^N$ where $N$ is the number of crossings of $D$. 
		
		Label the top and bottom strands $0$ and starting from the bottom strand, label the strand after the $k$-th crossing encountered with the rule described in Figure \ref{crossings_simple_2}. The resulting labeled diagram is called a \textit{state diagram} of $D$, we denote it $D_{\overline{i}}$.
		
		\begin{figure}[h!]
		\begin{subfigure}[b]{0.5\textwidth}
		 \centering
		  \def\svgwidth{25mm}
\begingroup%
  \makeatletter%
  \providecommand\color[2][]{%
    \errmessage{(Inkscape) Color is used for the text in Inkscape, but the package 'color.sty' is not loaded}%
    \renewcommand\color[2][]{}%
  }%
  \providecommand\transparent[1]{%
    \errmessage{(Inkscape) Transparency is used (non-zero) for the text in Inkscape, but the package 'transparent.sty' is not loaded}%
    \renewcommand\transparent[1]{}%
  }%
  \providecommand\rotatebox[2]{#2}%
  \newcommand*\fsize{\dimexpr\f@size pt\relax}%
  \newcommand*\lineheight[1]{\fontsize{\fsize}{#1\fsize}\selectfont}%
  \ifx\svgwidth\undefined%
    \setlength{\unitlength}{666.14173228bp}%
    \ifx\svgscale\undefined%
      \relax%
    \else%
      \setlength{\unitlength}{\unitlength * \real{\svgscale}}%
    \fi%
  \else%
    \setlength{\unitlength}{\svgwidth}%
  \fi%
  \global\let\svgwidth\undefined%
  \global\let\svgscale\undefined%
  \makeatother%
  \begin{picture}(1,1.34042553)%
    \lineheight{1}%
    \setlength\tabcolsep{0pt}%
    \put(0,0){\includegraphics[width=\unitlength,page=1]{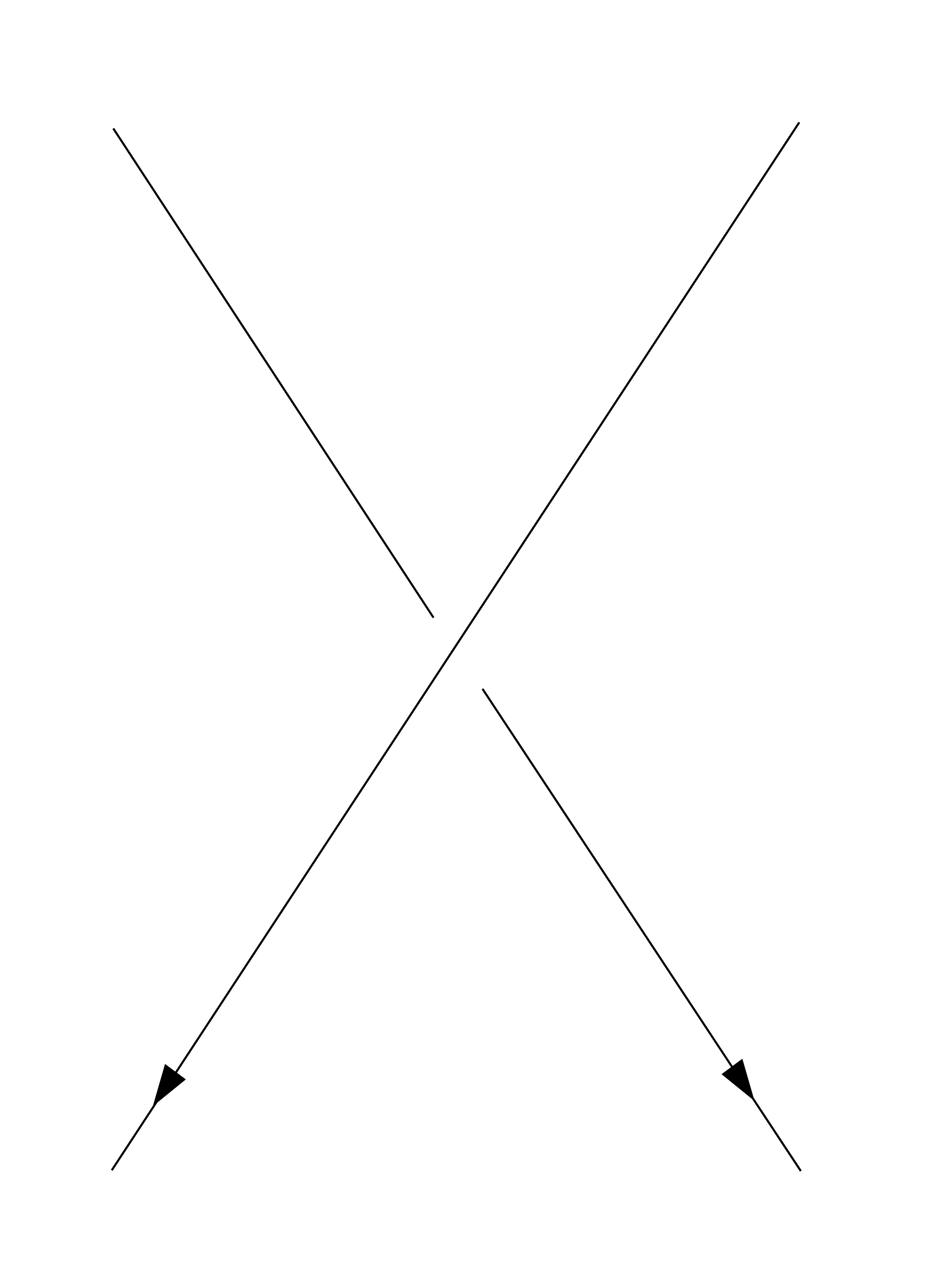}}%
    \put(0.84572423,0.02332179){\color[rgb]{0,0,0}\makebox(0,0)[lt]{\lineheight{1.25}\smash{\begin{tabular}[t]{l}$a_k$\end{tabular}}}}%
    \put(0.00726549,0.02199668){\color[rgb]{0,0,0}\makebox(0,0)[lt]{\lineheight{1.25}\smash{\begin{tabular}[t]{l}$b_k$\end{tabular}}}}%
    \put(0.73748325,1.25886345){\color[rgb]{0,0,0}\makebox(0,0)[lt]{\lineheight{1.25}\smash{\begin{tabular}[t]{l}$b_k -i_k$\end{tabular}}}}%
    \put(0.07803548,1.2652971){\color[rgb]{0,0,0}\makebox(0,0)[lt]{\lineheight{1.25}\smash{\begin{tabular}[t]{l}$a_k +i_k$\end{tabular}}}}%
  \end{picture}%
\endgroup%
		   \caption{Positive crossing.}
		 \end{subfigure}%
		 \begin{subfigure}[b]{0.5\textwidth}
		 \centering
		  \def\svgwidth{25mm}
\begingroup%
  \makeatletter%
  \providecommand\color[2][]{%
    \errmessage{(Inkscape) Color is used for the text in Inkscape, but the package 'color.sty' is not loaded}%
    \renewcommand\color[2][]{}%
  }%
  \providecommand\transparent[1]{%
    \errmessage{(Inkscape) Transparency is used (non-zero) for the text in Inkscape, but the package 'transparent.sty' is not loaded}%
    \renewcommand\transparent[1]{}%
  }%
  \providecommand\rotatebox[2]{#2}%
  \newcommand*\fsize{\dimexpr\f@size pt\relax}%
  \newcommand*\lineheight[1]{\fontsize{\fsize}{#1\fsize}\selectfont}%
  \ifx\svgwidth\undefined%
    \setlength{\unitlength}{666.14173228bp}%
    \ifx\svgscale\undefined%
      \relax%
    \else%
      \setlength{\unitlength}{\unitlength * \real{\svgscale}}%
    \fi%
  \else%
    \setlength{\unitlength}{\svgwidth}%
  \fi%
  \global\let\svgwidth\undefined%
  \global\let\svgscale\undefined%
  \makeatother%
  \begin{picture}(1,1.36170213)%
    \lineheight{1}%
    \setlength\tabcolsep{0pt}%
    \put(0.02895136,0.06594472){\color[rgb]{0,0,0}\makebox(0,0)[lt]{\lineheight{1.25}\smash{\begin{tabular}[t]{l}$a_k$\end{tabular}}}}%
    \put(0.84393733,0.06301101){\color[rgb]{0,0,0}\makebox(0,0)[lt]{\lineheight{1.25}\smash{\begin{tabular}[t]{l}$b_k$\end{tabular}}}}%
    \put(0.02204024,1.29666105){\color[rgb]{0,0,0}\makebox(0,0)[lt]{\lineheight{1.25}\smash{\begin{tabular}[t]{l}$b_k -i_k$\end{tabular}}}}%
    \put(0.65897041,1.30148629){\color[rgb]{0,0,0}\makebox(0,0)[lt]{\lineheight{1.25}\smash{\begin{tabular}[t]{l}$a_k +i_k$\end{tabular}}}}%
    \put(0,0){\includegraphics[width=\unitlength,page=1]{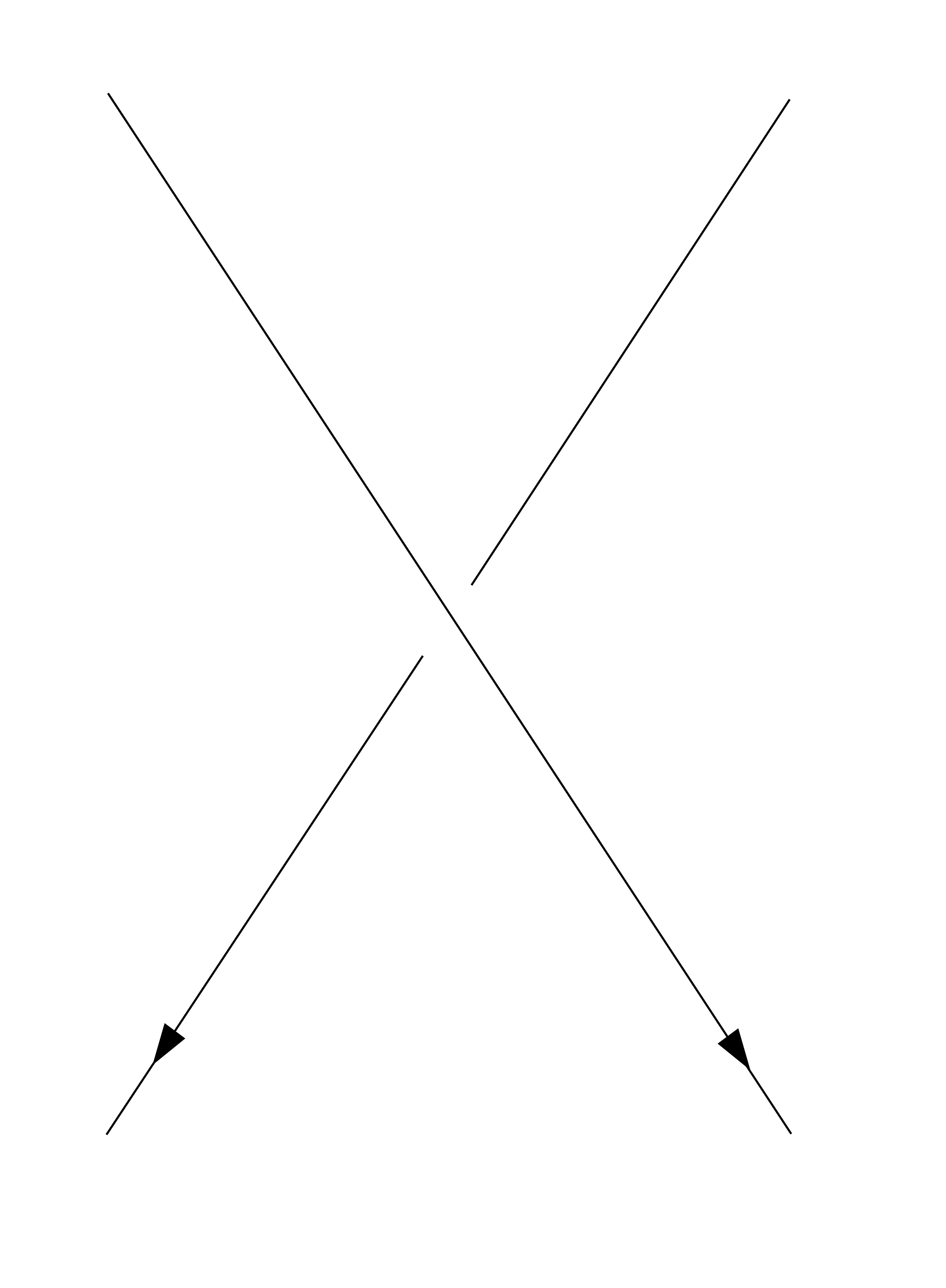}}%
  \end{picture}%
\endgroup%
		   \caption{Negative crossing.}
		 \end{subfigure}%
		 \caption{The two possibilities for the k-th crossing in $D$.}
		 \label{crossings_simple_2}
		 \end{figure}
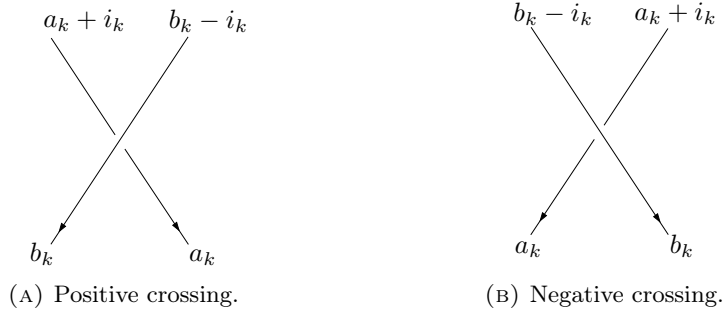

		\bigskip

		Let $D_{\overline{i}}$ be a state diagram of $D$, we define:
		\begin{align*}
		D(i_1, \dots, i_N) = &q^{\frac{f \alpha^2}{2}}(\prod_{j=1}^{S} q^{\mp(\alpha -2 \epsilon_j)})\prod_{k \in pos} q^{\frac{i_k(i_k-1)}{2}} \qbinom{a_k+i_k}{i_k}_{q}  \{\alpha-a_k; i_k\}_{q} \\ &  \times  q^{-(a_k+b_k) \alpha} q^{2(a_k+ i_k)(b_k-i_k)} \prod_{k \in neg}  (-1)^{i_k} q^{-\frac{i_k(i_k-1)}{2}} \qbinom{a_k+i_k}{i_k}_{q} \\ &  \times  \{\alpha-a_k; i_k\}_{q} q^{(a_k+b_k) \alpha} q^{- 2a_k b_k}
		\end{align*}
		where :
		\begin{itemize}
		\item $f$ is the writhe of $D$,
		\item $neg \ \cup \ pos = [|1, N|]$ and $k \in pos$ if the $k$-th crossing of D is positive, else $k \in neg$,
		\item $a_k, b_k$ are the strands' labels at the $k$-th crossing of the state diagram (see Figure \ref{crossings_simple_2}),
		\item  $S$ is the number of \def\svgwidth{5mm}
\begingroup%
  \makeatletter%
  \providecommand\color[2][]{%
    \errmessage{(Inkscape) Color is used for the text in Inkscape, but the package 'color.sty' is not loaded}%
    \renewcommand\color[2][]{}%
  }%
  \providecommand\transparent[1]{%
    \errmessage{(Inkscape) Transparency is used (non-zero) for the text in Inkscape, but the package 'transparent.sty' is not loaded}%
    \renewcommand\transparent[1]{}%
  }%
  \providecommand\rotatebox[2]{#2}%
  \newcommand*\fsize{\dimexpr\f@size pt\relax}%
  \newcommand*\lineheight[1]{\fontsize{\fsize}{#1\fsize}\selectfont}%
  \ifx\svgwidth\undefined%
    \setlength{\unitlength}{867.71048959bp}%
    \ifx\svgscale\undefined%
      \relax%
    \else%
      \setlength{\unitlength}{\unitlength * \real{\svgscale}}%
    \fi%
  \else%
    \setlength{\unitlength}{\svgwidth}%
  \fi%
  \global\let\svgwidth\undefined%
  \global\let\svgscale\undefined%
  \makeatother%
  \begin{picture}(1,0.55772942)%
    \lineheight{1}%
    \setlength\tabcolsep{0pt}%
    \put(0,0){\includegraphics[width=\unitlength,page=1]{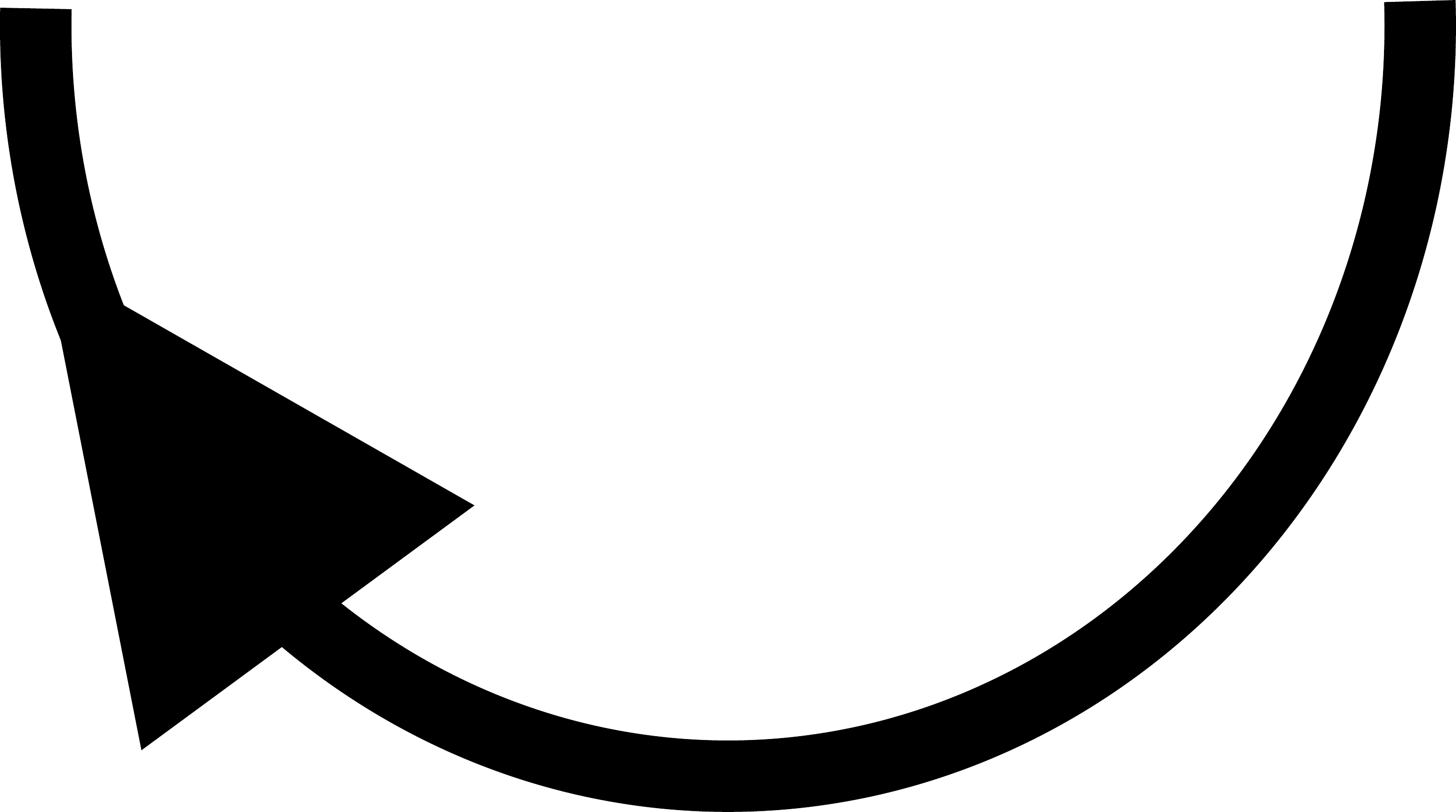}}%
  \end{picture}%
\endgroup%
 $+$ \def\svgwidth{5mm}
\begingroup%
  \makeatletter%
  \providecommand\color[2][]{%
    \errmessage{(Inkscape) Color is used for the text in Inkscape, but the package 'color.sty' is not loaded}%
    \renewcommand\color[2][]{}%
  }%
  \providecommand\transparent[1]{%
    \errmessage{(Inkscape) Transparency is used (non-zero) for the text in Inkscape, but the package 'transparent.sty' is not loaded}%
    \renewcommand\transparent[1]{}%
  }%
  \providecommand\rotatebox[2]{#2}%
  \newcommand*\fsize{\dimexpr\f@size pt\relax}%
  \newcommand*\lineheight[1]{\fontsize{\fsize}{#1\fsize}\selectfont}%
  \ifx\svgwidth\undefined%
    \setlength{\unitlength}{867.71048959bp}%
    \ifx\svgscale\undefined%
      \relax%
    \else%
      \setlength{\unitlength}{\unitlength * \real{\svgscale}}%
    \fi%
  \else%
    \setlength{\unitlength}{\svgwidth}%
  \fi%
  \global\let\svgwidth\undefined%
  \global\let\svgscale\undefined%
  \makeatother%
  \begin{picture}(1,0.55772942)%
    \lineheight{1}%
    \setlength\tabcolsep{0pt}%
    \put(0,0){\includegraphics[width=\unitlength,page=1]{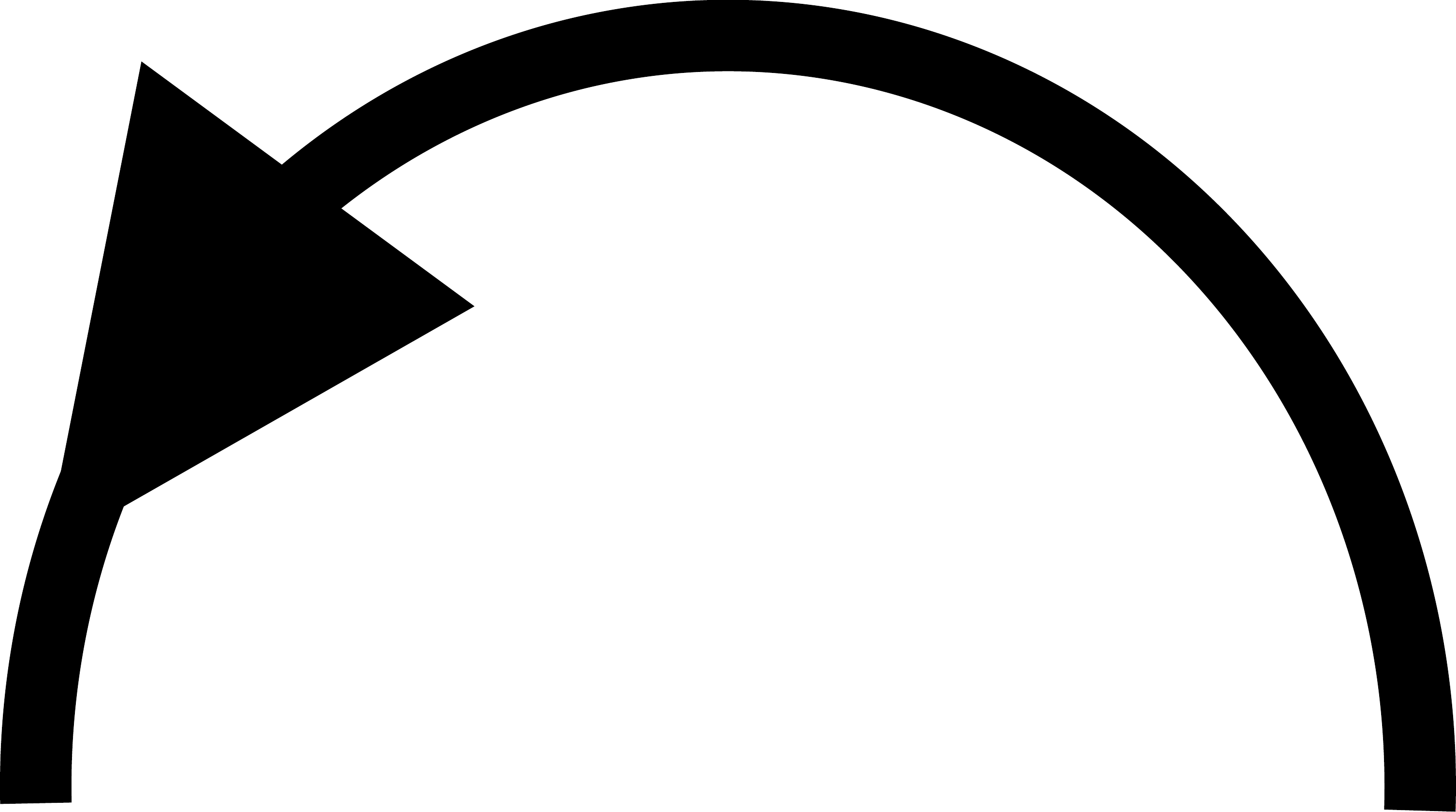}}%
  \end{picture}%
\endgroup%
 appearing in the diagram, and $\epsilon_j$ the strand label at the j-th \def\svgwidth{5mm} or \def\svgwidth{5mm}, the $\mp$ sign is negative if \def\svgwidth{5mm} and positive if \def\svgwidth{5mm}.
		\end{itemize}

		\begin{rem}
		$D(i_1, \dots, i_N)$ is the scalar one gets by considering only the $E^{i_k} \otimes F^{(i_k)}$ term in the the $R$-matrix action of the $k$-th crossing of $D$.
		
		\end{rem}

		\begin{ex}
		See Figure \ref{state_diagrams_examples} for some examples of state diagrams.
		\begin{figure}[h!]
		\begin{subfigure}[b]{0.5\textwidth}
		 \centering
		  \def\svgwidth{30mm}
\begingroup%
  \makeatletter%
  \providecommand\color[2][]{%
    \errmessage{(Inkscape) Color is used for the text in Inkscape, but the package 'color.sty' is not loaded}%
    \renewcommand\color[2][]{}%
  }%
  \providecommand\transparent[1]{%
    \errmessage{(Inkscape) Transparency is used (non-zero) for the text in Inkscape, but the package 'transparent.sty' is not loaded}%
    \renewcommand\transparent[1]{}%
  }%
  \providecommand\rotatebox[2]{#2}%
  \newcommand*\fsize{\dimexpr\f@size pt\relax}%
  \newcommand*\lineheight[1]{\fontsize{\fsize}{#1\fsize}\selectfont}%
  \ifx\svgwidth\undefined%
    \setlength{\unitlength}{278.06085061bp}%
    \ifx\svgscale\undefined%
      \relax%
    \else%
      \setlength{\unitlength}{\unitlength * \real{\svgscale}}%
    \fi%
  \else%
    \setlength{\unitlength}{\svgwidth}%
  \fi%
  \global\let\svgwidth\undefined%
  \global\let\svgscale\undefined%
  \makeatother%
  \begin{picture}(1,2.05525577)%
    \lineheight{1}%
    \setlength\tabcolsep{0pt}%
    \put(0,0){\includegraphics[width=\unitlength,page=1]{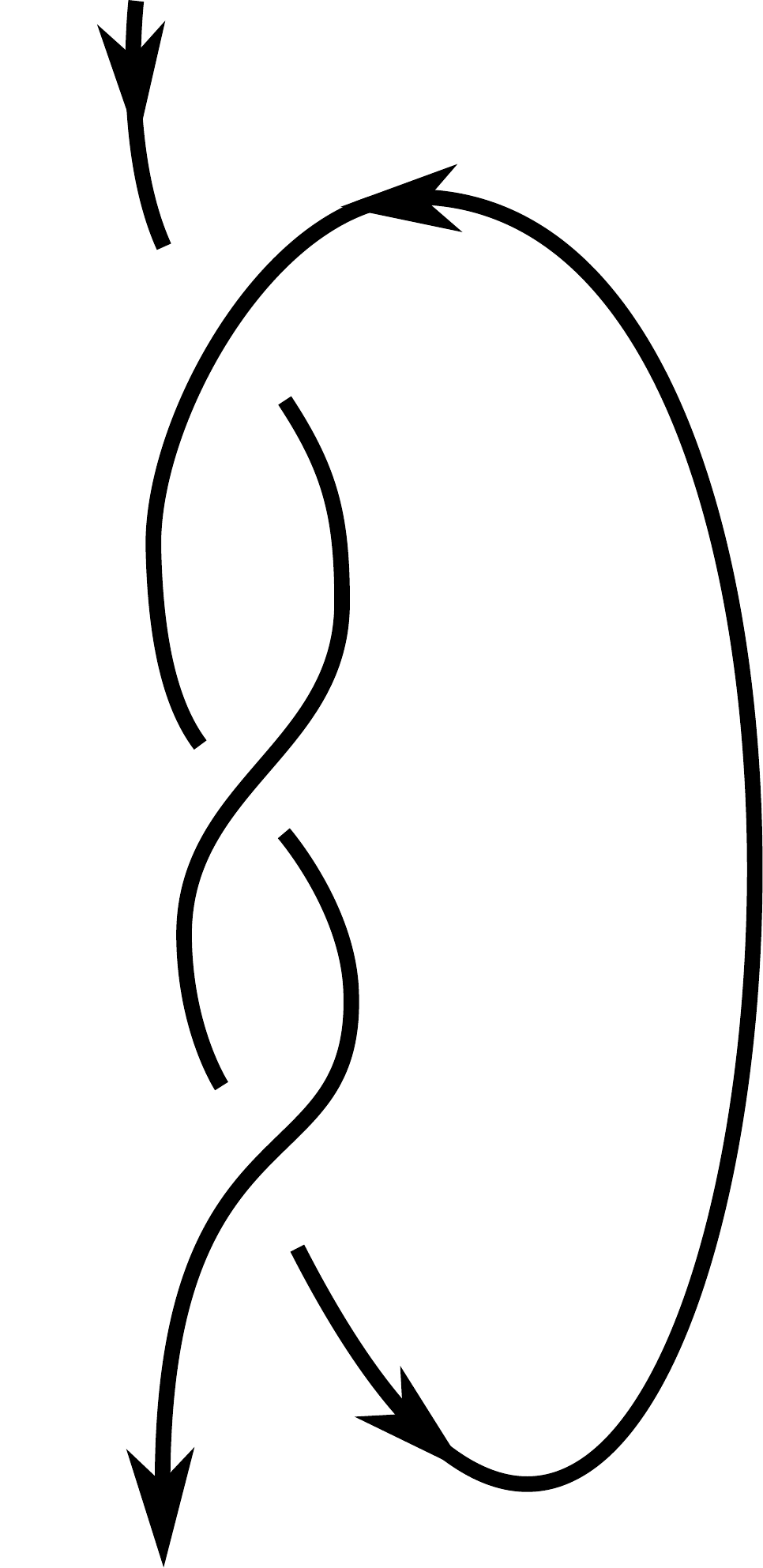}}%
    \put(0.10029114,0.24915924){\color[rgb]{0,0,0}\makebox(0,0)[lt]{\lineheight{1.25}\smash{\begin{tabular}[t]{l}$0$\end{tabular}}}}%
    \put(0.51682291,0.74544307){\color[rgb]{0,0,0}\makebox(0,0)[lt]{\lineheight{1.25}\smash{\begin{tabular}[t]{l}$0$\end{tabular}}}}%
    \put(-0.00358229,1.84347108){\color[rgb]{0,0,0}\makebox(0,0)[lt]{\lineheight{1.25}\smash{\begin{tabular}[t]{l}$0$\end{tabular}}}}%
    \put(0.5031998,1.34486289){\color[rgb]{0,0,0}\makebox(0,0)[lt]{\lineheight{1.25}\smash{\begin{tabular}[t]{l}$0$\end{tabular}}}}%
    \put(0.0512477,1.26544799){\color[rgb]{0,0,0}\makebox(0,0)[lt]{\lineheight{1.25}\smash{\begin{tabular}[t]{l}$i$\end{tabular}}}}%
    \put(0.09408526,0.76996479){\color[rgb]{0,0,0}\makebox(0,0)[lt]{\lineheight{1.25}\smash{\begin{tabular}[t]{l}$i$\end{tabular}}}}%
    \put(0.87060631,1.0015588){\color[rgb]{0,0,0}\makebox(0,0)[lt]{\lineheight{1.25}\smash{\begin{tabular}[t]{l}$i$\end{tabular}}}}%
  \end{picture}%
\endgroup%
		   \caption{The trefoil knot.}
		 \end{subfigure}%
		 \begin{subfigure}[b]{0.5\textwidth}
		 \centering
		  \def\svgwidth{55mm}
		    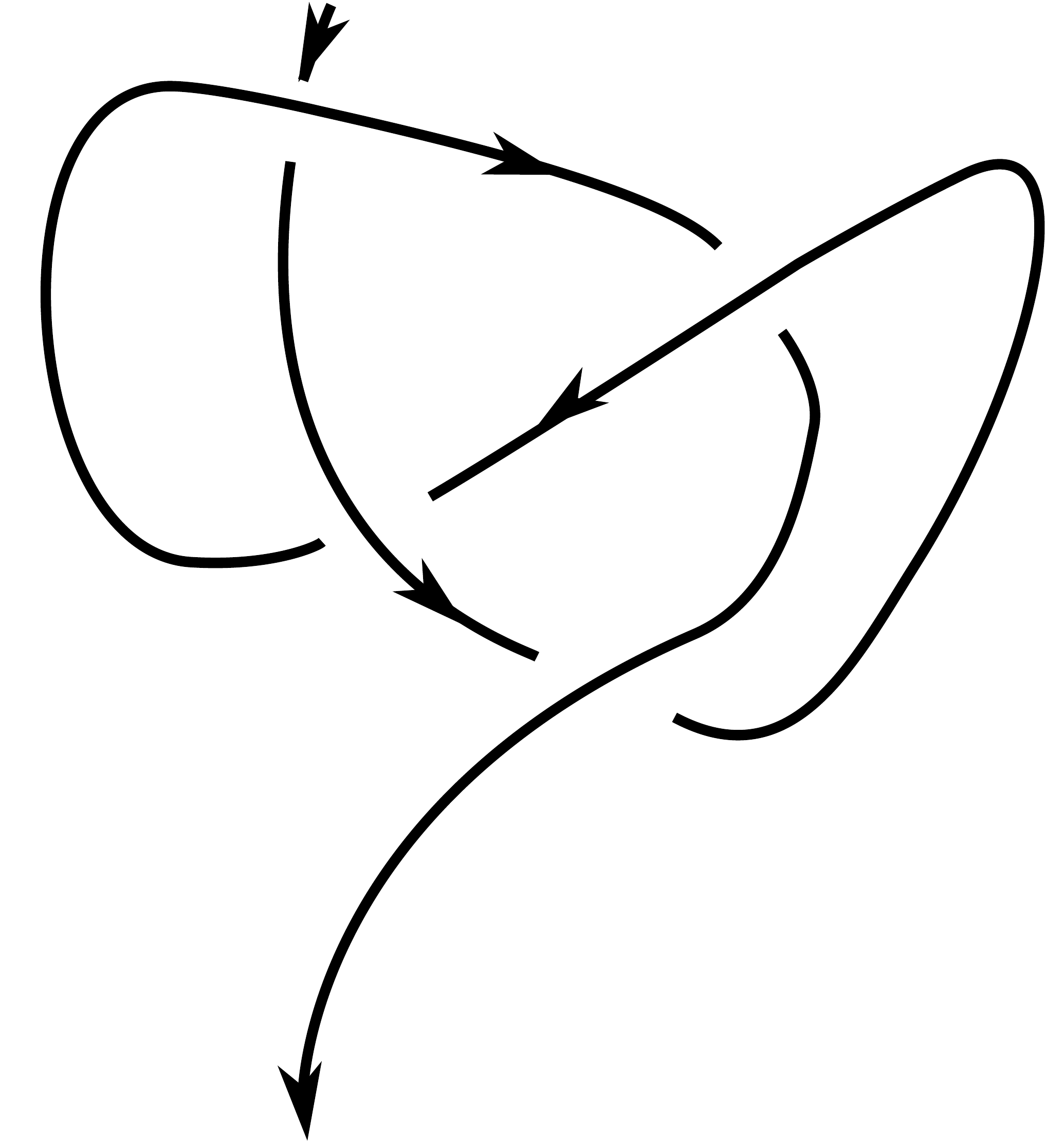%
		   \caption{The figure eight knot.}
		 \end{subfigure}%
		 \caption{Examples of state diagrams to compute the invariants.}
		 \label{state_diagrams_examples}
		 \end{figure}
		
		\end{ex}

		\begin{defn}[{\cite[Def. 20]{willetts2020unification}}]
		We use a definition removing the quadratic component: $$F_{\infty}(q,A,\mathcal{K}):=q^{\frac{-f \alpha^2}{2}}\underset{\overline{i}=0}{\overset{+\infty}{\sum}}D(i_1, \dots, i_N)$$.
		\end{defn}

%
%
		\subsection{Unified invariant from the action of braids on Verma modules}\label{sec_Foo_def_frombraids}
		
%
%

		
		We recall the definition of the braid group representation on tensor products of Verma modules:\[ \varphi_n(q^{\alpha},.) : B_n \to \End((V^{\alpha})^{\otimes n}). \]
		The notion of {\em partial trace} is used to compute knot invariants out of finite dimensional quantum braid representation. We extend this notion to infinite dimensional modules in the case of Verma modules. 

		\begin{defn}[Partial trace on Verma modules]
		Let $\beta \in B_n$ whose closure is a knot,
		\[\Tr_{2,\dots, n} ((1 \otimes K^{\otimes n-1}) \varphi_{n}(\beta)):=\underset{\overline{j}\in \BN_0^n}{\sum} [((1 \otimes K^{\otimes n-1})\varphi_n(q^{\alpha},\beta)) v_{\overline{j}} )]_{v_{\overline{j}}}\in \Laurentcomplet, \]
		where :
		\begin{itemize}
		\item $\BN_0^n := \lbrace (0, j_2, \dots, j_n) \in \N^n \rbrace$,
		\item $[((1 \otimes K^{\otimes n-1})\varphi_n(q^{\alpha},\beta)) v_{\overline{j}} )]_{v_{\overline{j}}} \in \Z[q^{\pm}, q^{\pm \alpha}]$ is the projection of $(1 \otimes K^{\otimes n-1})\varphi_n(q^{\alpha},\beta)v_{\overline{j}}$ on $v_{\overline{j}}$.
		\end{itemize}
		\end{defn}
		
		It is called {\em partial trace} inherited from the standard notion of partial trace on tensor products of vector spaces, see Sec. \ref{sec_Alex_sym}. 
		

		Let $\mathcal{K}$ be a long knot, $\beta \in B_n$ whose closure is $\mathcal{K}$ and $D^{\beta}$ be the diagram associated with $\mathcal{K}$ seen as the closure of $\beta$. The general picture is the following. 
		
		\begin{figure}[h!]
		\begin{subfigure}[b]{1\textwidth}
		 \centering
		  \def\svgwidth{70mm}
		    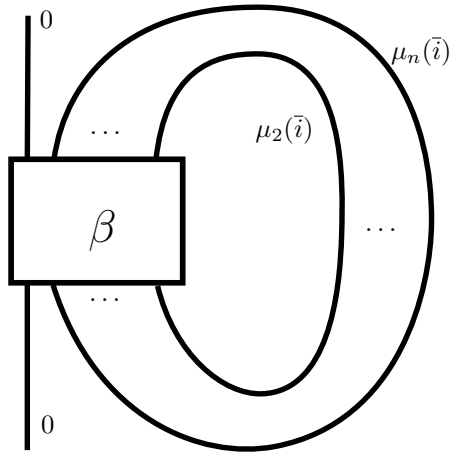
		 \end{subfigure}%
		 \caption{State diagram of a braid partial closure.}
		 \label{diagram_braid}
		 \end{figure}
		 
		 The result is the following.
		 
		 \begin{thm}\label{prop_unifed_braidrep}
		Let $\mathcal{K}$ be a knot in $S^3$ and $\beta \in B_n$ a braid whose closure is $\mathcal{K}$, then we have \[ F_{\infty}(q,q^{\alpha},\mathcal{K}) = \Tr_{2,\dots, n} ((1 \otimes K^{\otimes n-1})\varphi_n(q^{\alpha},\beta)) \in \Laurentcomplet\]
		\end{thm}
		\begin{proof}
		We denote $ \mu_2 (\overline{i}), \dots , \mu_n (\overline{i})$ the labels of the closure strands of the state diagram $D^{\beta}_{\overline{i}}$. Let $\mu(\overline{i})= (0,\mu_2 (\overline{i}), \dots , \mu_n (\overline{i}))$.
		
		We can then write, \[ [((1 \otimes K^{\otimes n-1})\varphi_n(q^{\alpha},\beta)) v_{\overline{j}} )]_{v_{\overline{j}}} = \underset{\underset{\mu(\overline{i})= \overline{j}}{\overline{i}=0} }{\overset{+\infty}{\sum}}D(i_1, \dots, i_N), \]
		where $[((1 \otimes K^{\otimes n-1})\varphi_n(q^{\alpha},\beta)) v_{\overline{j}} )]_{v_{\overline{j}}} \in \Z[q^{\pm}, q^{\pm \alpha}]$ is the projection of $(1 \otimes K^{\otimes n-1})\varphi_n(q^{\alpha},\beta)v_{\overline{j}}$ on $v_{\overline{j}}$.
		
		Hence, \[ \underset{\overline{j}=0}{\overset{+\infty}{\sum}} [((1 \otimes K^{\otimes n-1})\varphi_n(q^{\alpha},\beta)) v_{\overline{j}} )]_{v_{\overline{j}}} = \underset{\overline{i}=0 }{\overset{+\infty}{\sum}}D(i_1, \dots, i_N).\]
		
		Finally we get, \[  \Tr_{2,\dots, n} ((1 \otimes K^{\otimes n-1})\varphi_n(q^{\alpha},\beta)) =F_{\infty}(q,q^{\alpha},\mathcal{K}), \]
		
		concluding the proof.
		
		\end{proof}

		Using Theorem \ref{prop_unifed_braidrep}, we can then write the unified invariant using these sub-representations.
		
		\begin{cor}
		Let $\mathcal{K}$ be a knot in $S^3$ and $\beta \in B_n$ a braid whose closure is $\mathcal{K}$, then we have
		\[ F_{\infty}(q,q^{\alpha},\mathcal{K})= \sum_m \Tr_{2,\dots, n}((1 \otimes K^{\otimes n-1}) \varphi_{n,m}(\beta))\]
		\end{cor}

		\subsection{At roots of unity: factorization of the unified invariant}\label{sec_Foo_factorization}
		Now we can finally factorize the unified invariant at roots of unity using braid representations. The result is already given in \cite{willetts2020unification}, it uses a conjecture of Melvin--Morton--Rozanski (MMR) proved by Bar-Natan--Garoufalidis rather than a structural study of braid representations on Verma modules. Here, we give another proof of the result, using braid group representations. Hence it re-proves MMR conjecture and moreover a generalization of it. This subsection assumes $q=\zeta_{2r}$. We refer to Section \ref{section_rparts} for notations of submodules in this case. 
		\medskip
		
		First of all we can obtain {\em $ADO$ polynomials} (sometime called {\em colored Alexander invariants}) with the $0$ $r$-part representation.
		
		\begin{prop}\label{prop_braid_ado}
		\[ ADO_r(q^{\alpha},\mathcal{K})= \Tr_{2,\dots, n}((1 \otimes (K^{1-r})^{\otimes n-1})\varphi_n^0(\beta)) \]
		\end{prop}
		
		\begin{proof}
		This is the same proof as that for the unified invariant in Theorem \ref{prop_unifed_braidrep} (but from the definition of ADO polynomials, see Prop. 13 of \cite{willetts2020unification}), setting $q=\zeta_{2r}$, using truncated $R$ matrix and $K^{r-1}$ as pivotal element.
		\end{proof}

		Now let's see a factorization proposition at roots of unity, first recall that $F_r(q^{\alpha})=q^{r\alpha}$ defines the Frobenius morphism.
		
%
%
		
		The following is a direct corollary of Propositions \ref{prop_braid_factor} and \ref{prop_braid_ado}.
		\begin{cor}\label{cor_factor_trace}
		\[ \Tr_{2,\dots, n}((1 \otimes K^{\otimes n-1})\varphi_n^m(\beta)) = ADO_r(q^{\alpha}, \mathcal{K}) \times F_r(\Tr_{2, \dots, n}((1 \otimes K^{\otimes n-1})\psi_{n,m}(\beta)) \] 
		\end{cor}
		
		Moreover we can use Mac-Mahon master theorem to prove the following proposition.
		
		\begin{prop}\label{prop_trace_factor}
		For $\beta \in B_n$ whose closure is a knot $\mathcal{K}$, then :
		\[ \sum_m \Tr_{2,\dots,n}((1 \otimes K^{\otimes n-1})\psi_{n,m}(\beta))= \frac{q^{f\alpha}}{A_{\mathcal{K}}(q^{2\alpha})}.\]
		where $A_{\mathcal{K}}$ is the {\em Alexander polynomial}. 
		\end{prop}
		
		\begin{proof}
		Using Mac-mahon master theorem we have \[ \sum_m [\Sym^m(\psi_{n,1}(\beta)) v_{\overline{I}}]_{v_{\overline{i}}}t_1^{i_1} \otimes \dots \otimes t_n^{i_n} = \frac{1}{\det\left(I_n - 	\begin{psmallmatrix}
		t_1 & & \\
		& \ddots & \\
		& & t_n
		\end{psmallmatrix}
		 \psi_{n,1}(\beta) \right)} \]
		 
		 Now if one takes $t_1=0$ and $t_i=1$ for $i \neq 1$, we have the following equality:
		\[ \sum_m \Tr_{2,\dots, n } (\Sym^m(\psi_{n,1}(\beta) ))=  \frac{1}{\det(I_n - 			\begin{psmallmatrix}
		  0 
		  & 0\\
		  0 
		  & I_{n-1}
		\end{psmallmatrix}
		 \psi_{n,1}(\beta) )} \]
		 
		 Since $\psi_{n,1}(\beta)$ is the unreduced Burau representation $B(t)$ in the basis $f_k= q^{-k\alpha} e_k$ setting $t=q^{-2 \alpha}$ , and since we are taking a $(n-1) \times (n-1)$ minor of $I_n - B(t)$, we obtain: \[ det(I_n- \begin{psmallmatrix}
		  0 
		  & 0\\
		  0 
		  & I_{n-1}
		\end{psmallmatrix} \psi_{n,1}(\beta) = q^{(n-1+f)\alpha}A_{\mathcal{K}}(q^{2\alpha}). \]
		\end{proof}
		
		Now,
		\begin{align*} 
		F_{\infty}(\zeta_{2r},q^{\alpha},\mathcal{K}) &= \Tr_{2,\dots, n} ((1 \otimes K^{\otimes n-1})\varphi_n(q^{\alpha},\beta))\\
		&= \sum_m \Tr_{2,\dots, n} ((1 \otimes K^{\otimes n-1})\varphi_n^m(q^{\alpha},\beta))
		\end{align*} 
		using Corollary \ref{cor_factor_trace} and Proposition \ref{prop_trace_factor}, we get back the factorisation Theorem :
		
		\begin{thm}[Factorization of $F_{\infty}$ at roots of unity] \label{thm_factorisation_unified_ADO}
		For a knot $\mathcal{K}$ and an integer $r \in \N^*$, we have the following factorization in $\Laurentcomplet$:
		\[ F_{\infty} (\zeta_{2r}, A, \mathcal{K}) =  \frac{(q^{\alpha})^{rf} \times ADO_r(A,\mathcal{K})}{A_{\mathcal{K}} (A^{2r})} \]
		where $f$ is the framing of the knot. (We have named the variable $A$ instead of $s$ used to define the Verma modules. It is more standard when working with Alexander-like invariants.)
		\end{thm}
		
		
		
		We recall the Melvin--Morton--Rozanski conjecture, which is a theorem of Bar-Natan and Garoufalidis. 
		\begin{thm}[Bar-Natan, Garoufalidis {\cite{MMR}}] \label{MMR} ~\\
		For $\mathcal{K}$ a knot, we have the following equality in $\Q[[h]]$:
		\[ \underset{n \to \infty}{\lim} \  J_n (e^{h/n}) = \frac{1}{A_{\mathcal{K}} (e^h)}\]
		in the sense that, $\forall m \in \N$, \[ \underset{n \to \infty}{\lim} \  \text{coeff}\left(J_n (e^{h/n}),h^m \right) = \text{coeff}\left(\frac{1}{A_{\mathcal{K}} (e^h)}, h^m \right)  \]
		where, for any analytic function $f$, $\text{coeff}(f(h),h^m)=\frac{1}{m!} \frac{d^m}{dh^m} f(h)|_{h=0}$, and $J_n$ is the {\em $n$-th colored Jones polynomial}. 
		\end{thm}
		
		\begin{rem}[Re-proof of MMR conjecture]
		Let $\mathcal{K}$ be a $0$ framed knot. From the unified invariant,
		\begin{itemize}
		 \item on one hand we get the colored Jones polynomials back (see Corollary 59 in \cite{willetts2020unification}) \[F_{\infty}(q,q^n,  \mathcal{K})= J_n(q^2, \mathcal{K}), \]
		 \item on the other hand, using Theorem \ref{thm_factorisation_unified_ADO} at $r=1$, we get \[ F_{\infty} (1, A, \mathcal{K}) =  \frac{1}{A_{\mathcal{K}} (A^{2})} \]
		\end{itemize}
		 Using the identification $q=e^h$ and $q^{\alpha}=e^{\alpha h}$, we have an injective map (see Prop 6.8, 6.9 in \cite{habiro2007integral}) \[ \Laurentcomplet \to \Q[\alpha][[h]] \]
		 Hence, as elements in $\Q[\alpha][[h]]$, we have the following limit (in the sense defined in Theorem \ref{MMR}) \[ \underset{n \to \infty}{\lim} \  F_{\infty}(q^{\frac{1}{n}},q ,  \mathcal{K}) = F_{\infty}(1,q ,  \mathcal{K})  \] 
		 
		 Thus we get \begin{align*}
		\underset{n \to \infty}{\lim} \  J_n (e^{\frac{2h}{n}}) &=F_{\infty}(q^{\frac{1}{n}},q ,  \mathcal{K}) \\
		&= F_{\infty}(1,q ,  \mathcal{K}) \\
		&= \frac{1}{A_{\mathcal{K}} (e^{2h})}
		 \end{align*}
		giving us another proof of Theorem \ref{MMR}.
		\end{rem}
		
		
		\begin{prop}[Unicity property of $F_{\infty}$]\label{prop_unicity_Foo}
		Let $\CK$ be a knot, $F_{\infty}(q,A,\CK)$ is the only element in $\Laurentcomplet$ interpolating colored Jones polynomials or ADO over Alexander elements at an infinite number of values.
		
		In other word, if $u(q,A)\in \Laurentcomplet$ is such that, for an infinite number of $r$ or $N$ in $\N^*$, we have: \[ u(\zeta_{2r}, A) =  \frac{(q^{\alpha})^{rf} \times ADO_r(A,\mathcal{K})}{A_{\mathcal{K}} (A^{2r})} \] or \[  u(q,q^N)= J_N(q,\CK)\]  then, we have the equality: \[ u(q,A)=F_{\infty}(q,A,\CK).\]
		\end{prop}
		
		\begin{proof}
		The map $f: \Q[\alpha][[h]] \to \underset{k \in \N}{\prod} \Q[[h]]$, $x \mapsto (f_k(x))_{k \in \N}$ defined in subsection 4.4 in \cite{willetts2020unification} is injective. In fact, for any infinite subset $J \in \N^*$, $f_J: \Q[\alpha][[h]] \to \underset{k \in J}{\prod} \Q[[h]]$, $x \mapsto (f_k(x))_{k \in J}$ is injective. Thus if, for an infinite number of $N\in \N^*$: \[  u(q,q^N)= J_N(q,\CK)\]  then, \[ u(q,A)=F_{\infty}(q,A,\CK).\]
		
		Moreover, for any $N\in\N^*$, $ADO_r(\zeta_{2r}^N,\CK)= J_N(\zeta_{2r},\CK)$. Hence if, for an infinite number of $r \in \N^*$: \[ u(\zeta_{2r}, A) =  \frac{(q^{\alpha})^{rf} \times ADO_r(A,\mathcal{K})}{A_{\mathcal{K}} (A^{2r})} \] then: \[ u(\zeta_{2r}, \zeta_{2r}^N) =J_N(\zeta_{2r},\CK) \] hence, $J_N(q,\CK)$ being a Laurent polynomial, we get: \[ u(q, q^N)=J_N(q,\CK) \]
		\end{proof}
	
\subsection{Generalization of Alexander polynomials' symmetry}\label{sec_Alex_sym}
In order to prove a symmetry for the ADO invariants, we must change a bit how we use the partial trace. Throughout the paper we have set the first element in the tensor products to be $v_0$. 

In fact we can define the partial trace as :

\[ \Trpt : \End\left(((V^{\alpha})^{\otimes n} \right) \to \End(V^{\alpha})\]
and we have, by definition: \[ \Trpt (f) . v_0 = \Tr_{2, \ldots, n} (f) v_0.\]

Since $V^{\alpha}$ is absolutely simple, the partial trace $\Trpt(f)$ is scalar, allowing us to identify to its value $\Trp (f)$. In other words:

\[\Trpt (f) . w = \Tr_{2, \ldots, n} (f)w  \] for any $w \in V^{\alpha}$.

Combining this fact with Lemma \ref{lemma_quotient_Verma_integer}, one can get a symmetry for the unified invariant.

\begin{thm}[An Alexander-like symmetry for $F_{\infty}$]\label{symmetry_unified_invariant}
Let $\mathcal{K}$ be a $0$ framed knot, \[F_{\infty}(q,q^{\alpha}, \mathcal{K})= F_{\infty}(q, q^{-\alpha-2}, \mathcal{K}). \] In other words, $F_{\infty}$ is not sensitive to $s \mapsto s^{-1} q^{-2}$. 
\end{thm}
\begin{proof}
Using Theorem \ref{prop_unifed_braidrep} at $V^N$, we have the identity \[ \widetilde{\Tr}_{2 \dots n}((1\otimes K^{\otimes n-1})\varphi_n(\beta)) v_0 = F_{\infty}(q, q^{N}, \mathcal{K})v_0 \] and, since the partial trace $\Trpt((1\otimes K^{\otimes n-1})\varphi_n(\beta))$ is scalar, we get: \[ \Trpt((1\otimes K^{\otimes n-1})\varphi_n(\beta)) v_{N+1} = F_{\infty}(q, q^{N}, \mathcal{K})v_{N+1} . \]
But using Lemma \ref{lemma_quotient_Verma_integer}, we also have in $V^N/S_N$:
\[\Trpt((1\otimes K^{\otimes n-1})\varphi_n(\beta)) \overline{v_{N+1}} = F_{\infty}(q, q^{-N-2}, \mathcal{K}) \overline{v_{N+1}}. \]
Thus for all $N \in \N^*$, we have \[F_{\infty}(q, q^{N}, \mathcal{K})=F_{\infty}(q, q^{-N-2}, \mathcal{K}).\]
Using Proposition \ref{prop_unicity_Foo}, we have the equality at formal weight $q^{\alpha}$: \[F_{\infty}(q, q^{\alpha}, \mathcal{K})=F_{\infty}(q, q^{-\alpha-2}, \mathcal{K}).\]
\end{proof}

\begin{cor}[Colored Alexander symmetry]\label{cor_symmetry_ado}
Let $\mathcal{K}$ be a $0$ framed knot, \[ADO_r(A, \mathcal{K})= ADO_r(A^{-1} \zeta_{2r}^{-2}, \mathcal{K}). \]
\end{cor}
\begin{proof}
At $q= \zeta_{2r}$, we have the factorization: \[ F_{\infty}(\zeta_{2r},q^{\alpha}, \mathcal{K}) =  \frac{A^{rf} \times ADO_r(A,\mathcal{K})}{A_{\mathcal{K}} (A^{2r})}\] and since \[A_{\mathcal{K}} (A^{2r})=A_{\mathcal{K}} (A^{-2r}) \] one gets the desired identity.
\end{proof}

In \cite[Sec.~2.2]{costantino2014quantum}, Costantino, Geer and Patureau define an invariant of trivalent graphs denoted $N_r$ constructed also from the theory of $\Uq$ at $q=\zeta_{2r}$. For knots it is the ADO polynomial, more precisely there is a simple change of variable (coming from the fact that they take the variable to be the middle weight instead of here being the highest weight):
\[
N_r\left(q^{\alpha},\CK\right) = \ADO_r\left(q^{\alpha+1-r},\CK\right).
\]

\begin{cor}\label{cor_orientation_Nr}
The $\Uq$ non semi-simple invariant $N_r$ is not sensitive to orientation for knots. 
\end{cor}
\begin{proof}
From \cite[Sec.~2.2,(1)]{costantino2014quantum}, we know that if $\CK^{-1}$ is the knot $\CK$ with reversed orientation, then:
\[
N_r(q^{\alpha},\CK^{-1}) = N_r(q^{-\alpha},\CK).
\]
Re-expressing $N_r$ as ADO and using the colored Alexander symmetry from previous corollary, one deduces directly the invariance under reverse of orientation. 
\end{proof}

\begin{rem}
Authors don't know how much latter invariance generalizes to next objects (e.g. links, graphs) for the non semi-simple invariant $N_r$. 
\end{rem}

\section{Unified invariant from homology of configuration spaces}\label{sec_Foo_fromHomology}

This section first re-defines the tensor products of Verma modules and the action of braid groups upon them by homeomorphisms using homology of configuration spaces of points in punctured disks (Sec. \ref{sec_martel_homological}). This is another point of view independent of quantum groups theory of representation, that was established in \cite{martel2020homological}. This homological interpretation of quantum Verma tensors as Lagrangians of configuration spaces is the key point for Theorem \ref{thm_homol_formula_Foo} expressing $F_{\infty}$ as an intersection pairing between such Lagrangians. Sec. \ref{sec_homological_duality} presents two families of manifolds in configuration spaces defining dual homology classes regarding the Poincaré duality. In Sec. \ref{sec_Foo_thm_fromHomology} we prove Theorem \ref{thm_homol_formula_Foo} and we discuss its consequences. 

\subsection{A homological definition for $\Uq$ Verma modules}\label{sec_martel_homological}


\begin{defn}\label{configspaceofpoints}
Let $r\in \BN$, $n \in \BN$, $D$ be the unit disk, and $\left\lbrace w_1 , \ldots , w_n \rbrace\right. \in D^n$ points chosen on the real line in the interior of $D$. Let $D_n = D \setminus \left\lbrace w_1 , \ldots , w_n \rbrace\right.$ be the unit disk with $n$ punctures. Let:
\[
\Conf_r(D_n):= \left\lbrace (z_1 , \ldots , z_r ) \in (D_n)^r \text{ s.t. } \begin{array}{c} z_i \neq z_j \forall i,j  \end{array} \right\rbrace
\]
be the configuration space of points in the punctured disk $D_n$. 
We define the following space:
\begin{eqnarray}\label{NotConfig}
X_r(w_1 , \ldots , w_n) & := &  \Conf_r(D_n) \Big/ \Sk_r. 
\end{eqnarray}
to be the space of {\em unordered} configurations of $r$ points inside $D_n$, where the permutation group $\Sk_r$ acts by permutation on coordinates.
\end{defn} 

When no confusion arises in what follows, we omit the dependence in $w_1 , \ldots , w_n$ to simplify notations. All the following computations rely on a choice of base point that we fix from now on.

\begin{defn}[Base point]\label{basepoint}
Let ${\pmb \xi^r}= \lbrace  \xi_1 , \ldots , \xi_r \rbrace$ be the base point of $X_r$
chosen so that $ \xi_i \in \partial D_n$ ($\forall i$) as in the following picture:
\begin{equation*}
\begin{tikzpicture}[scale=0.7]
\node (w0) at (-3,0) {};
\node (w1) at (1,0) {};
\node (w2) at (2,0) {};
\node[gray] at (2.8,0) {\ldots};
\node (wn) at (3.4,0) {};


\draw[thick,gray] (4,2) -- (-3,2) -- (-3,-2) -- (4,-2) -- (4,2);


\node[below,red] at (-1,-2) {$\xi_r$};
\node[below,red] at (-0.3,-2) {$\xi_{r-1}$};
\node[below=3pt,red] at (0.3,-2) {\small $\ldots$};
\node[below,red] at (0.8,-2) {$\xi_{1}$};

\node at (-1,-2)[red,circle,fill,inner sep=1pt]{};
\node at (-0.3,-2) [red,circle,fill,inner sep=1pt]{};
\node at (0.8,-2) [red,circle,fill,inner sep=1pt]{};

%

\node[gray] at (w0)[left=5pt] {$w_0$};
\node[gray] at (w1)[above=5pt] {$w_1$};
\node[gray] at (w2)[above=5pt] {$w_2$};
\node[gray] at (wn)[above=5pt] {$w_n$};
\foreach \n in {w1,w2,wn}
  \node at (\n)[gray,circle,fill,inner sep=3pt]{};
\node at (w0)[gray,circle,fill,inner sep=3pt]{};
\end{tikzpicture} .
\end{equation*}
We have illustrated the unit disk (as a square) with the punctures $w_1 , \ldots , w_n$, we have add another point $w_0$ on the boundary that will be used later on, and also the base point just defined. 
\end{defn}



We give a presentation of $\pi_1(X_r, {\pmb \xi^r})$ as a braid subgroup ({\em the mixed braid group}). 

\begin{rem}[{\cite[Remark~2.2]{martel2020homological}}]\label{pi_1X_r}
The group $\pi_1(X_, {\pmb \xi^r})$ is isomorphic to the subgroup of $\CB_{r+n}$ generated by:
\[
\langle \sigma_1 , \ldots , \sigma_{r-1}, B_{r,1} , \ldots , B_{r,n} \rangle 
\]
where the $\sigma_i$ ($i=1,\ldots ,r-1$) are the first standard generators of $\CB_{r+n}$, and $B_{r,k}$ (for $k=1,\ldots ,n$) is the following pure braid:
\[
B_{r,k} = \sigma_{r} \cdots \sigma_{r+k-2} \sigma_{r+k-1}^2 \sigma_{r+k-2}^{-1} \cdots \sigma_{r}^{-1} .
\]
\end{rem}

See \cite[Example~2.3]{martel2020homological} for a picture that illustrates the correspondence between above generators and braids. It will help the reader understanding the following definition of a local system. 


%
%
%


\begin{defn}[Local ring $\Laurent_r$.]\label{localsystXr}
We define the following morphism:
\[
\rho_r : \bfct
\BZ\left[ \pi_1(X_r, {\pmb \xi^r}) \right] & \to & \Laurent := \BZ \left[ s^{\pm1} , t^{\pm 1}  \right]\\
\sigma_i & \mapsto & t \\
B_{r,k} & \mapsto & s^2 . \\
\efct
\]
In what follows we will use the notation $q^{\alpha}:=s$. Using this notation, the morphism becomes:
\[
\rho_r : \bfct
\BZ\left[ \pi_1(X_r, {\pmb \xi^r}) \right] & \to & \Laurent := \BZ \left[ q^{\pm \alpha} , t^{\pm 1}  \right]\\
\sigma_i & \mapsto & t \\
B_{r,k} & \mapsto & q^{2\alpha} . \\
\efct
\]
(We may sometimes omit the dependence in $(w_1,\ldots,w_n)$.)
The data set $(\rho_r,\Laurent)$ will be re-united under the notation $\Laurent_r$ and named local ring of coefficients. 
\end{defn}


\begin{defn}[{\cite[Definition~2.6]{martel2020homological}}]\label{defofH}
Let $r \in \BN$, and let $w_0 = -1$ be the leftmost point in the boundary of $D_n$ (see the picture in Def. \ref{basepoint}), we define the following set:
\[
X_r^-(w_1 , \ldots , w_n) = \left\lbrace \left\lbrace z_1 , \ldots , z_r \right\rbrace \in X_r(w_1 , \ldots , w_n) \text{ s.t. } \exists i, z_i=w_0 \right\rbrace .
\]
We let {\em $\Hlf$} designates the homology of locally finite chains, and we use the following notation for relative homology modules with local coefficients in the ring $\Laurent$:
\[
\Hrelm_r := \Hlf _r \left( X_r, X^{-}_r ; \Laurent_r \right).
\] 
See next remark for precision on such construction. 
\end{defn}

\begin{rem}
We recall how this homology modules are constructed, namely we work with the following homology theories:
\begin{itemize}
\item the {\em locally finite} version of the singular homology, for which we consider locally finite infinite linear combination of singular simplices, (see \cite[Appendix]{martel2020homological}).
\item the homology of the pair $(X_r,X_r^{-})$. 
\item the local ring $\Laurent_r$. Let $\rho_r$ be the morphism from Definition \ref{localsystXr}. This can be seen as the homology associated with the chain complex $ C_{\bullet} (\widehat{X_r}) $ where $\widehat{X_r}$ is the covering naturally associated with the kernel of $\rho_r$ which is naturally endowed with an action of $\Laurent$ by deck transformation as $\rho_r$ is surjective (hence the deck transformation group of $\widehat{X_r}$ is generated by $t$ and $s$). 
\end{itemize}
\end{rem}


We define classes in $\Hrelm_r$. We refer the reader to \cite{martel2020homological} for further details on these constructions.

\begin{defn}[Multi-arc diagrams]
Let $(k_0 , \ldots,  k_{n-1})$ such that $\sum k_i = r$. we define $A'(k_0 , \ldots , k_{n-1})$ to be the following diagram:
\begin{equation*}
A'(k_0 , \ldots , k_{n-1}) = \vcenter{\hbox{\begin{tikzpicture}[scale=0.55, every node/.style={scale=0.8},decoration={
    markings,
    mark=at position 0.5 with {\arrow{>}}}
    ]
\node (w0) at (-5,0) {};
\node (w1) at (-3,0) {};
\node (w2) at (-1,0) {};
\node[gray] at (0.0,0.0) {\ldots};
\node (wn1) at (1,0) {};
\node (wn) at (3,0) {};

\draw[dashed] (w0) -- (w1) node[midway] (k0) {$k_0$};
\draw[dashed] (w0) to[bend right=20] node[near end] (k1) {$k_1$} (w2);
\draw[dashed] (w0) to[bend right=40] node[pos=0.85] (k2) {$k_{n-2}$} (wn1);
\draw[dashed] (w0) to[bend right=60] node[pos=0.85] (k3) {$k_{n-1}$} (wn);

\node[gray] at (w0)[left=5pt] {$w_0$};
\node[gray] at (w1)[above=5pt] {$w_1$};
\node[gray] at (w2)[above=5pt] {$w_2$};
\node[gray] at (wn1)[above=5pt] {$w_{n-1}$};
\node[gray] at (wn)[above=5pt] {$w_n$};
\foreach \n in {w1,w2,wn1,wn}
  \node at (\n)[gray,circle,fill,inner sep=3pt]{};
\node at (w0)[gray,circle,fill,inner sep=3pt]{};

\draw[double,thick,red] (-4,-0.10) -- (-4,-3);
\draw[double,thick,red] (k1) -- (-2.1,-3);
\draw[double,thick,red] (k2) -- (0.05,-3);
\draw[double,thick,red] (k3) -- (2.1,-3);
%
%

\node[gray,circle,fill,inner sep=0.8pt] at (-4.8,-4) {};
\node[below,gray] at (-4.8,-4) {$\xi_r$};
\node[below=5pt,gray] at (-4.2,-4) {$\ldots$};
\node[gray,circle,fill,inner sep=0.8pt] at (-3.5,-4) {};
\node[below,gray] at (-3.5,-4) {$\xi_1$};


\draw[red] (-4.8,-4) -- (-4,-3);
\draw[red] (-3.5,-4) -- (2.1,-3);
\draw[red] (-3.9,-4) -- (0.05,-3);
\draw[red] (-4.2,-4) -- (-2.1,-3);

\draw[gray] (-5,0) -- (-5,2);
\draw[gray] (-5,0) -- (-5,-4);
\draw[gray] (-5,-4) -- (4,-4) -- (4,2) -- (-5,2);

\end{tikzpicture}}} .
\end{equation*}

\end{defn}

\begin{rem}\label{Aprimes}
These above diagrams are denoted $A'$ since there will be slightly different versions for them later on and denoted $A$.
\end{rem}

For $A'(k_0 , \ldots , k_{n-1})$ defined above, let:
\[
\phi_i : I_i \to D_n
\]
be the embedding of the dashed black arc number $i$ indexed by $k_{i-1}$, where $I_i$ is a copy of the unit interval.
Let $\Delta^k$ be the standard (open) $k$ simplex:
\[
\Delta^k = \lbrace 0 < t_1 < \cdots < t_k < 1 \rbrace 
\]
for $k \in \BN$.
For all $i$, we consider the map $\phi^{k_{i-1}}$:
\[
\phi^{k_{i-1}}: \bfct
\Delta^{k_{i-1}} & \to & X_{k_{i-1}} \\
(t_1, \ldots , t_{k_{i-1}} ) & \mapsto & \lbrace \phi_i(t_1) , \ldots, \phi_i(t_{k_{i-1}}) \rbrace
\efct
\]
which is a singular locally finite $(k_{i-1})$-chain and moreover a cycle in $X_{k_{i-1}}$ since locally finite homology of an open ball is one dimensional and concentrated in the ambient dimension (\cite[Appendix]{martel2020homological}). 

To get a class in the homology with $\Laurent$ coefficients, one may choose a lift of the chain to the cover $\widehat{X_r}$ associated with the morphism $\rho_r$. We do so using the red handles of $A'(k_0 , \ldots , k_{n-1})$ (the union of red paths) with which is naturally associated a path:
\[
{\bf h}=\lbrace h_1,\ldots,h_r \rbrace: I \to X_r
\]
joining the base point $\pmb{\xi}$ and (a point in) the $r$-chain assigned to the union of dashed arcs. At the cover level ($\widehat{X_r}$) there is a unique lift $\widehat{{\bf h}}$ of ${\bf h}$ that starts at $\widehat{{\pmb \xi}}$, \textbf{a choice of lift of the base point to $\widehat{X_r}$ that we fix from now on}. The lift $\widehat{A}(k_0,\ldots , k_{n-1})$ of $A(k_0, \ldots , k_{n-1})$ passing by $\widehat{\pmb \xi} (1)$ defines a cycle in $\Crelm_r$, and we still call (by abuse of notation) $A'(k_0 , \ldots , k_{n-1})$ the associated class in $\Hrelm_r$ as we will only use this class out of the original object. 

\begin{defn}[Multi arcs (first version)]\label{AprimesDef}
Following the above construction, we naturally assign a class $A'(k_0 , \ldots , k_{n-1}) \in \Hrelm_r$ with any $n$-tuple such that $\sum k_i = r$. This class is called a {\em multi-arc}. 
\end{defn}

Now we state a proposition that clarifies the structure of the homology as $\Laurent$-modules.

\begin{prop}[Multi-arcs generate the homology, {\cite[Proposition~3.6]{martel2020homological}}]
Let $r \in \BN$, the homology of the pair $(X_r, X_r^{-})$ has the following structure:
\begin{itemize}
\item The module $\Hrelm_r$ is free over $\Laurent$. 
\item The set of multi-arcs:
\[
\lbrace A'(k_0, \ldots, k_{n-1}) \text{ s.t. } \sum k_i = r \rbrace
\]
yields a basis of $\Hrelm_r$. 
\item The module $\Hrelm_r$ is the only non vanishing module of $\Hlf_{\bullet}\left( X_r , X_r^{-}; \Laurent \right) $.
\end{itemize}
\end{prop}


The braid group was earlier defined (Def. \ref{Artinpres}) using its so called Artin presentation. Here we give another definition, relying on topological objects.

\begin{defn}
The braid group on $n$ strands is the mapping class group of $D_n$.
\[
\Bn = \Mod(D_n) = \Homeo^+(D_n, \partial D) \big/ {\Homeo}_0(D_n, \partial D),
\]
namely the group of isotopy classes of homeomorphisms of the unit disk: preserving the orientation, the set of punctures, and being the identity on the boundary. 
\end{defn}

\begin{rem}\label{halfDehntwist}
This definition is isomorphic to the Artin presentation of the braid group (Definition \ref{Artinpres}) by sending generator $\sigma_i$ to the isotopy class of the half Dehn twist swapping punctures $w_i$ and $w_{i+1}$. 
\end{rem}

\begin{lemma}[Lawrence representations]\label{Lawrence_rep}
For all $r,n \in \BN$, the modules $\Hrelm_r$ are endowed with an action of the braid group $\Bn$. 
\end{lemma}
\begin{proof}[Idea of the construction]
It is Lawrence construction of braid groups representations \cite{Law}. See \cite[Lemma~6.33]{martel2020homological} for this precise Lemma. The representations are constructed as follows (sketch of proof).
\begin{itemize}
\item Let $S_i$ be the Dehn twist associated with the standard Artin generator $\sigma_i$ of $\Bn$, for $i \in \lbrace 1, \ldots, n-1 \rbrace$ (see Remark \ref{halfDehntwist}).
\item The homeomorphism $S_i$ extends to $X_r$ coordinate by coordinate.
\item The action of $S_i$ on $X_r$ naturally lifts to $\Hrelm_r$ (it is the heart of \cite[Lemma~6.33]{martel2020homological} and of Lawrence's work). 
\item By defining the action of $\sigma_i$ on $\Hrelm_r$ by that of $S_i$ one obtains a well defined (and multiplicative) action of $\Bn$ on $\Hrelm_r$. It is well defined as braids are homeomorphisms considered up to isotopy while we study their homological action.  
\end{itemize}
\end{proof}

The above representations are often called Lawrence(-like) representations. 
%

We can now recall the main result from \cite{martel2020homological} relating these homological representations with Verma modules representations defined in Section \ref{VermaBraiding}. 

\begin{thm}[{\cite[Theorem~2,3]{martel2020homological}}]\label{ModelMartelHomological2020}
The isomorphism of $\Laurent$-modules:
\[
\bfct
\CH := \bigoplus_{m \in \BN} \Hrelm_m & \to & V_{\alpha}^{\otimes n} = \bigoplus_{m \in \BN} V_{n,m}  \\
A(k_0,\ldots,k_{n-1}) & \mapsto & v_{k_0} \otimes \cdots \otimes v_{k_{n-1}} .
\efct
\]
is $\Bn$ equivariant. In the above isomorphism, the following vectors are involved:
\[
A(k_0,\ldots , k_{n-1}) :=  q^{\alpha \sum_{i=1}^{n-1} i k_i } A'(k_0,\ldots , k_{n-1})
\]
for any $(k_0, \ldots k_{n-1}) \in \BN^n$ (see Definition 6.16 of \cite{martel2020homological}). The identification of rings $\Laurent$ is made by considering $q^{-2} = -t$ (the variable $s$ being the same on both sides).  
\end{thm}

\begin{rem}
\begin{itemize}
\item A diagonal term $q^{-\alpha \left(\sum_{i=1}^{n-1} i k_i \right)}$ normalizes vectors $A$ (\cite[Definition~6.16]{martel2020homological}).
\item The isomorphism from the above theorem also respects the $\UqhL$ action that is defined on Verma modules in Section \ref{VermaBraiding}. In the sense that there is an action of $\UqhL$ defined on $\CH$, see \cite[Theorem~1]{martel2020homological}. 
\end{itemize}
\end{rem}

We will use this isomorphism relating quantum braid representations with homology so for interpreting the partial trace defining $F_{\infty}$ in terms of homological intersections. 

\subsection{Homological duality}\label{sec_homological_duality}

\subsubsection{Multi-arcs: another version}
We recall that for $(k_0, \ldots , k_{n-1})$ such that $\sum k_i = m$, there is a multi-arc $A'(k_0, \ldots , k_{n-1})$ defining a vector in $\Hrelm_m$, and so that the whole family yields a basis. We draw such an element but with a slightly different drawing that better fits with the knot invariant we are seeking.

\begin{defn}[Multi arcs (second version)]
For $(k_0,\ldots,k_{n-1}) \in \BN^n$ s.t. $\sum k_i = r$, we define the following diagram.
\begin{equation*}
A''(k_0 , \ldots , k_{n-1}) := \vcenter{\hbox{\begin{tikzpicture}[scale=0.55, every node/.style={scale=0.8},decoration={
    markings,
    mark=at position 0.5 with {\arrow{>}}}
    ]
    
\draw[gray] (-5,0) -- (-5,6);
\draw[gray] (-5,0) -- (-5,-6);
\draw[gray] (-5,-6) -- (4,-6) -- (4,6) -- (-5,6);

\node (w0) at (-5,0) {};
\node (w1) at (0,4) {};
\node (w2) at (0,1.5) {};
\node[gray] at (0,0) {\vdots};
\node (wn1) at (0,-1.5) {};
\node (wn) at (0,-4) {};

\node[gray] at (w0)[left=5pt] {$w_0$};
\node[gray] at (w1)[right=5pt] {$w_1$};
\node[gray] at (w2)[right=5pt] {$w_2$};
\node[gray] at (wn1)[right=5pt] {$w_{n-1}$};
\node[gray] at (wn)[right=5pt] {$w_n$};
\foreach \n in {w1,w2,wn1,wn}
  \node at (\n)[gray,circle,fill,inner sep=3pt]{};
\node at (w0)[gray,circle,fill,inner sep=3pt]{};


\node[gray,circle,fill,inner sep=0.8pt] (x1) at (4,-3.5) {};
\node[right,gray] at (x1) {$\xi_r$};
\node[gray,circle,fill,inner sep=0.8pt] (x1p) at (4,-3) {};
\node[right,gray] at (x1p) {$\xi_{i_4}$};
\node[gray,circle,fill,inner sep=0.8pt] (x2)  at (4,-1) {};
\node[right,gray] at (x2) {$\xi_{j_3}$};
\node[gray,circle,fill,inner sep=0.8pt] (x2p)  at (4,-0.5) {};
\node[right,gray] at (x2p) {$\xi_{i_3}$};
\node[below=5pt,gray] at (4,-5) {$\ldots$};
\node[gray,circle,fill,inner sep=0.8pt] (x3) at (4,2) {};
\node[right,gray] at (x3) {$\xi_{j_2}$};
\node[gray,circle,fill,inner sep=0.8pt] (x3p) at (4,2.5) {};
\node[right,gray] at (x3p) {$\xi_{i_2}$};
\node[gray,circle,fill,inner sep=0.8pt] (x4) at (4,4.5) {};
\node[right,gray] at (x4) {$\xi_{k_0}$};
\node[gray,circle,fill,inner sep=0.8pt] (x4p) at (4,5) {};
\node[right,gray] at (x4p) {$\xi_1$};

\draw[dashed] (w0) to[bend left=10]  node[pos=0.4]  {$k_0$} node[pos=0.4] (k0) {} (w1);
\draw[dashed] (w0) to node[pos=0.4] (k1) {} node[pos=0.4] {$k_1$} (w2);
\draw[dashed] (w0) to node[pos=0.6] (k2) {} node[pos=0.4] {$k_{n-2}$} (wn1);
\draw[dashed] (w0) to[bend right=10] node[pos=0.8] (k3) {} node[pos=0.3] {$k_{n-1}$} (wn);

\coordinate (b) at (-3,-3);
\coordinate (bp) at (-2.5,-2.5);
\coordinate (bpp) at (-1.5,-1.5);

\coordinate (t4) at (3.5,4.75);
\coordinate (t3) at (3.5,2.25);
\coordinate (t2) at (3.5,-0.75);
\coordinate (t1) at (3.5,-3.25);

\draw[red] (x4p)--(t4);
\draw[red] (x4)--(t4);
\draw[red] (x3p)--(t3);
\draw[red] (x3)--(t3);
\draw[red] (x2p)--(t2);
\draw[red] (x2)--(t2);
\draw[red] (x1p)--(t1);
\draw[red] (x1)--(t1);

\draw[double,red] (t1)--(k3);
\draw[double,red] (t2)--(t2-|bpp)--(k2);
\draw[double,red] (t3)--(t3-|bpp)--(k1);
\draw[double,red] (t4)--(t4-|bp)--(k0);

\end{tikzpicture}}} .
\end{equation*}
As diagrams from Definition \ref{Aprimes} naturally defines classes in $\Hrelm_r$ (see Definition \ref{AprimesDef}, natural process explained above it), same natural process associates classes in $\Hrelm_r$ with the above $A''(k_0,\ldots,k_{n-1})$. We use latter notation to designate the homology class also.
\end{defn}

We have three families of diagrams corresponding to homology classes. They are related diagonally as follows. 

\begin{prop}\label{rel_A_A'_A''}
In $\Hrelm_r$, the following relations hold.
\begin{align}
A(k_0,\ldots, k_{n-1}) & = q^{\sum_{i=1}^{n-1} i k_i \alpha }A'(k_0,\ldots,k_{n-1}), \\
A'(k_0,\ldots, k_{n-1}) & = (-t)^{\frac{r(r-1)}{2}} q^{2 \alpha \sum_{i=0}^{n-1} (n-i) k_i } A''(k_0,\ldots,k_{n-1})
\end{align}
for all $(k_0,\ldots,k_{n-1})$ such that $\sum k_i = r$. Finally:
\begin{align}
A(k_0,\ldots, k_{n-1}) & = (-t)^{\frac{r(r-1)}{2}} q^{ \alpha  2nr} q^{ - \alpha \sum_{i=0}^{n-1} i k_i }A''(k_0,\ldots,k_{n-1})
\end{align}
\end{prop}
\begin{proof}
The first equality of the proposition was already considered in \cite{martel2020homological} and was recalled in Theorem \ref{ModelMartelHomological2020}. The second one follows from the following equalities:
\begin{align*}
\vcenter{\hbox{\begin{tikzpicture}[scale=0.55, every node/.style={scale=0.8},decoration={
    markings,
    mark=at position 0.5 with {\arrow{>}}}
    ]
\node (w0) at (-5,0) {};
\node (w1) at (-3,0) {};
\node (w2) at (-1,0) {};
\node[gray] at (0.0,0.0) {\ldots};
\node (wn1) at (1,0) {};
\node (wn) at (3,0) {};
\draw[dashed] (w0) -- (w1) node[midway] (k0) {$k_0$};
\draw[dashed] (w0) to[bend right=20] node[near end] (k1) {$k_1$} (w2);
\draw[dashed] (w0) to[bend right=40] node[pos=0.85] (k2) {$k_{n-2}$} (wn1);
\draw[dashed] (w0) to[bend right=60] node[pos=0.85] (k3) {$k_{n-1}$} (wn);
\node[gray] at (w0)[left=5pt] {$w_0$};
\node[gray] at (w1)[above=5pt] {$w_1$};
\node[gray] at (w2)[above=5pt] {$w_2$};
\node[gray] at (wn1)[above=5pt] {$w_{n-1}$};
\node[gray] at (wn)[above=5pt] {$w_n$};
\foreach \n in {w1,w2,wn1,wn}
  \node at (\n)[gray,circle,fill,inner sep=3pt]{};
\node at (w0)[gray,circle,fill,inner sep=3pt]{};
\draw[double,thick,red] (-4,-0.10) -- (-4,-3);
\draw[double,thick,red] (k1) -- (-2.1,-3);
\draw[double,thick,red] (k2) -- (0.05,-3);
\draw[double,thick,red] (k3) -- (2.1,-3);
%
%
\node[gray,circle,fill,inner sep=0.8pt] at (-4.8,-4) {};
\node[below,gray] at (-4.8,-4) {$\xi_r$};
\node[below=5pt,gray] at (-4.2,-4) {$\ldots$};
\node[gray,circle,fill,inner sep=0.8pt] at (-3.5,-4) {};
\node[below,gray] at (-3.5,-4) {$\xi_1$};
\draw[red] (-4.8,-4) -- (-4,-3);
\draw[red] (-3.5,-4) -- (2.1,-3);
\draw[red] (-3.9,-4) -- (0.05,-3);
\draw[red] (-4.2,-4) -- (-2.1,-3);
\draw[gray] (-5,0) -- (-5,2);
\draw[gray] (-5,0) -- (-5,-4);
\draw[gray] (-5,-4) -- (4,-4) -- (4,2) -- (-5,2);
\end{tikzpicture}}}
& = \vcenter{\hbox{\begin{tikzpicture}[scale=0.4, every node/.style={scale=0.6},decoration={
    markings,
    mark=at position 0.5 with {\arrow{>}}}
    ]  
\draw[gray] (-5,0) -- (-5,6);
\draw[gray] (-5,0) -- (-5,-6.5);
\draw[gray] (-5,-6.5) -- (4,-6.5) -- (4,6) -- (-5,6);
\node (w0) at (-5,0) {};
\node (w1) at (0,4) {};
\node (w2) at (0,1.5) {};
\node[gray] at (0,0) {\vdots};
\node (wn1) at (0,-1.5) {};
\node (wn) at (0,-4) {};
\node[gray] at (w0)[left=5pt] {$w_0$};
\node[gray] at (w1)[above=5pt] {$w_1$};
\node[gray] at (w2)[above=5pt] {$w_2$};
\node[gray] at (wn1)[above=5pt] {$w_{n-1}$};
\node[gray] at (wn)[above=5pt] {$w_n$};
\foreach \n in {w1,w2,wn1,wn}
  \node at (\n)[gray,circle,fill,inner sep=3pt]{};
\node at (w0)[gray,circle,fill,inner sep=3pt]{};
\node[gray,circle,fill,inner sep=0.8pt] (x1) at (4,-3.5) {};
\node[right,gray] at (x1) {$\xi_r$};
\node[gray,circle,fill,inner sep=0.8pt] (x1p) at (4,-3) {};
\node[right,gray] at (x1p) {$\xi_{i_4}$};
\node[gray,circle,fill,inner sep=0.8pt] (x2)  at (4,-1) {};
\node[right,gray] at (x2) {$\xi_{j_3}$};
\node[gray,circle,fill,inner sep=0.8pt] (x2p)  at (4,-0.5) {};
\node[right,gray] at (x2p) {$\xi_{i_3}$};
\node[below=5pt,gray] at (4,-5) {$\ldots$};
\node[gray,circle,fill,inner sep=0.8pt] (x3) at (4,2) {};
\node[right,gray] at (x3) {$\xi_{j_2}$};
\node[gray,circle,fill,inner sep=0.8pt] (x3p) at (4,2.5) {};
\node[right,gray] at (x3p) {$\xi_{i_2}$};
\node[gray,circle,fill,inner sep=0.8pt] (x4) at (4,4.5) {};
\node[right,gray] at (x4) {$\xi_{k_0}$};
\node[gray,circle,fill,inner sep=0.8pt] (x4p) at (4,5) {};
\node[right,gray] at (x4p) {$\xi_1$};
\draw[dashed] (w0) to[bend left=10]  node[above,pos=0.2]  {$k_0$} node[pos=0.2] (k0) {} (w1);
\draw[dashed] (w0) to node[pos=0.4] (k1) {} node[above,pos=0.4] {$k_1$} (w2);
\draw[dashed] (w0) to node[pos=0.6] (k2) {} node[above,pos=0.6] {$k_{n-2}$} (wn1);
\draw[dashed] (w0) to[bend right=10] node[pos=0.8] (k3) {} node[above,pos=0.8] {$k_{n-1}$} (wn);
\coordinate (t4) at (3.5,4.75);
\coordinate (t3) at (3.5,2.25);
\coordinate (t2) at (3.5,-0.75);
\coordinate (t1) at (3.5,-3.25);
\draw[red] (x4p)--(t4);
\draw[red] (x4)--(t4);
\draw[red] (x3p)--(t3);
\draw[red] (x3)--(t3);
\draw[red] (x2p)--(t2);
\draw[red] (x2)--(t2);
\draw[red] (x1p)--(t1);
\draw[red] (x1)--(t1);
\draw[double,red] (t1)--(2.5,-3.25)--(2.5,-6)--(-4,-6)--(k0);
\draw[double,red] (t2)--(2,-0.75)--(2,-5.5)--(-3,-5.5)--(k1);
\draw[double,red] (t3)--(1.5,2.25)--(1.5,-5)--(-2,-5)--(k2);
\draw[double,red] (t4)--(1,4.75)--(1,-4.5)--(-1.25,-4.5)--(k3);
\end{tikzpicture}}} \\
& = (-t)^{\frac{r(r-1)}{2}} q^{2 \alpha \sum_{i=0}^{n-1} (n-i) k_i } \vcenter{\hbox{\begin{tikzpicture}[scale=0.4, every node/.style={scale=0.6},decoration={
    markings,
    mark=at position 0.5 with {\arrow{>}}}
    ]    
\draw[gray] (-5,0) -- (-5,6);
\draw[gray] (-5,0) -- (-5,-6);
\draw[gray] (-5,-6) -- (4,-6) -- (4,6) -- (-5,6);
\node (w0) at (-5,0) {};
\node (w1) at (0,4) {};
\node (w2) at (0,1.5) {};
\node[gray] at (0,0) {\vdots};
\node (wn1) at (0,-1.5) {};
\node (wn) at (0,-4) {};
\node[gray] at (w0)[left=5pt] {$w_0$};
\node[gray] at (w1)[right=5pt] {$w_1$};
\node[gray] at (w2)[right=5pt] {$w_2$};
\node[gray] at (wn1)[right=5pt] {$w_{n-1}$};
\node[gray] at (wn)[right=5pt] {$w_n$};
\foreach \n in {w1,w2,wn1,wn}
  \node at (\n)[gray,circle,fill,inner sep=3pt]{};
\node at (w0)[gray,circle,fill,inner sep=3pt]{};
\node[gray,circle,fill,inner sep=0.8pt] (x1) at (4,-3.5) {};
\node[right,gray] at (x1) {$\xi_r$};
\node[gray,circle,fill,inner sep=0.8pt] (x1p) at (4,-3) {};
\node[right,gray] at (x1p) {$\xi_{i_4}$};
\node[gray,circle,fill,inner sep=0.8pt] (x2)  at (4,-1) {};
\node[right,gray] at (x2) {$\xi_{j_3}$};
\node[gray,circle,fill,inner sep=0.8pt] (x2p)  at (4,-0.5) {};
\node[right,gray] at (x2p) {$\xi_{i_3}$};
\node[below=5pt,gray] at (4,-5) {$\ldots$};
\node[gray,circle,fill,inner sep=0.8pt] (x3) at (4,2) {};
\node[right,gray] at (x3) {$\xi_{j_2}$};
\node[gray,circle,fill,inner sep=0.8pt] (x3p) at (4,2.5) {};
\node[right,gray] at (x3p) {$\xi_{i_2}$};
\node[gray,circle,fill,inner sep=0.8pt] (x4) at (4,4.5) {};
\node[right,gray] at (x4) {$\xi_{k_0}$};
\node[gray,circle,fill,inner sep=0.8pt] (x4p) at (4,5) {};
\node[right,gray] at (x4p) {$\xi_1$};
\draw[dashed] (w0) to[bend left=10]  node[pos=0.4]  {$k_0$} node[pos=0.4] (k0) {} (w1);
\draw[dashed] (w0) to node[pos=0.4] (k1) {} node[pos=0.4] {$k_1$} (w2);
\draw[dashed] (w0) to node[pos=0.6] (k2) {} node[pos=0.4] {$k_{n-2}$} (wn1);
\draw[dashed] (w0) to[bend right=10] node[pos=0.8] (k3) {} node[pos=0.3] {$k_{n-1}$} (wn);
\coordinate (b) at (-3,-3);
\coordinate (bp) at (-2.5,-2.5);
\coordinate (bpp) at (-1.5,-1.5);
\coordinate (t4) at (3.5,4.75);
\coordinate (t3) at (3.5,2.25);
\coordinate (t2) at (3.5,-0.75);
\coordinate (t1) at (3.5,-3.25);
\draw[red] (x4p)--(t4);
\draw[red] (x4)--(t4);
\draw[red] (x3p)--(t3);
\draw[red] (x3)--(t3);
\draw[red] (x2p)--(t2);
\draw[red] (x2)--(t2);
\draw[red] (x1p)--(t1);
\draw[red] (x1)--(t1);
\draw[double,red] (t1)--(k3);
\draw[double,red] (t2)--(t2-|bpp)--(k2);
\draw[double,red] (t3)--(t3-|bpp)--(k1);
\draw[double,red] (t4)--(t4-|bp)--(k0);
\end{tikzpicture}}} .
\end{align*}
The first equality comes from an isotopy of the disc, the second one comes from the application of the \textit{handle rule} (\cite[Remark~4.1]{martel2020homological}, see details in following Remark \ref{HandleruleRecall}). Then one recognizes leftmost diagram to be $A(k_0, \ldots, k_{n-1})$ and last one to be $A''(k_0,\ldots , k_{n-1})$. 
Finally:
\begin{align*}
A(k_0,\ldots, k_{n-1}) & = (-t)^{\frac{r(r-1)}{2}} q^{ \alpha \sum_{i=0}^{n-1} (2n-i) k_i }A''(k_0,\ldots,k_{n-1}) \\
& = (-t)^{\frac{r(r-1)}{2}} q^{ \alpha \left( 2nr - \sum_{i=0}^{n-1} i k_i \right) }A''(k_0,\ldots,k_{n-1})
\end{align*}
provides last relation. 
\end{proof}

\begin{rem}[Handle rule]\label{HandleruleRecall}
We give more details on the handle rule applied once in the proof of the previous proposition. The handle rule (\cite[Remark~4.1]{martel2020homological}) states:
\begin{align*}
\vcenter{\hbox{\begin{tikzpicture}[scale=0.4, every node/.style={scale=0.6},decoration={
    markings,
    mark=at position 0.5 with {\arrow{>}}}
    ]  
\draw[gray] (-5,0) -- (-5,6);
\draw[gray] (-5,0) -- (-5,-6.5);
\draw[gray] (-5,-6.5) -- (4,-6.5) -- (4,6) -- (-5,6);
\node (w0) at (-5,0) {};
\node (w1) at (0,4) {};
\node (w2) at (0,1.5) {};
\node[gray] at (0,0) {\vdots};
\node (wn1) at (0,-1.5) {};
\node (wn) at (0,-4) {};
\node[gray] at (w0)[left=5pt] {$w_0$};
\node[gray] at (w1)[above=5pt] {$w_1$};
\node[gray] at (w2)[above=5pt] {$w_2$};
\node[gray] at (wn1)[above=5pt] {$w_{n-1}$};
\node[gray] at (wn)[above=5pt] {$w_n$};
\foreach \n in {w1,w2,wn1,wn}
  \node at (\n)[gray,circle,fill,inner sep=3pt]{};
\node at (w0)[gray,circle,fill,inner sep=3pt]{};
\node[gray,circle,fill,inner sep=0.8pt] (x1) at (4,-3.5) {};
\node[right,gray] at (x1) {$\xi_r$};
\node[gray,circle,fill,inner sep=0.8pt] (x1p) at (4,-3) {};
\node[gray,circle,fill,inner sep=0.8pt] (x2)  at (4,-1) {};
\node[gray,circle,fill,inner sep=0.8pt] (x2p)  at (4,-0.5) {};
\node[below=5pt,gray] at (4,-5) {$\ldots$};
\node[gray,circle,fill,inner sep=0.8pt] (x3) at (4,2) {};
\node[gray,circle,fill,inner sep=0.8pt] (x3p) at (4,2.5) {};
\node[gray,circle,fill,inner sep=0.8pt] (x4) at (4,4.5) {};
\node[gray,circle,fill,inner sep=0.8pt] (x4p) at (4,5) {};
\node[right,gray] at (x4p) {$\xi_1$};
\draw[dashed] (w0) to[bend left=10]  node[above,pos=0.2]  {$k_0$} node[pos=0.2] (k0) {} (w1);
\draw[dashed] (w0) to node[pos=0.4] (k1) {} node[above,pos=0.4] {$k_1$} (w2);
\draw[dashed] (w0) to node[pos=0.6] (k2) {} node[above,pos=0.6] {$k_{n-2}$} (wn1);
\draw[dashed] (w0) to[bend right=10] node[pos=0.8] (k3) {} node[above,pos=0.8] {$k_{n-1}$} (wn);
\coordinate (t4) at (3.5,4.75);
\coordinate (t3) at (3.5,2.25);
\coordinate (t2) at (3.5,-0.75);
\coordinate (t1) at (3.5,-3.25);
\draw[red] (x4p)--(t4);
\draw[red] (x4)--(t4);
\draw[red] (x3p)--(t3);
\draw[red] (x3)--(t3);
\draw[red] (x2p)--(t2);
\draw[red] (x2)--(t2);
\draw[red] (x1p)--(t1);
\draw[red] (x1)--(t1);
\draw[double,red] (t1)--(2.5,-3.25)--(2.5,-6)--(-4,-6)--(k0);
\draw[double,red] (t2)--(2,-0.75)--(2,-5.5)--(-3,-5.5)--(k1);
\draw[double,red] (t3)--(1.5,2.25)--(1.5,-5)--(-2,-5)--(k2);
\draw[double,red] (t4)--(1,4.75)--(1,-4.5)--(-1.25,-4.5)--(k3);
\end{tikzpicture}}} 
& = \rho_r(\alpha \beta^{-1})_{|t=-t} \vcenter{\hbox{\begin{tikzpicture}[scale=0.4, every node/.style={scale=0.6},decoration={
    markings,
    mark=at position 0.5 with {\arrow{>}}}
    ]    
\draw[gray] (-5,0) -- (-5,6);
\draw[gray] (-5,0) -- (-5,-6);
\draw[gray] (-5,-6) -- (4,-6) -- (4,6) -- (-5,6);
\node (w0) at (-5,0) {};
\node (w1) at (0,4) {};
\node (w2) at (0,1.5) {};
\node[gray] at (0,0) {\vdots};
\node (wn1) at (0,-1.5) {};
\node (wn) at (0,-4) {};
\node[gray] at (w0)[left=5pt] {$w_0$};
\node[gray] at (w1)[right=5pt] {$w_1$};
\node[gray] at (w2)[right=5pt] {$w_2$};
\node[gray] at (wn1)[right=5pt] {$w_{n-1}$};
\node[gray] at (wn)[right=5pt] {$w_n$};
\foreach \n in {w1,w2,wn1,wn}
  \node at (\n)[gray,circle,fill,inner sep=3pt]{};
\node at (w0)[gray,circle,fill,inner sep=3pt]{};
\node[gray,circle,fill,inner sep=0.8pt] (x1) at (4,-3.5) {};
\node[right,gray] at (x1) {$\xi_r$};
\node[gray,circle,fill,inner sep=0.8pt] (x1p) at (4,-3) {};
\node[right,gray] at (x1p) {$\xi_{i_4}$};
\node[gray,circle,fill,inner sep=0.8pt] (x2)  at (4,-1) {};
\node[right,gray] at (x2) {$\xi_{j_3}$};
\node[gray,circle,fill,inner sep=0.8pt] (x2p)  at (4,-0.5) {};
\node[right,gray] at (x2p) {$\xi_{i_3}$};
\node[below=5pt,gray] at (4,-5) {$\ldots$};
\node[gray,circle,fill,inner sep=0.8pt] (x3) at (4,2) {};
\node[right,gray] at (x3) {$\xi_{j_2}$};
\node[gray,circle,fill,inner sep=0.8pt] (x3p) at (4,2.5) {};
\node[right,gray] at (x3p) {$\xi_{i_2}$};
\node[gray,circle,fill,inner sep=0.8pt] (x4) at (4,4.5) {};
\node[right,gray] at (x4) {$\xi_{k_0}$};
\node[gray,circle,fill,inner sep=0.8pt] (x4p) at (4,5) {};
\node[right,gray] at (x4p) {$\xi_1$};
\draw[dashed] (w0) to[bend left=10]  node[pos=0.4]  {$k_0$} node[pos=0.4] (k0) {} (w1);
\draw[dashed] (w0) to node[pos=0.4] (k1) {} node[pos=0.4] {$k_1$} (w2);
\draw[dashed] (w0) to node[pos=0.6] (k2) {} node[pos=0.4] {$k_{n-2}$} (wn1);
\draw[dashed] (w0) to[bend right=10] node[pos=0.8] (k3) {} node[pos=0.3] {$k_{n-1}$} (wn);
\coordinate (b) at (-3,-3);
\coordinate (bp) at (-2.5,-2.5);
\coordinate (bpp) at (-1.5,-1.5);
\coordinate (t4) at (3.5,4.75);
\coordinate (t3) at (3.5,2.25);
\coordinate (t2) at (3.5,-0.75);
\coordinate (t1) at (3.5,-3.25);
\draw[red] (x4p)--(t4);
\draw[red] (x4)--(t4);
\draw[red] (x3p)--(t3);
\draw[red] (x3)--(t3);
\draw[red] (x2p)--(t2);
\draw[red] (x2)--(t2);
\draw[red] (x1p)--(t1);
\draw[red] (x1)--(t1);
\draw[double,red] (t1)--(k3);
\draw[double,red] (t2)--(t2-|bpp)--(k2);
\draw[double,red] (t3)--(t3-|bpp)--(k1);
\draw[double,red] (t4)--(t4-|bp)--(k0);
\end{tikzpicture}}} .
\end{align*}
where $\alpha$ is the path in $X_r$ corresponding to the (red)-handle on the left and $\beta$ to that on the right, and $\rho_r$ the representation of $\pi_1(X_r)$ recalled in Definition \ref{localsystXr}. We evaluate $\rho_r$ at $t=-t$, see \cite[Rem.~4.2]{martel2020homological}, because in diagrams the permutation of the red strands implies a permutation of embeddings of configurations. Hence the homology class must be multiplied by the sign of the permutation (i.e. the power of $t$ in $\rho_r(\alpha \beta^{-1})$) corresponding to the induced change of orientation. 

In the present case the path $\alpha \beta^{-1}$ is drawn below.
\[
\vcenter{\hbox{
\begin{tikzpicture}[scale=0.65, every node/.style={scale=0.7}]

\coordinate (w1b) at (-4,-4);
\coordinate (w2b) at (-2,-4);
\coordinate (w3b) at (2,-4);
\coordinate (w4b) at (4,-4);

\node[gray] at (0,0) {$\cdots$};

\coordinate (w1h) at (-4,3);
\coordinate (w2h) at (-2,3);
\coordinate (w3h) at (2,3);
\coordinate (w4h) at (4,3);

\node[above,gray] at (w1h) {$w_n$};
\node[above,gray] at (w2h) {$w_{n-1}$};
\node[gray] at (0,3) {$\cdots$};
\node[above,gray] at (w3h) {$w_2$};
\node[above,gray] at (w4h) {$w_1$};

\draw[gray,thick] (w1b)--(w1h);
\draw[gray,thick] (w2b)--(w2h);
\draw[gray,thick] (w3b)--(w3h);
\draw[gray,thick] (w4b)--(w4h);

\coordinate (x1d) at (-3,-5);
\coordinate (x2d) at (-1,-5);
\coordinate (x3d) at (3,-5);
\coordinate (x4d) at (5,-5);

\node[red] at (1,-5.5) {$\ldots$};

\coordinate (s1g) at (-3.25,-5.5);
\coordinate (s2g) at (-1.25,-5.5);
\coordinate (s3g) at (2.75,-5.5);
\coordinate (s4g) at (4.75,-5.5);

\coordinate (s1d) at (-2.75,-5.5);
\coordinate (s2d) at (-0.75,-5.5);
\coordinate (s3d) at (3.25,-5.5);
\coordinate (s4d) at (5.25,-5.5);

\coordinate (x1a) at (-3,2);
\coordinate (x2a) at (-1,2);
\coordinate (x3a) at (3,2);
\coordinate (x4a) at (5,2);

\coordinate (t1g) at (-3.25,2.5);
\coordinate (t2g) at (-1.25,2.5);
\coordinate (t3g) at (2.75,2.5);
\coordinate (t4g) at (4.75,2.5);

\coordinate (t1d) at (-2.75,2.5);
\coordinate (t2d) at (-0.75,2.5);
\coordinate (t3d) at (3.25,2.5);
\coordinate (t4d) at (5.25,2.5);


\node[white,circle,fill,inner sep=2.9pt] at (-4,-4) {};
\node[white,circle,fill,inner sep=2.9pt] at (-4,-3.5) {};
\node[white,circle,fill,inner sep=2.9pt] at (-4,-3) {};

\node[white,circle,fill,inner sep=2.9pt] at (-2,-4) {};
\node[white,circle,fill,inner sep=2.9pt] at (-2,-3.5) {};
\node[white,circle,fill,inner sep=2.9pt] at (-2,-3) {};

\node[white,circle,fill,inner sep=2.9pt] at (2,-3.5) {};
\node[white,circle,fill,inner sep=2.9pt] at (2,-3) {};

\node[white,circle,fill,inner sep=2.9pt] at (4,-3) {};

\draw[double,red] (x1d)--(-3,-4.5)--(-6,-4.5)--(-6,1)--(-4.2,1);
\draw[double,red] (-3.8,1)--(-2.2,1);
\draw[double,red] (-1.8,1)--(1.8,1);
\draw[double,red] (2.2,1)--(3.8,1);
\draw[double,red] (4.2,1)--(5,1)--(x4a);

\node[white,circle,fill,inner sep=2.5pt] at (3,1) {};
\draw[double,red] (x2d)--(-1,-4)--(-5.5,-4)--(-5.5,0)--(-4.2,0);
\draw[double,red] (-3.8,0)--(-2.2,0);
\draw[double,red] (-1.8,0)--(1.8,0);
\draw[double,red] (2.2,0)--(3,0)--(x3a);

\node[white,circle,fill,inner sep=2.5pt] at (-1,0) {};
\node[white,circle,fill,inner sep=2.5pt] at (-1,1) {};
\draw[double,red] (x3d)--(3,-3.5)--(-5,-3.5)--(-5,-1)--(-4.2,-1);
\draw[double,red] (-3.8,-1)--(-2.2,-1);
\draw[double,red] (-1.8,-1)--(-1,-1)--(x2a);

\node[white,circle,fill,inner sep=2.5pt] at (-3,-1) {};
\node[white,circle,fill,inner sep=2.5pt] at (-3,0) {};
\node[white,circle,fill,inner sep=2.5pt] at (-3,1) {};
\draw[double,red] (x4d)--(5,-3)--(-4.5,-3)--(-4.5,-2)--(-4.2,-2);
\draw[double,red] (-3.8,-2)--(-3,-2)--(x1a);


\node[red,draw,rectangle,fill=white] at (-3,-2) {$\Delta_{k_{n-1}}$};
\node[red,draw,rectangle,fill=white] at (-1,-1) {$\Delta_{k_{n-2}}$};
\node[red,draw,rectangle,fill=white] at (3,0) {$\Delta_{k_{1}}$};
\node[red,draw,rectangle,fill=white] at (5,1) {$\Delta_{k_{0}}$};

\draw[red] (x1a)--(t1g);
\draw[red] (x2a)--(t2g);
\draw[red] (x3a)--(t3g);
\draw[red] (x4a)--(t4g);

\draw[red] (x1a)--(t1d);
\draw[red] (x2a)--(t2d);
\draw[red] (x3a)--(t3d);
\draw[red] (x4a)--(t4d);

\draw[red] (x1d)--(s1g);
\draw[red] (x2d)--(s2g);
\draw[red] (x3d)--(s3g);
\draw[red] (x4d)--(s4g);

\draw[red] (x1d)--(s1d);
\draw[red] (x2d)--(s2d);
\draw[red] (x3d)--(s3d);
\draw[red] (x4d)--(s4d);

\node[red,circle,fill,inner sep=0.8pt] at (t1d) {};
\node[red,circle,fill,inner sep=0.8pt] at (t2d) {};
\node[red,circle,fill,inner sep=0.8pt] at (t3d) {};
\node[red,circle,fill,inner sep=0.8pt] at (t4d) {};

\node[red,circle,fill,inner sep=0.8pt] at (t1g) {};
\node[red,circle,fill,inner sep=0.8pt] at (t2g) {};
\node[red,circle,fill,inner sep=0.8pt] at (t3g) {};
\node[red,circle,fill,inner sep=0.8pt] at (t4g) {};

\node[red,circle,fill,inner sep=0.8pt] at (s1d) {};
\node[red,circle,fill,inner sep=0.8pt] at (s2d) {};
\node[red,circle,fill,inner sep=0.8pt] at (s3d) {};
\node[red,circle,fill,inner sep=0.8pt] at (s4d) {};

\node[red,circle,fill,inner sep=0.8pt] at (s1g) {};
\node[red,circle,fill,inner sep=0.8pt] at (s2g) {};
\node[red,circle,fill,inner sep=0.8pt] at (s3g) {};
\node[red,circle,fill,inner sep=0.8pt] at (s4g) {};

\node[above,red] at (t4d) {$\xi_1$};
\node[above,red] at (t1g) {$\xi_r$};

\node[below,red] at (s4d) {$\xi_1$};
\node[below,red] at (s1g) {$\xi_r$};

\end{tikzpicture}
}}
\]
Every little red tube means parallel red paths not crossing with each other. Crossings involving these tubes are materialized using $\Delta$ boxes inside which the following happens:
\begin{equation*}
\vcenter{\hbox{
\begin{tikzpicture}
\draw[red] (-1,-0.5)--(1,-0.5)--(1,0.5)--(-1,0.5)--(-1,-0.5);
\node[red] at (0,0) {$\Delta_k$};
\end{tikzpicture} }} = 
{\color{red} k \lbrace} \vcenter{\hbox{
\begin{tikzpicture}
\draw[red] (-1,-0.5)--(1,-0.5)--(1,0.5)--(-1,0.5)--(-1,-0.5);


\coordinate (gh) at (-1,0.25);
\coordinate (gb) at (-1,-0.25);


\coordinate (hg) at (-0.5,0.5);
\coordinate (hd) at (0.5,0.5);

\node[red] at (0.15,0.35) {$\cdots$}; 


\draw[red] (gh)--(0.5,0.25)--(hd);

\node[white,circle,fill, inner sep=1pt] at (-0.25,0.25) {};
\draw[red] (-1,-0.1)--(-0.25,-0.1)--(-0.25,0.5);

\node[white,circle,fill, inner sep=1pt] at (-0.5,-0.1) {};
\node[white,circle,fill, inner sep=1pt] at (-0.5,0.25) {};
\draw[red] (gb)--(-0.5,-0.25)--(hg);

\end{tikzpicture} }}
\end{equation*}
Then $\rho_r (\alpha \beta^{-1})_{|t=-t} = (-t)^a q^{\alpha b}$ such that $a$ is the sum (with signs) of red-red corssings, and $b$ is the total winding number of red strands around gray ones. The reader should pay attention to the fact that in \cite{martel2020homological} braids a read from top to bottom as in the present work we do the converse. 
\end{rem}

\subsubsection{Barcodes}

We now define homology classes in $\Hnot_r \left(X_r, \partial X_r \setminus X_r^-; \Laurent_r \right)$ that are usually called {\em barcodes}.

\begin{defn}[Barcodes]
We fix notation for the following diagrams:
\begin{equation*}
B''(k_0 , \ldots , k_{n-1}) := \vcenter{\hbox{\begin{tikzpicture}[scale=0.55, every node/.style={scale=0.8},decoration={
    markings,
    mark=at position 0.5 with {\arrow{>}}}
    ]
    
\draw[gray] (-5,0) -- (-5,6);
\draw[gray] (-5,0) -- (-5,-6);
\draw[gray] (-5,-6) -- (4,-6) -- (4,6) -- (-5,6);

\node (w0) at (-5,0) {};
\node (w1) at (0,4) {};
\node (w2) at (0,1.5) {};
\node[gray] at (0,0) {\vdots};
\node (wn1) at (0,-1.5) {};
\node (wn) at (0,-4) {};

\node[gray] at (w0)[left=5pt] {$w_0$};
\node[gray] at (w1)[right=5pt] {$w_1$};
\node[gray] at (w2)[right=5pt] {$w_2$};
\node[gray] at (wn1)[right=5pt] {$w_{n-1}$};
\node[gray] at (wn)[right=5pt] {$w_n$};
\foreach \n in {w1,w2,wn1,wn}
  \node at (\n)[gray,circle,fill,inner sep=3pt]{};
\node at (w0)[gray,circle,fill,inner sep=3pt]{};


\node[gray,circle,fill,inner sep=0.8pt] (x1) at (4,-3.5) {};
\node[right,gray] at (x1) {$\xi_r$};
\node[gray,circle,fill,inner sep=0.8pt] (x1p) at (4,-3) {};
\node[right,gray] at (x1p) {$\xi_{i_4}$};
\node[gray,circle,fill,inner sep=0.8pt] (x2)  at (4,-1) {};
\node[right,gray] at (x2) {$\xi_{j_3}$};
\node[gray,circle,fill,inner sep=0.8pt] (x2p)  at (4,-0.5) {};
\node[right,gray] at (x2p) {$\xi_{i_3}$};
\node[below=5pt,gray] at (4,-5) {$\ldots$};
\node[gray,circle,fill,inner sep=0.8pt] (x3) at (4,2) {};
\node[right,gray] at (x3) {$\xi_{j_2}$};
\node[gray,circle,fill,inner sep=0.8pt] (x3p) at (4,2.5) {};
\node[right,gray] at (x3p) {$\xi_{i_2}$};
\node[gray,circle,fill,inner sep=0.8pt] (x4) at (4,4.5) {};
\node[right,gray] at (x4) {$\xi_{k_0}$};
\node[gray,circle,fill,inner sep=0.8pt] (x4p) at (4,5) {};
\node[right,gray] at (x4p) {$\xi_1$};

\coordinate (y1) at (4,-4.5) {};
\coordinate (y1p) at (4,-5) {};
\coordinate (y2)  at (4,-2) {};
\coordinate (y2p)  at (4,-2.5) {};
\coordinate (y3) at (4,1) {};
\coordinate (y3p) at (4,0.5) {};
\coordinate (y4) at (4,3.5) {};
\coordinate (y4p) at (4,3) {};

\node[blue,right] at (4,-4.75) {$\rbrace k_{n-1}$};
\node[blue,right] at (4,-2.25) {$\rbrace k_{n-2}$};
\node[blue,right] at (4,0.75) {$\rbrace k_{1}$};
\node[blue,right] at (4,3.25) {$\rbrace k_{0}$};

\coordinate (c) at (-1.5,-1.5);

\draw[blue] (y4p)--(y4p-|c) arc (270:90:1)--(x4p);
\draw[blue] (y4)--(y4-|c) arc (270:90:0.5)--(x4);
\draw[blue] (y3p)--(y3p-|c) arc (270:90:1)--(x3p);
\draw[blue] (y3)--(y3-|c) arc (270:90:0.5)--(x3);
\draw[blue] (y2p)--(y2p-|c) arc (270:90:1)--(x2p);
\draw[blue] (y2)--(y2-|c) arc (270:90:0.5)--(x2);
\draw[blue] (y1p)--(y1p-|c) arc (270:90:1)--(x1p);
\draw[blue] (y1)--(y1-|c) arc (270:90:0.5)--(x1);


\end{tikzpicture}}} .
\end{equation*}
where $\sum k_i = r$. We naturally assign a class in $\Hnot_r \left(X_r, \partial X_r \setminus X_r^-; \Laurent_r \right)$ with the above diagram according to the following process:
\begin{itemize}
\item the union of blue arcs well defines an embedding:
\[
\Phi : I^r \to X_r.
\]
where $I$ is the unit interval. 
\item As ends of blue arcs are lying in $\partial X_r \setminus X_r^-$, the hypercube $\Phi$ defines a homology class in $\Hnot_r \left(X_r, \partial X_r \setminus X_r^-; \BZ \right)$. 
\item It remains to choose a lift of $\Phi$ to the cover $\widehat{X_r}$ so to work in the local ring set-up. We do so using the fact that the image of $\Phi$ contains the base point ${\pmb \xi}$, so that we choose the only lift of $\Phi$ containing our choice of lift for the base point $\widehat{\pmb \xi}$. 
\end{itemize}
We still denote $B''(k_0, \ldots, k_{n-1})$ the resulting element of $\Hnot_r \left(X_r, \partial X_r \setminus X_r^-; \Laurent_r \right)$. 
\end{defn}

\begin{prop}\label{dual_bases_A''_B''}
We have the following:
\[
\left\langle A''(k_0,\ldots , k_{n-1}) \cap B''(k'_0,\ldots, k'_{n-1}) \right\rangle = \delta_{(k_0,\ldots , k_{n-1}),(k'_0,\ldots ,k'_{n-1})}
\]
where $\delta$ is the Kronecker (list) symbol, and $\langle \cdot \cap \cdot \rangle$ is the intersection pairing arising from the Poincaré--Lefschetz duality:
\[
\Hrelm_r \times \Hnot_r \left(X_r, \partial X_r \setminus X_r^-; \Laurent_r \right) \to \Laurent .
\]
\end{prop}
\begin{proof}
We put diagrams associated with $A''(k_0,\ldots , k_{n-1})$ and $B''(k'_0,\ldots, k'_{n-1})$ all together in one picture. 
\begin{equation*}
\vcenter{\hbox{\begin{tikzpicture}[scale=0.55, every node/.style={scale=0.8},decoration={
    markings,
    mark=at position 0.5 with {\arrow{>}}}
    ]
    
\draw[gray] (-5,0) -- (-5,6);
\draw[gray] (-5,0) -- (-5,-6);
\draw[gray] (-5,-6) -- (4,-6) -- (4,6) -- (-5,6);

\node (w0) at (-5,0) {};
\node (w1) at (0,4) {};
\node (w2) at (0,1.5) {};
\node[gray] at (0,0) {\vdots};
\node (wn1) at (0,-1.5) {};
\node (wn) at (0,-4) {};

\node[gray] at (w0)[left=5pt] {$w_0$};
\node[gray] at (w1)[right=5pt] {$w_1$};
\node[gray] at (w2)[right=5pt] {$w_2$};
\node[gray] at (wn1)[right=5pt] {$w_{n-1}$};
\node[gray] at (wn)[right=5pt] {$w_n$};
\foreach \n in {w1,w2,wn1,wn}
  \node at (\n)[gray,circle,fill,inner sep=3pt]{};
\node at (w0)[gray,circle,fill,inner sep=3pt]{};


\node[gray,circle,fill,inner sep=0.8pt] (x1) at (4,-3.5) {};
\node[right,gray] at (x1) {$\xi_r$};
\node[gray,circle,fill,inner sep=0.8pt] (x1p) at (4,-3) {};
\node[right,gray] at (x1p) {$\xi_{i_4}$};
\node[gray,circle,fill,inner sep=0.8pt] (x2)  at (4,-1) {};
\node[right,gray] at (x2) {$\xi_{j_3}$};
\node[gray,circle,fill,inner sep=0.8pt] (x2p)  at (4,-0.5) {};
\node[right,gray] at (x2p) {$\xi_{i_3}$};
\node[below=5pt,gray] at (4,-5) {$\ldots$};
\node[gray,circle,fill,inner sep=0.8pt] (x3) at (4,2) {};
\node[right,gray] at (x3) {$\xi_{j_2}$};
\node[gray,circle,fill,inner sep=0.8pt] (x3p) at (4,2.5) {};
\node[right,gray] at (x3p) {$\xi_{i_2}$};
\node[gray,circle,fill,inner sep=0.8pt] (x4) at (4,4.5) {};
\node[right,gray] at (x4) {$\xi_{k_0}$};
\node[gray,circle,fill,inner sep=0.8pt] (x4p) at (4,5) {};
\node[right,gray] at (x4p) {$\xi_1$};

\draw[dashed] (w0) to[bend left=10]  node[pos=0.4]  {$k_0$} node[pos=0.4] (k0) {} (w1);
\draw[dashed] (w0) to node[pos=0.4] (k1) {} node[pos=0.4] {$k_1$} (w2);
\draw[dashed] (w0) to node[pos=0.6] (k2) {} node[pos=0.4] {$k_{n-2}$} (wn1);
\draw[dashed] (w0) to[bend right=10] node[pos=0.8] (k3) {} node[pos=0.3] {$k_{n-1}$} (wn);

\coordinate (b) at (-3,-3);
\coordinate (bp) at (-2.5,-2.5);
\coordinate (bpp) at (-1.5,-1.5);

\coordinate (t4) at (3.5,4.75);
\coordinate (t3) at (3.5,2.25);
\coordinate (t2) at (3.5,-0.75);
\coordinate (t1) at (3.5,-3.25);

\draw[red] (x4p)--(t4);
\draw[red] (x4)--(t4);
\draw[red] (x3p)--(t3);
\draw[red] (x3)--(t3);
\draw[red] (x2p)--(t2);
\draw[red] (x2)--(t2);
\draw[red] (x1p)--(t1);
\draw[red] (x1)--(t1);

\draw[double,red] (t1)--(k3);
\draw[double,red] (t2)--(t2-|bpp)--(k2);
\draw[double,red] (t3)--(t3-|bpp)--(k1);
\draw[double,red] (t4)--(t4-|bp)--(k0);



\coordinate (y1) at (4,-4.5) {};
\coordinate (y1p) at (4,-5) {};
\coordinate (y2)  at (4,-2) {};
\coordinate (y2p)  at (4,-2.5) {};
\coordinate (y3) at (4,1) {};
\coordinate (y3p) at (4,0.5) {};
\coordinate (y4) at (4,3.5) {};
\coordinate (y4p) at (4,3) {};

\node[blue,right] at (4,-4.75) {$\rbrace k_{n-1}$};
\node[blue,right] at (4,-2.25) {$\rbrace k_{n-2}$};
\node[blue,right] at (4,0.75) {$\rbrace k_{1}$};
\node[blue,right] at (4,3.25) {$\rbrace k_{0}$};

\coordinate (c) at (-1.5,-1.5);

\draw[blue] (y4p)--(y4p-|c) arc (270:90:1)--(x4p);
\draw[blue] (y4)--(y4-|c) arc (270:90:0.5)--(x4);
\draw[blue] (y3p)--(y3p-|c) arc (270:90:1)--(x3p);
\draw[blue] (y3)--(y3-|c) arc (270:90:0.5)--(x3);
\draw[blue] (y2p)--(y2p-|c) arc (270:90:1)--(x2p);
\draw[blue] (y2)--(y2-|c) arc (270:90:0.5)--(x2);
\draw[blue] (y1p)--(y1p-|c) arc (270:90:1)--(x1p);
\draw[blue] (y1)--(y1-|c) arc (270:90:0.5)--(x1);


\end{tikzpicture}}} .
\end{equation*}
An intersection point is a configuration lying on both manifolds, namely such a configuration has one point on each blue arc, and $k$ of them on a dashed arc indexed by $k$. The only way there is an intersection is: $k_0=k_0', \ldots, k_{n-1}= k_{n-1}'$ which explains the Kronecker symbol in the formula. When we are in this equality case, we call ${\pmb p}:= \lbrace p_1 , \ldots , p_r \rbrace$ the single intersection configuration, and it remains to prove that the intersection is equal to $1$ at this configuration. We recall that the Poincaré--Lefschetz duality gives an intersection pairing:
\[
\Hrelm_r \times \Hnot_r \left(X_r, \partial X_r \setminus X_r^-; \Laurent_r \right) \to \Laurent .
\]
see \cite[Lemma~4.1]{martel2020coloredJones}. This pairing is given by graded intersection, where each intersection point contributes for a sign (that of the intersection) times a monomial in $\Laurent$. Let $\widehat{A''}(k_0,\ldots , k_{n-1})$ and $\widehat{B''}(k_0,\ldots , k_{n-1})$ the lifts of the corresponding manifolds chosen using the red handle resp. the one that contains $\hat{{\pmb \xi}^r}$. In our case, the only monomial $m_{\pmb p}$ involved could be computed by defining the following loop in $X_r$, by composing paths:
\begin{itemize}
\item First the path going from $\lbrace \xi_1 , \ldots , \xi_r \rbrace$ to $A''(k_0,\ldots , k_{n-1})$ following red handles,
\item then joining $\lbrace p_1 , \ldots , p_r \rbrace$ going along $A''(k_0,\ldots , k_{n-1})$,
\item then going back to ${\pmb \xi}^r$ running along $B''(k_0,\ldots , k_{n-1})$.
\end{itemize}
This composition of paths yields a loop denoted $\gamma_{\pmb p}$ of $X_r$ based at ${\pmb \xi}$. By considering one of its lift to $\widehat{X_r}$, one can check that it relates $\widehat{\pmb \xi}$ and $m_{\pmb p } \widehat{\pmb \xi}$. The explanation of this fact is exactly the same as the one given before Lemma 3.11 in \cite{martel2020colored} which is adapted from Section 3.1 of \cite{Big1}.  Knowing this, we directly conclude:
\[
m_{\pmb p} = \rho_r(\gamma_{\pmb p }),
\]
and moreover that:
\[
m_{\pmb p} = \rho_r(\gamma{\pmb p })=1. 
\]
One can check that the braid given by $\gamma_{\pmb p}$ seen as an element of $\pi_1(X_r,{\pmb \xi})$ (following the model \cite[Remark~2.2]{martel2020homological}) is trivial. This ensures the above equality by the definition of $\rho_r$ (Definition \ref{localsystXr}). 
\end{proof}

\begin{rem}[Homological dual bases]
The above theorem says that sets $\lbrace A''(k_0,\ldots , k_{n-1}) \text{ s.t. } \sum k_i=r \rbrace$ and $\lbrace B''(k_0,\ldots , k_{n-1}) \text{ s.t. } \sum k_i=r \rbrace$ are dual bases of $\Hrelm_r$ resp. $\Hnot_r \left(X_r, \partial X_r \setminus X_r^-; \Laurent_r \right)$ (the one to one correspondence being given by the canonical indexing). 
\end{rem}


\subsection{Unified invariant from homological intersection}\label{sec_Foo_thm_fromHomology}

\begin{thm}[The unified invariant from homological intersection pairing]\label{thm_homol_formula_Foo}
Let $\beta \in \Bn$ a braid such that its closure is the knot $\CK$. Then, letting $t=-q^{-2}$:
\[
F_{\infty}(\CK) = s^{n-1}\sum_{\overline{k} \in \BN_0^n} \left\langle \beta  \cdot A''(0,k_1,\ldots, k_{n-1}) \cap B''(0,k_1,\ldots, k_{n-1}) \right\rangle q^{-2\sum k_i} 
\]
where the action of $\beta$ is that from Lemma \ref{Lawrence_rep}. The latter means that the right term in the equation, which is an infinite sum of intersection pairing of middle dimension homology classes, lives in $\Laurentcomplet$ and is invariant under Markov moves. 
\end{thm}
\begin{proof}
The main tool is Theorem \ref{ModelMartelHomological2020} which shows that (under $t=-q^{-2}$) by sending $v_{i_1} \otimes \cdots \otimes v_{i_{n-1}}$ to $A(i_0,\ldots , i_{n-1})$ (for any integers $i_0,\ldots , i_{n-1}$) then matrices for the quantum action and the homological actions of $\beta$ are strictly identical. The partial trace formula from Theorem \ref{prop_unifed_braidrep} is the same replacing the $v_i$'s by the $A$ vectors from the homological side. Then the partial trace of any endomorphism $f$ of $\CH$ could be expressed as follows:
\[
\Trp (f) = \sum_{\overline{k} \in \BN_0^n} \left\langle f\left((A(0,k_1,\ldots , k_{n-1}) \right) , A_*(0,k_1,\ldots , k_{n-1}) \rangle\right. 
\]
where $A_*$ means the dual family of $A$ regarding the Poincaré--Lefschetz duality studied in Prop. \ref{dual_bases_A''_B''}. As the change of bases from $A$'s to $A''$'s is diagonal (see Prop. \ref{rel_A_A'_A''}), we can replace $A$ and $A_*$ in the above formula by the $A''$'s and its dual family, namely the $B''$ as it was proved in Prop. \ref{dual_bases_A''_B''}. Now the $f$ we wish to consider here is $(1 \otimes K^{\otimes n-1}) \circ \phi_n(\beta)$. One notices that $(1 \otimes K^{\otimes n-1})$ on the image by $\phi_n(\beta)$ of any $A''(0,k_1,\ldots, k_{n-1})$ is the multiplication by $s^{n-1} q^{-2 \sum{k_i}}$ which concludes the proof. 
\end{proof}

The fact that $F_{\infty}$ interpolates ADO invariants and colored Jones polynomials by some specialization, implies the above theorem at the corresponding specialization gives infinite sum from which one can extract these invariants out of Lagrangian intersections. Moreover, infinite sums are not crucial. 

\begin{cor}
Let $\beta \in \Bn$ a braid such that its closure is the knot $\CK$.
\begin{enumerate}
\item Let $J_N$ be the $N$-th colored Jones polynomial and $\spec$ the specialization morphism sending $s$ to $q^N$, then:
\[
J_N(\CK) = q^{N(n-1)}\sum_{\begin{array}{c} \overline{k} \text{ s.t. } \\ \sum k_i < N, \forall i \end{array}} \spec\left( \left\langle \beta  \cdot A''(0,k_1,\ldots, k_{n-1}) \cap B''(0,k_1,\ldots, k_{n-1}) \right\rangle q^{-2\sum k_i} \right) 
\]
\item Let $\ADO_r$ be the $r$-th ADO polynomial and $\spec$ the specialization morphism sending $q$ to $\zeta_{2r}$, then:
\[
\ADO_r(\CK) = s^{n-1}\sum_{\begin{array}{c} \overline{k} \text{ s.t. } \\ \forall i, k_i < r  \end{array}} \spec\left( \left\langle \beta  \cdot A''(0,k_1,\ldots, k_{n-1}) \cap B''(0,k_1,\ldots, k_{n-1}) \right\rangle q^{-2\sum k_i} \right) 
\]
\end{enumerate}
\end{cor}
\begin{proof}
The proof is the same as that of the previous theorem. The sum is truncated directly as only first weight levels are necessary to compute them, see:
\begin{enumerate}
\item \cite[Lemma~51]{willetts2020unification} for the colored Jones case. 
\item Prop. \ref{prop_braid_ado} for the ADO case. 
\end{enumerate}
\end{proof}

\begin{rem}[Normalized colored Jones]
As a partial trace is involved in the above formula for the colored Jones polynomial, we are dealing with the normalized version (being $1$ on the unknot), which is a different version as that in \cite{martel2020coloredJones}. It explains the difference of the homology classes involved in the sums from here and the mentioned paper. 
\end{rem}

\begin{ex}[The trefoil knot]
We illustrate the fact that the homological formula from Theorem \ref{thm_homol_formula_Foo} gives an independent algorithm of computation arising from homological computation by computing $F_{\infty}$ with this formula and comparing with the expression given in \cite[Sec.~5]{willetts2020unification}. To do the computation, we use the element $A'(0,k) \in \CH$ (instead of $A''$ in the formula, for clarity of diagrams) for $k\in \BN$, the disk is considered with two punctures ($w_1,w_2$) as we need a braid with two strands to get the trefoil knot as a braid closure (namely the closure of $\sigma_1^{-3} \in \CB_2$, we use inverse twists for simplicity of diagrams). We use notation $s=q^{\alpha}$. 
\begin{equation*}
\sigma_1^{-3} (A'(0,k)) = \sigma_1^{-3} \left( \vcenter{\hbox{
\begin{tikzpicture}[scale=0.35,every node/.style={scale=0.7},decoration={
    markings,
    mark=at position 0.5 with {\arrow{>}}}
    ] 
\node (w0) at (-5,0) {};
\node (w00) at (5,0) {};
\node (w1) at (-2,0) {};
\node (w2) at (2,0) {};

\node[gray,circle,fill,inner sep=0.8pt] (xir) at (-2,-3) {};
\node[below,gray] at (xir) {$\xi_k$};
\node[gray,circle,fill,inner sep=0.8pt] (xirk0) at (-1,-3) {};
\node[below,gray] at (xirk0) {$\xi_{1}$};

\coordinate (a) at (-3,-3);
\coordinate (b) at (4,4);

\draw[postaction=decorate,dashed] (w0) to[bend right=30] node[pos=0.5,above] (k0) {$k$} (w2) ;

\node[gray] at (w2)[above] {$w_2$};
\node[gray] at (w1)[above] {$w_1$};
\node[gray] at (w0) [left=5pt] {$w_0$};
\foreach \n in {w1,w2}
  \node at (\n)[gray,circle,fill,inner sep=2pt]{};
\node at (w0)[gray,circle,fill,inner sep=2pt]{};

\coordinate (c) at (-2.5,-2.5);

\draw[double,red] (k0)--(k0|-c);
\draw[red] (k0|-c) -- (xir);
\draw[red] (k0|-c) -- (xirk0);

\draw[gray] (-5,2) -- (-5,-3);
\draw[gray] (-5,-3) -- (3,-3);
\draw[gray] (3,2) -- (3,-3);
\draw[gray] (3,2) -- (-5,2);
\end{tikzpicture} }}
\right)= 
\vcenter{\hbox{
\begin{tikzpicture}[scale=0.35,every node/.style={scale=0.7},decoration={
    markings,
    mark=at position 0.5 with {\arrow{>}}}
    ] 
\node (w0) at (-5,1) {};
\node (w00) at (5,1) {};
\node (w1) at (1,1) {};
\node (w2) at (-2.5,1) {};
\coordinate (step1) at (3.5,1);
\coordinate (step2) at (-4,1);
\coordinate (step3) at (2.5,1);

\node[gray,circle,fill,inner sep=0.8pt] (xir) at (-4,-2) {};
\node[below,gray] at (xir) {$\xi_k$};
\node[gray,circle,fill,inner sep=0.8pt] (xirk0) at (-3,-2) {};
\node[below,gray] at (xirk0) {$\xi_{1}$};

\coordinate (a) at (-3,-3);
\coordinate (b) at (4,4);

\draw[postaction=decorate,dashed] (w0) to[bend left=90] node[pos=0.3,above] (k0) {} (step3) ;
\draw[dashed] (step3) to[bend left=90] (w2) ;

\node[gray] at (w2)[above] {$w_2$};
\node[gray] at (w1)[above] {$w_1$};
\node[gray] at (w0) [left=5pt] {$w_0$};
\foreach \n in {w1,w2}
  \node at (\n)[gray,circle,fill,inner sep=2pt]{};
\node at (w0)[gray,circle,fill,inner sep=2pt]{};

\coordinate (c) at (-1.5,-1.5);

\draw[double,red] (k0)--(k0|-c);
\draw[red] (k0|-c) -- (xir);
\draw[red] (k0|-c) -- (xirk0);

\draw[gray] (-5,4) -- (-5,-2);
\draw[gray] (-5,-2) -- (3,-2);
\draw[gray] (3,4) -- (3,-2);
\draw[gray] (3,4) -- (-5,4);
\end{tikzpicture} }}
\end{equation*}
We recall that this element have to be paired with the dual class of $A'(0,k)$ (and then summed over $k\in \BN$) so to obtain $F_{\infty}$. The class $B'(0,k)$ with the following diagram is this dual class (i.e. $A'(0,k)$ and $B'(0,k)$ have pairing $1$, see the proof of Prop. \ref{dual_bases_A''_B''}):
\[
B'(0,k):= \vcenter{\hbox{
\begin{tikzpicture}[scale=0.35,decoration={
    markings,
    mark=at position 0.5 with {\arrow{<}}}
    ] 
\node (w0) at (-5,1) {};
\node (w00) at (5,1) {};
\node (w1) at (1,1) {};
\node (w2) at (-2.5,1) {};

\node[gray,circle,fill,inner sep=0.8pt] (xik) at (-1,-3) {};
\node[below,gray] at (xik) {$\xi_k$};
\node[gray,circle,fill,inner sep=0.8pt] (xi1) at (0,-3) {};
\node[below,gray] at (xi1) {$\xi_{1}$};

\node[gray,circle,fill,inner sep=0.8pt] (xikp) at (3,-3) {};
\node[gray,circle,fill,inner sep=0.8pt] (xi1p) at (2,-3) {};

\coordinate (a) at (-3,-3);
\coordinate (b) at (4,4);

\node[gray] at (w2)[above] {$w_2$};
\node[gray] at (w1)[above] {$w_1$};
\node[gray] at (w0) [left=5pt] {$w_0$};
\foreach \n in {w1,w2}
  \node at (\n)[gray,circle,fill,inner sep=2pt]{};
\node at (w0)[gray,circle,fill,inner sep=2pt]{};

\coordinate (c) at (1,1);
\coordinate (d) at (0.5,0.5);

\draw[blue] (xik)--(xik|-c) arc(180:0:2)--(xikp);
\draw[blue] (xi1)--(xi1|-c) arc(180:0:1)--(xi1p);

\draw[gray] (-5,4) -- (-5,-3);
\draw[gray] (-5,-3) -- (5,-3);
\draw[gray] (5,4) -- (5,-3);
\draw[gray] (5,4) -- (-5,4);
\end{tikzpicture} }}
\]
We simplify the diagram of $\sigma_1^{-3} (A'(0,k))$ using rules from \cite[Sec.~4]{martel2020homological} to simplify the computation of the intersection.
\[
\vcenter{\hbox{
\begin{tikzpicture}[scale=0.35,every node/.style={scale=0.7},decoration={
    markings,
    mark=at position 0.5 with {\arrow{>}}}
    ] 
\node (w0) at (-5,1) {};
\node (w1) at (1,1) {};
\node (w2) at (-2.5,1) {};
\coordinate (step1) at (3.5,1);
\coordinate (step2) at (-4,1);
\coordinate (step3) at (2.5,1);

\node[gray,circle,fill,inner sep=0.8pt] (xir) at (-4,-2) {};
\node[below,gray] at (xir) {$\xi_k$};
\node[gray,circle,fill,inner sep=0.8pt] (xirk0) at (-3,-2) {};
\node[below,gray] at (xirk0) {$\xi_{1}$};

\coordinate (a) at (-3,-3);
\coordinate (b) at (4,4);

\draw[postaction=decorate,dashed] (w0) to[bend left=90] node[pos=0.3,above] (k0) {} (step3) ;
\draw[dashed] (step3) to[bend left=90] (w2) ;

\node[gray] at (w2)[above] {$w_2$};
\node[gray] at (w1)[above] {$w_1$};
\node[gray] at (w0) [left=5pt] {$w_0$};
\foreach \n in {w1,w2}
  \node at (\n)[gray,circle,fill,inner sep=2pt]{};
\node at (w0)[gray,circle,fill,inner sep=2pt]{};

\coordinate (c) at (-1.5,-1.5);

\draw[double,red] (k0)--(k0|-c);
\draw[red] (k0|-c) -- (xir);
\draw[red] (k0|-c) -- (xirk0);

\draw[gray] (-5,4) -- (-5,-2);
\draw[gray] (-5,-2) -- (3,-2);
\draw[gray] (3,4) -- (3,-2);
\draw[gray] (3,4) -- (-5,4);
\end{tikzpicture} }}=
%
%
%
%
%
%
%
\sum_{l=0}^k 
\vcenter{\hbox{
\begin{tikzpicture}[scale=0.35,every node/.style={scale=0.7},decoration={
    markings,
    mark=at position 0.5 with {\arrow{>}}}
    ] 
\node (w0) at (-5,1) {};
\node (w00) at (5,1) {};
\node (w1) at (1,1) {};
\node (w2) at (-1.5,1) {};
\coordinate (step1) at (3.5,1);
\coordinate (step2) at (2,1);
\coordinate (step3) at (2.5,1);

\node[gray,circle,fill,inner sep=0.8pt] (xir) at (-4.5,-2) {};
\node[below,gray] at (xir) {$\xi_k$};
\node[gray,circle,fill,inner sep=0.8pt] (xirk0) at (-3.5,-2) {};
\node[below,gray] at (xirk0) {};
\node[gray,circle,fill,inner sep=0.8pt] (xil) at (-3.3,-2) {};
\node[below,gray] at (xil) {$\xi_{l}$};
\node[gray,circle,fill,inner sep=0.8pt] (xi1) at (-2.3,-2) {};
\node[below,gray] at (xi1) {$\xi_{1}$};

\coordinate (a) at (-3,-3);
\coordinate (b) at (4,4);

\draw[postaction=decorate,dashed] (w0) to[bend left=90] node[pos=0.2,above] (k0) {$k-l$} (step3) ;
\draw[dashed] (step3) to[bend left=90] (w0) ;
\draw[dashed] (w0) to node[pos=0.6] (k11) {} node[pos=0.8,above] (k1) {} node[pos=0.5,above] {$l$} (w2) ;

\node[gray] at (w2)[above] {$w_2$};
\node[gray] at (w1)[above] {$w_1$};
\node[gray] at (w0) [left=5pt] {$w_0$};
\foreach \n in {w1,w2}
  \node at (\n)[gray,circle,fill,inner sep=2pt]{};
\node at (w0)[gray,circle,fill,inner sep=2pt]{};

\coordinate (c) at (-1.5,-1.5);

\draw[double,red] (k0)--(k0|-c);
\draw[red] (k0|-c) -- (xir);
\draw[red] (k0|-c) -- (xirk0);

\draw[double,red] (k1) to[bend right=90] (step2) to[bend right=90] (k11) to (k11|-c);
\draw[red] (k11|-c) -- (xi1);
\draw[red] (k11|-c) -- (xil);

\draw[gray] (-5,4) -- (-5,-2);
\draw[gray] (-5,-2) -- (3,-2);
\draw[gray] (3,4) -- (3,-2);
\draw[gray] (3,4) -- (-5,4);
\end{tikzpicture} }}
\]
This is similar to \cite[Example~4.6]{martel2020homological}. In the right hand sum, when ever $l$ is not $0$, the manifold associated with the diagram has no intersection with $B'(0,k)$ (an intersection point is a $k$-tuple, one on each blue line of $B'(0,k)$ and respecting the indices of $A'(0,k)$). Hence the only diagram having non trivial intersection with $B'(0,k)$ is when $l=0$ so that:%
\begin{align*}
\left\langle \sigma_1^{-3} A'(0,k) \cap B'(0,k) \right\rangle & =  \sum_{l=0}^k  \left\langle 
\vcenter{\hbox{
\begin{tikzpicture}[scale=0.35,every node/.style={scale=0.7},decoration={
    markings,
    mark=at position 0.5 with {\arrow{<}}}
    ] 
\node (w0) at (-5,1) {};
\node (w00) at (5,1) {};
\node (w1) at (1,1) {};
\node (w2) at (-1.5,1) {};
\coordinate (step1) at (3.5,1);
\coordinate (step2) at (2,1);
\coordinate (step3) at (2.5,1);
\node[gray,circle,fill,inner sep=0.8pt] (xir) at (-4.5,-2) {};
\node[below,gray] at (xir) {$\xi_k$};
\node[gray,circle,fill,inner sep=0.8pt] (xirk0) at (-3.5,-2) {};
\node[below,gray] at (xirk0) {};
\node[gray,circle,fill,inner sep=0.8pt] (xil) at (-3.3,-2) {};
\node[below,gray] at (xil) {$\xi_{l}$};
\node[gray,circle,fill,inner sep=0.8pt] (xi1) at (-2.3,-2) {};
\node[below,gray] at (xi1) {$\xi_{1}$};
\coordinate (a) at (-3,-3);
\coordinate (b) at (4,4);
\draw[dashed] (w0) to[bend left=90] node[pos=0.2,above] (k0) {$k-l$} (step3) ;
\draw[dashed] (step3) to[bend left=90] (w0) ;
\draw[dashed] (w0) to node[pos=0.6] (k11) {} node[pos=0.8,above] (k1) {} node[pos=0.5,above] {$l$} (w2) ;
\node[gray] at (w2)[above] {$w_2$};
\node[gray] at (w1)[above] {$w_1$};
\node[gray] at (w0) [left=5pt] {$w_0$};
\foreach \n in {w1,w2}
  \node at (\n)[gray,circle,fill,inner sep=2pt]{};
\node at (w0)[gray,circle,fill,inner sep=2pt]{};
\coordinate (c) at (-1.5,-1.5);
\draw[double,red] (k0)--(k0|-c);
\draw[red] (k0|-c) -- (xir);
\draw[red] (k0|-c) -- (xirk0);
\draw[double,red] (k1) to[bend right=90] (step2) to[bend right=90] (k11) to (k11|-c);
\draw[red] (k11|-c) -- (xi1);
\draw[red] (k11|-c) -- (xil);
\draw[gray] (-5,4) -- (-5,-2);
\draw[gray] (-5,-2) -- (3,-2);
\draw[gray] (3,4) -- (3,-2);
\draw[gray] (3,4) -- (-5,4);
\end{tikzpicture} }}
\cap B'(0,k) \right\rangle \\
& =   \left\langle 
\vcenter{\hbox{
\begin{tikzpicture}[scale=0.35,every node/.style={scale=0.7},decoration={
    markings,
    mark=at position 0.5 with {\arrow{>}}}
    ] 
\node (w0) at (-5,1) {};
\node (w00) at (5,1) {};
\node (w1) at (1,1) {};
\node (w2) at (-1.5,1) {};
\coordinate (step1) at (3.5,1);
\coordinate (step2) at (2,1);
\coordinate (step3) at (2.5,1);
\node[gray,circle,fill,inner sep=0.8pt] (xir) at (-4.5,-2) {};
\node[below,gray] at (xir) {$\xi_k$};
\node[gray,circle,fill,inner sep=0.8pt] (xirk0) at (-3.5,-2) {};
\node[below,gray] at (xirk0) {$\xi_1$};
\coordinate (a) at (-3,-3);
\coordinate (b) at (4,4);
\draw[dashed,postaction=decorate] (w0) to[bend left=90] node[pos=0.3,above] (k0) {$k$} (step3) ;
\draw[dashed] (step3) to[bend left=90] (w0) ;
\node[gray] at (w2)[above] {$w_2$};
\node[gray] at (w1)[above] {$w_1$};
\node[gray] at (w0) [left=5pt] {$w_0$};
\foreach \n in {w1,w2}
  \node at (\n)[gray,circle,fill,inner sep=2pt]{};
\node at (w0)[gray,circle,fill,inner sep=2pt]{};
\coordinate (c) at (-1.5,-1.5);
\draw[double,red] (k0)--(k0|-c);
\draw[red] (k0|-c) -- (xir);
\draw[red] (k0|-c) -- (xirk0);
\draw[gray] (-5,4) -- (-5,-2);
\draw[gray] (-5,-2) -- (3,-2);
\draw[gray] (3,4) -- (3,-2);
\draw[gray] (3,4) -- (-5,4);
\end{tikzpicture} }}
\cap B'(0,k) \right\rangle
\end{align*}
We use the {\em handle rule} \cite[Rem~4.2]{martel2020homological} (see also in the proof of Prop. \ref{rel_A_A'_A''}), to rearrange the red handle:
\[
\vcenter{\hbox{
\begin{tikzpicture}[scale=0.35,every node/.style={scale=0.7},decoration={
    markings,
    mark=at position 0.5 with {\arrow{>}}}
    ] 
\node (w0) at (-5,1) {};
\node (w00) at (5,1) {};
\node (w1) at (1,1) {};
\node (w2) at (-1.5,1) {};
\coordinate (step1) at (3.5,1);
\coordinate (step2) at (2,1);
\coordinate (step3) at (2.5,1);
\node[gray,circle,fill,inner sep=0.8pt] (xir) at (-4.5,-2) {};
\node[below,gray] at (xir) {$\xi_k$};
\node[gray,circle,fill,inner sep=0.8pt] (xirk0) at (-3.5,-2) {};
\node[below,gray] at (xirk0) {$\xi_1$};
\coordinate (a) at (-3,-3);
\coordinate (b) at (4,4);
\draw[dashed,postaction=decorate] (w0) to[bend left=90] node[pos=0.3,above] (k0) {$k$} (step3) ;
\draw[dashed] (step3) to[bend left=90] (w0) ;
\node[gray] at (w2)[above] {$w_2$};
\node[gray] at (w1)[above] {$w_1$};
\node[gray] at (w0) [left=5pt] {$w_0$};
\foreach \n in {w1,w2}
  \node at (\n)[gray,circle,fill,inner sep=2pt]{};
\node at (w0)[gray,circle,fill,inner sep=2pt]{};
\coordinate (c) at (-1.5,-1.5);
\draw[double,red] (k0)--(k0|-c);
\draw[red] (k0|-c) -- (xir);
\draw[red] (k0|-c) -- (xirk0);
\draw[gray] (-5,4) -- (-5,-2);
\draw[gray] (-5,-2) -- (3,-2);
\draw[gray] (3,4) -- (3,-2);
\draw[gray] (3,4) -- (-5,4);
\end{tikzpicture} }} = (-1)^k (-t)^{\frac{k(k-1)}{2}} \left( q^{2 \alpha} \right)^{2k}
\vcenter{\hbox{
\begin{tikzpicture}[scale=0.35,every node/.style={scale=0.7},decoration={
    markings,
    mark=at position 0.5 with {\arrow{>}}}
    ] 
\node (w0) at (-5,1) {};
\node (w00) at (5,1) {};
\node (w1) at (1,1) {};
\node (w2) at (-1.5,1) {};
\coordinate (step1) at (3.5,1);
\coordinate (step2) at (2,1);
\coordinate (step3) at (2.5,1);
\node[gray,circle,fill,inner sep=0.8pt] (xir) at (-4.5,-2) {};
\node[below,gray] at (xir) {$\xi_k$};
\node[gray,circle,fill,inner sep=0.8pt] (xirk0) at (-3.5,-2) {};
\node[below,gray] at (xirk0) {$\xi_1$};
\coordinate (a) at (-3,-3);
\coordinate (b) at (4,4);
\draw[dashed,postaction=decorate] (w0) to[bend right=90] node[pos=0.3,above] (k0) {$k$} (step3) ;
\draw[dashed] (step3) to[bend right=90] (w0) ;
\node[gray] at (w2)[above] {$w_2$};
\node[gray] at (w1)[above] {$w_1$};
\node[gray] at (w0) [left=5pt] {$w_0$};
\foreach \n in {w1,w2}
  \node at (\n)[gray,circle,fill,inner sep=2pt]{};
\node at (w0)[gray,circle,fill,inner sep=2pt]{};
\coordinate (c) at (-1.5,-1.5);
\draw[double,red] (k0)--(k0|-c);
\draw[red] (k0|-c) -- (xir);
\draw[red] (k0|-c) -- (xirk0);
\draw[gray] (-5,4) -- (-5,-2);
\draw[gray] (-5,-2) -- (3,-2);
\draw[gray] (3,4) -- (3,-2);
\draw[gray] (3,4) -- (-5,4);
\end{tikzpicture} }}
\]
The coefficient showing up is the image by $\rho_k$ of the loop in $X_k$ defined as the composition of the red path on the left with that on the right. The $(-1)^k$ appears because we have reversed the orientation of the dashed arc. Then:
\begin{align*}
\left\langle \sigma_1^{-3} A'(0,k) \cap B'(0,k) \right\rangle & =  \left\langle (-1)^k (-t)^{\frac{k(k-1)}{2}} \left( q^{2 \alpha} \right)^{2k}
\vcenter{\hbox{
\begin{tikzpicture}[scale=0.35,every node/.style={scale=0.7},decoration={
    markings,
    mark=at position 0.5 with {\arrow{>}}}
    ] 
\node (w0) at (-5,1) {};
\node (w00) at (5,1) {};
\node (w1) at (1,1) {};
\node (w2) at (-1.5,1) {};
\coordinate (step1) at (3.5,1);
\coordinate (step2) at (2,1);
\coordinate (step3) at (2.5,1);
\node[gray,circle,fill,inner sep=0.8pt] (xir) at (-4.5,-2) {};
\node[below,gray] at (xir) {$\xi_k$};
\node[gray,circle,fill,inner sep=0.8pt] (xirk0) at (-3.5,-2) {};
\node[below,gray] at (xirk0) {$\xi_1$};
\coordinate (a) at (-3,-3);
\coordinate (b) at (4,4);
\draw[dashed,postaction=decorate] (w0) to[bend right=90] node[pos=0.3,above] (k0) {$k$} (step3) ;
\draw[dashed] (step3) to[bend right=90] (w0) ;
\node[gray] at (w2)[above] {$w_2$};
\node[gray] at (w1)[above] {$w_1$};
\node[gray] at (w0) [left=5pt] {$w_0$};
\foreach \n in {w1,w2}
  \node at (\n)[gray,circle,fill,inner sep=2pt]{};
\node at (w0)[gray,circle,fill,inner sep=2pt]{};
\coordinate (c) at (-1.5,-1.5);
\draw[double,red] (k0)--(k0|-c);
\draw[red] (k0|-c) -- (xir);
\draw[red] (k0|-c) -- (xirk0);
\draw[gray] (-5,4) -- (-5,-2);
\draw[gray] (-5,-2) -- (3,-2);
\draw[gray] (3,4) -- (3,-2);
\draw[gray] (3,4) -- (-5,4);
\end{tikzpicture} }}
\cap B'(0,k) \right\rangle \\
& =  (-1)^k (-t)^{\frac{k(k-1)}{2}} \left( q^{2 \alpha} \right)^{2k} \left\langle \vcenter{\hbox{\begin{tikzpicture}[scale=0.35,every node/.style={scale=0.7},decoration={
    markings,
    mark=at position 0.5 with {\arrow{>}}}
    ] 
\node (w0) at (-5,1) {};
\node (w00) at (5,1) {};
\node (w1) at (1,1) {};
\node (w2) at (-2.5,1) {};
\coordinate (step1) at (3,1);
\coordinate (step2) at (-5,1);
\coordinate (step3) at (2.5,1);
\coordinate (step4) at (0,1);
\coordinate (step2p) at (-1,1);
\node[gray,circle,fill,inner sep=0.8pt] (xik) at (-4,-3) {};
\node[below,gray] at (xik) {$\xi_k$};
\node[gray,circle,fill,inner sep=0.8pt] (xilp) at (-3,-3) {};
\node[below,gray] at (xilp) {$\xi_{1}$};
\coordinate (a) at (-3,-3);
\coordinate (b) at (4,4);
\draw[dashed,postaction=decorate] (w0) to[bend right=90] node[pos=0.3,above] (k0) {} (step1) ;
\draw[dashed] (step1) to[bend right=90] node[pos=0.3] {$k$} (step2p) ;
\draw[dashed] (step2p) to[bend left=70] (w0);
\node[gray] at (w2)[above] {$w_2$};
\node[gray] at (w1)[above] {$w_1$};
\node[gray] at (w0) [left=5pt] {};
\foreach \n in {w1,w2}
  \node at (\n)[gray,circle,fill,inner sep=2pt]{};
\node at (w0)[gray,circle,fill,inner sep=2pt]{};
\coordinate (c) at (-2.5,-2.5);
\coordinate (d) at (3.75,3.75);
\draw[double,red] (k0)--(k0|-c);
\draw[red] (k0|-c) -- (xilp);
\draw[red] (k0|-c) -- (xik);
\draw[gray] (-5,3) -- (-5,-3);
\draw[gray] (-5,-3) -- (4,-3);
\draw[gray] (4,3) -- (4,-3);
\draw[gray] (4,3) -- (-5,3);
\end{tikzpicture} }}
\cap B'(0,k) \right\rangle \\
& =  (-1)^k (-t)^{\frac{k(k-1)}{2}} \left( q^{2 \alpha} \right)^{2k} \left\langle \prod_{i=0}^{k-1} (1 - q^{-2\alpha} (-t)^{-i}) A'(0,k)\cap B'(0,k) \right\rangle
\end{align*}
where in the second equality, we have again applied \cite[Ex.~4.6]{martel2020homological} dividing the $k$ indexed arc into two arcs passing by $w_0$, but again, we have kept the only term of the hypothetical sum that has non trivial intersection with $B'(0,k)$. Last equality is straightforward from \cite[Prop.~7.4]{martel2020homological}. 
 Finally, to compute $F_{\infty}$ one has to do the identification $-t=q^{-2}$, and:
\begin{align*}
F_{\infty}(\sigma_1^{-3}) & = \sum_{k \in \BN} \langle A'(0,k) \cap B'(0,k) \rangle q^{\alpha-2k} \\ & = \sum_{k \in \BN} q^{\alpha-2k} q^{-k(k-1)} q^{4 \alpha k} (-1)^k \prod_{i=0}^{k-1} (1 - q^{-2\alpha} q^{2i}) \\
& = \sum_{k \in \BN} q^{\alpha-2k} q^{-k(k-1)} q^{4 \alpha k} (-1)^k q^{-\alpha k } q^{\frac{k(k-1)}{2}} \prod_{i=0}^{k-1} (q^{\alpha} q^{-i} - q^{-\alpha} q^{i})\\ & = \sum_{k \in \BN} q^{\alpha-2k} q^{3 \alpha k}  q^{-\frac{k(k-1)}{2}} (-1)^k \prod_{i=0}^{k-1} (q^{\alpha} q^{-i} - q^{-\alpha} q^{i}).
\end{align*}
One recovers precisely the formula for $F_{\infty}$ for the trefoil knot given in \cite[Sec.~5]{willetts2020unification}, with zero framing. The identification is under the change $q \to q^{-1}$ as we chose $\sigma_1^{-3}$ for the trefoil instead of $\sigma_1^{3}$. 
\end{ex}

Theorem \ref{thm_homol_formula_Foo} expresses $F_{\infty}$ from intersections of middle dimension submanifolds of configuration spaces, sometimes called Lagrangians. Such interpretations for quantum invariants was initiated by Lawrence, and then Bigelow for the Jones polynomial \cite{BigelowJones,LawrenceJones}, then in a more quantum way colored Jones and ADO polynomials were formulated in the same spirit in e.g. \cite{AnghelJones,martel2020coloredJones} resp. \cite{ItoAlex,AnghelAlex}. The present theorem should interpolate all these formulae (sometimes under a simple change of \emph{dual bases}, corresponding to changing the manifolds to pair). Moreover there is the uniqueness property of interpolation (Prop. \ref{prop_unicity_Foo}), that we recall in the following remark. 

\begin{rem}
\begin{itemize}
\item The fact that $F_{\infty}$ is the only two variables element interpolating both colored Jones polynomials or ADO polynomials (Prop. \ref{prop_unicity_Foo}), could be interpreted as the only intersection pairing computed from manifolds in abelian covers of configuration spaces of disks interpolating both families. More over it is a knot invariant. 
\item In \cite{martel2020coloredJones}, the second author has also showed that colored Jones polynomials compute some Lefschetz numbers. This is because colored Jones polynomials could be computed from a full trace on homological representations of braids, not only with a partial trace. With some study of the structure of homology modules, the trace formula then satisfies the Lefschetz formula. Unfortunately, authors have tried to interpret $F_{\infty}$ as a full trace on homological braid action, and only found convergence problems seeming to be essential.
\end{itemize}
\end{rem}

\section{Unified invariant from a quantum determinant}\label{sec_Foo_from_qdet}

This section is inspired by the paper \cite{H-L} where Huynh and Lê compute the colored Jones polynomials from $\Uq$ Verma modules. By some assimilation of tensor products of Verma modules with some quantum plane, they succeed in giving a formula for colored Jones polynomials involving a \emph{quantum determinant} for quantum matrices by use of the quantum Mac-Mahon Master theorem \cite{GLZ}. We follow Sec. 0.1 of \cite{H-L} to state the theorem and we will give a direct proof involving their theorem and some interpolation argument. 

\subsection{Deformed Burau matrix}

We recall that $\Laurent_0 := \BZ \left[  q^{\pm 1} \right]$. On the polynomial ring $\Laurent_0 \left[ x^{\pm 1} , y^{\pm 1}, u^{\pm 1} \right]$ we define operators.

\begin{defn}
Let $\hat{x}, \hat{y} , \hat{u}$ and $\tau_x, \tau_y, \tau_u$ be operator acting on $\Laurent_0 \left[ x^{\pm 1} , y^{\pm 1}, u^{\pm 1} \right]$ as follows:
\[
\hat{x} f(x,y,u) = x f(x,y,u) \text{ , } \tau_x f(x,y,u)  = f(qx, y,u) 
\]
the reader can guess definitions of the four remainder operators. Let $x_1,x_2 \in \lbrace x,y,u \rbrace$ then:
\[
 x_1 \tau_{x_2} = q^{\delta_{x_1,x_2}} \tau_{x_2}  x_1,
\]
namely operators $\hat{ . }$ and $\tau$ q-commute if they involve the same variable, commute otherwise. Operators $\hat{ . }$ commute one with each other, so do operators $\tau$. 
\end{defn}

From these operators we define other ones:
\begin{equation}
\begin{array}{ccc}
a_+ := (\hat{u} - \hat{y} \tau_x^{-1}) \tau_y^{-1} & b_+ := \hat{u}^2 & c_+ = \hat{x} \tau_y^{-2}  \tau_u^{-1}
\end{array}
\end{equation}

\begin{equation}
\begin{array}{ccc}
a_- := (\tau_y - \hat{x}^{-1}) \tau_x^{-1} \tau_u & b_- := \hat{u}^2 & c_- = \hat{y}^{-1} \tau_x^{-1}  \tau_u
\end{array}
\end{equation}

and two matrices:
\begin{equation}
\begin{array}{cc}
S_+ := \begin{pmatrix} a_+ & b_+ \\ c_+ & 0  \end{pmatrix} & S_- := \begin{pmatrix} 0 & c_- \\ b_- & a_-  \end{pmatrix}
\end{array} .
\end{equation}

We add more variables, for a fixed index $j \in \BN$, we define operators $a_{j,\pm}, b_{j,\pm}, c_{j,\pm}$ acting on $\Laurent_0\left[ x_j^{\pm 1},y_j^{\pm 1},u_j^{\pm 1} \right] $ as $a_{\pm},b_{\pm},c_{\pm}$ (resp.) do and trivially on any $\Laurent_0\left[ x_i^{\pm 1},y_i^{\pm 1},u_i^{\pm 1} \right] $ if $i \neq j$. We extend as well definitions of $S_{j,+}$ (resp. $S_{j,-}$) as those of $S_+$ (resp. $S_-$) involving $a_j,b_j,c_j$. 

\begin{defn}[Deformed Burau matrix]
Let $\beta := \sigma_{i_1}^{\epsilon_1} \cdots \sigma_{i_k}^{\epsilon_k} \in \Bn$ be a braid written as a product of Artin generators. We define its {\em deformed Burau matrix} as follows:
\[
\rho(\beta) := \prod_{j=1}^k A_k
\]
where $A_j :=   \Id_{i_j -1 } \oplus S_{j,\epsilon_{j}} \oplus \Id_{n-i_j-1}$. Entries of $\rho(\beta)$ are operators acting on $\CP_k := \bigotimes_{j=1}^k \Laurent_0 \left[ x_j^{\pm 1} , y_j^{\pm 1} , u_j^{\pm 1}  \right]$. 
\end{defn}

\begin{defn}[Evaluation of operators]
Let $P$ be a polynomial in operators $a_{\pm} ,b_{\pm} , c_{\pm}$ (with maybe indices) with coefficients in $\Laurent_0$. The evaluation $\CE(P) \in \Laurent := \BZ \left[q^{\pm 1} , s^{\pm 1} \right]$ is defined to be the application of $P$ to the constant function $1 \in \CP_k$ then substituting $u_j$ by $1$ and $x_j,y_j$ by the formal variable $s$ for all $j=1 , \ldots , k$. 
\end{defn}

The following lemma is part of Lemma 1.4 in \cite{HL}:

\begin{lemma}\label{lemma_image_Epsilon_converge}
Let $d,r,s \in \N$, we have:

\[  \CE(b_+^s c_+^r a_+^d)= q^{-rd} s^r (1-sq^{-r})_{q^{-1}}^d \]
\[  \CE(b_-^s c_-^r a_-^d)= s^{-r} (1-s^{-1}q^{r})_{q}^d \]

where $(1-x)_q^d = \prod_{i=0}^{d-1} (1-xq^i)$.

\end{lemma}

We hence have a convergent series in $\Laurentcomplet$ when evaluating $a_{\pm}$ series operators with $\CE$.

\begin{defn}[Evaluation of series operators]
Let $P$ be a in operators $b_{\pm} , c_{\pm}$ and a series in $a_{\pm}$ (with maybe indices) with coefficients in $\Laurent_0$. The evaluation $\CE(P) \in \Laurentcomplet$ is defined to be the application of $P$ to the constant function $1 \in \CP_k$ then substituting $u_j$ by $1$ and $x_j,y_j$ by the formal variable $s$ for all $j=1 , \ldots , k$. 
\end{defn}

\subsection{Quantum Determinant}

\begin{defn}[Right quantum matrices]
A $2 \times 2$-matrix $\begin{pmatrix} a & b \\ c & d\end{pmatrix}$ is said to be \emph{right quantum} if:
\[
\begin{array}{ccc}
ac = q ca & bd = q db & ad = da + q cb - q^{-1} bc
\end{array} .
\]
An $m \times m$ matrix is right quantum if any of its $2 \times 2$-submatrix is. 
\end{defn}

\begin{defn}[Quantum determinant]
Let $A = (a_{ij})$ be a right quantum matrix. Its quantum determinant is defined as follows:
\[
\det_q(A) := \sum_{\pi \in \Sk_m} (-q)^{\inv(\pi)} a_{\pi 1 , 1 } a_{\pi 2 , 2 } \cdots a_{\pi m , m }
\]
where $\inv(\pi)$ is the number of inversions. 
\[
\widetilde{\det_q}(\Id - A) := 1-C \text{ , with } C := \sum_{J \subset \lbrace 1 , \ldots , m \rbrace } (-1)^{|J|-1} \det_q (A_J) ,
\]
where $A_J$ is the $J \times J$ submatrix of $A$ (which is right-quantum). 
\end{defn}

\subsection{Unified invariant from a quantum determinant}~


\begin{lemma}\label{lemma_qdet_a_series}
Let $\beta \in \CB_m$ be a braid which standard closure is a knot. The operator $\frac{1}{\widetilde{\det_q} \left( \Id - q \rho'(\beta) \right)}$ is a series in $a_{\pm}$.
\end{lemma}
\begin{proof}
We first need to symmetric powers of the deformed Burau matrices of Artin generators $A_j$.

Let $(x_i)_{1 \leq i \leq m}$ form a $m$ dimensional quantum algebra, meaning that $x_j x_i =q x_i x_j$ for $i <j$. For a right quantum matrix $A=(a_{ij})$ let $X_i=\sum_{j=1}^m a_{ij}x_j$ and let $G(j_1, \dots, j_m)$ be the coefficient of $x_1^{j_1} \dots x_1^{j_m}$ in $\prod_{i=1}^m X_i^{j_i}$.

Recall that $A_j$ is the deformed Burau matrix of the Artin generator $\sigma_k$. Let $\sum_{i=1}^m j_i =N$,
\begin{align*}
\Sym^N(A_i)\prod_{k=1}^m (x_k)^{j_k} &= \prod_{k=1}^m (A_i x_k)^{j_k}\\
&= x_1^{j_1} \cdots x_{i-1}^{j_{i-1}}(a_+ x_i+ b_+ x_{i+1})^{j_i} \times (c_+ x_i)^{j_{i+1}} x_{i+2}^{j_{i+2}} \cdots x_m^{j_m}\\
&= \sum_{l=0}^{j_i} \binom{j_i}{l}_{q^{-1}} q^{(j_i-l)j_{i+1}} a_+^l b_+^{j_i-l} c_+^{j_{i+1}} x_1^{j_1} \dots   x_i^{j_{i+1}+l}  x_{i+1}^{j_i-l} \dots x_m^{j_m}
\end{align*}

Hence, each Artin generator will induce a sum at the level of the action, the index $l$ of the sum is called \textit{state} and the sum is called \textit{state sum}.

Recall that $\beta=\sigma_{i_1}^{\epsilon_1} \cdots \sigma_{i_k}^{\epsilon_k}$ is such that the induced permutation $\pi(\beta)$ is a derangement. The element $G(0, j_2, \dots, j_m)$ is a $k$-states sum that verifies \[j_{\pi(\beta)(i)}= j_i + \textit{linear combination of states with coeff } 1 \textit{ or } -1.\] 

Since $\pi(\beta)$ is a derangement, $j_i$ is a linear combination of states with coeff $1$ or $-1$.

Hence if $\sum_{i=2}^m j_i =N$, there is always a state $l$ that verifies $l \geq \frac{N}{mk}$.

Using quantum Mac Mahon Master theorem \cite{GLZ}, we know that \[\frac{1}{\widetilde{\det_q} \left( \Id - q \rho'(\beta) \right)}= \sum_{j_2,\dots,j_m} G(0,j_2,\dots, j_m). \]

Hence, $\frac{1}{\widetilde{\det_q} \left( \Id - q \rho'(\beta) \right)}$ is a series in $a_{\pm}$.
\end{proof}

\begin{thm}\label{thm_F_from_qdet}
Let $\beta\in \CB_m$ be a braid which standard closure is a knot. One remarks that $\rho(\beta)$ is right quantum. Let $\rho'(\beta)$ be obtained from $\rho(\beta)$ by removing first row and column. Then:
\[
F_{\infty}(\hat{\beta}) = s^{(w(\beta)-m+1)/2} \CE \left( \frac{1}{\widetilde{\det_q} \left( \Id - q \rho'(\gamma) \right)} \right)
\]
\end{thm}
\begin{proof}
Let $\CE_N$ be the evaluation corresponding to the substitution $s= q^{N-1}$. Then, for any $N \in \BN$:
\[
\CE_N \left( s^{(w(\beta)-m+1)/2} \CE \left( \frac{1}{\widetilde{\det_q} \left( \Id - q \rho'(\gamma)\right)} \right)  \right) = \Jones'_{\hat{\beta}} (N)
\]
where $\Jones'_{\hat{\beta}} (N)$ is the open $N^{\text{th}}$-colored Jones polynomial of $\hat{\beta}$ (It is \cite[Theorem~1]{H-L}). We conclude using the unique element interpolating colored Jones polynomials property, Prop. \ref{prop_unicity_Foo}. 
\end{proof}

\begin{rem}
The entire proof of Theorem~1 from \cite{H-L} adapts to $F_{\infty}$ almost step by step and word by word, although here we have preferred to use a stronger and concise argument. The proof from \cite{H-L} explains in details the relations between the operators $a_{\pm} , b_{\pm} , c_{\pm}$ and the braiding of Verma modules, so that it is important for the understanding of the formula. 
\end{rem}

\begin{cor}[ADO polynomials from quantum determinant]
Let $\beta\in \CB_m$ be a braid which standard closure is a knot, and $\CE_{\zeta_r}$ be the composition of $\CE$ with the substitution $q= \zeta_r$. Then:

\[
\ADO_r(\hat{\beta})  = s^{(w(\beta)-m+1)/2}   F_r \left( \CE \left( \det \left( \Id - \rho'(\gamma) \right) \right) \right) \CE_{\zeta_{r}} \left( \frac{1}{ \widetilde{\det_q} \left( \Id - q \rho'(\gamma) \right)} \right) 
\] where $F_r: \Z[s^{\pm1}] \to \Z[s^{\pm1}]$, $s \mapsto s^r$.

\end{cor}
\begin{proof}
Straightforward consequence of Theorems \ref{thm_F_from_qdet} and \ref{thm_factorisation_unified_ADO}, and the fact that $\CE \left( \det \left( \Id - \rho'(\gamma) \right) \right)$ is the Alexander polynomial (\cite[Theorem~1(b)]{H-L}). 
\end{proof}

\bibliographystyle{abbrv}
\bibliography{unified_inv_braid_rep.bib}
\end{document}